\documentclass[10pt]{amsart}
\usepackage{amsthm, amsmath,amssymb}
\usepackage{amsmath,amssymb}
\usepackage[pdftex, pdfborder={0 0 0}, pdftitle={Topological triangulated categories},pdfauthor={Stefan Schwede},pdfstartview={FitH}]{hyperref}
\usepackage{enumerate}
\usepackage[curve,matrix,arrow]{xy}
\DeclareMathOperator{\Id}{Id}
\DeclareMathOperator{\colim}{colim}
\DeclareMathOperator{\map}{map}
\DeclareMathOperator{\Hom}{Hom}
\DeclareMathOperator{\Ext}{Ext}
\DeclareMathOperator{\Cone}{Cone}
\DeclareMathOperator{\ord}{-ord}
\DeclareMathOperator{\Mod}{-mod}

\newcommand{\mF}{{\mathbb F}}
\newcommand{\mQ}{{\mathbb Q}}
\newcommand{\mR}{{\mathbb R}}
\newcommand{\mS}{{\mathbb S}}
\newcommand{\mZ}{{\mathbb Z}}

\newcommand{\cA}{{\mathcal A}}
\newcommand{\cB}{{\mathcal B}}
\newcommand{\cC}{{\mathcal C}}
\newcommand{\cD}{{\mathcal D}}
\newcommand{\cF}{{\mathcal F}}
\newcommand{\cO}{{\mathcal O}}
\newcommand{\cS}{{\mathcal S}}
\newcommand{\cT}{{\mathcal T}}
\newcommand{\cZ}{{\mathcal Z}}

\newcommand{\bK}{{\mathbf K}}
\newcommand{\bD}{{\mathbf D}}
\newcommand{\bH}{{\mathbf H}}
\newcommand{\bHo}{\mathbf{Ho}}
\newcommand{\bSH}{\mathbf {SH}}
\newcommand{\iso}{\cong}
\newcommand{\xra}{\xrightarrow}
\newcommand{\xla}{\xleftarrow}
\newcommand{\sm}{\wedge}
\newcommand{\tensor}{\otimes}
\newcommand{\un}{\underline}
\newcommand{\td}[1]{\langle #1\rangle}
\renewcommand{\to}{\longrightarrow}

\numberwithin{equation}{section}
\newtheorem{theorem}[equation]{Theorem}
\newtheorem{lemma}[equation]{Lemma}
\newtheorem{prop}[equation]{Proposition}
\newtheorem{cor}[equation]{Corollary}
\theoremstyle{definition}
\newtheorem{defn}[equation]{Definition}
\newtheorem{construction}[equation]{Construction}
\newtheorem{rk}[equation]{Remark}
\newtheorem{eg}[equation]{Example}

\begin{document}

\title{Topological triangulated categories}

\date{\today; 2000 AMS Math.\,Subj.\,Class.: 18E30, 55P42}
\author{Stefan Schwede}
\address{Mathematisches Institut, Universit\"at Bonn, Germany}
\email{schwede@math.uni-bonn.de}
\maketitle

Many triangulated categories arise from chain complexes in an 
additive or abelian category by passing to chain homotopy classes 
or inverting quasi-isomorphisms. Such examples are called `algebraic' 
because they have underlying additive categories.
Stable homotopy theory produces examples of triangulated categories
by quite different means, and in this context the underlying categories
are usually very `non-additive' before passing to homotopy classes of 
morphisms. We call such triangulated categories {\em topological},
and formalize this in Definition~\ref{def-topological}
via homotopy categories of stable cofibration categories.

The purpose of this paper is to explain some systematic differences 
between algebraic and topological triangulated categories. 
There are certain properties 
-- defined entirely in terms of the triangulated structure --
which hold in all algebraic examples, but which can fail in general. 
The precise statements use the {\em $n$-order}
of a triangulated category, for $n$ a natural number
(see Definition~\ref{def-order}).
The $n$-order is a non-negative integer (or infinity), 
and it measures, roughly speaking, `how strongly'
the relation $n\cdot Y/n=0$ holds for the objects $Y$ 
in a given triangulated category
(where $Y/n$ denotes a cone of multiplication by~$n$ on~$Y$).
Our main results are:
\begin{itemize}
\item The $n$-order of every algebraic triangulated category is infinite
(Theorem~\ref{thm-general algebraic}).
\item For every prime $p$, the $p$-order of every topological 
triangulated category is at least $p-1$ (Theorem~\ref{thm-general topological}).
\item For every prime $p$, the $p$-order of the $p$-local stable homotopy
category is exactly~$p-1$ (Corollary~\ref{cor-exact order of Moore}).
\end{itemize}

In particular, the $p$-local stable homotopy category is not algebraic
for any prime $p$; this is folklore for $p=2$, but seems to be
a new result for odd primes~$p$.

One of the technical innovations in this paper is the use of
$\Delta$-sets for studying cofibration categories.
We develop certain foundations about
enrichments of cofibration categories by $\Delta$-sets that seem
to be new  -- and hopefully of independent interest.
In particular we generalize the theory
of `framings' (or `cosimplicial resolutions') from model categories
to cofibration categories.
Theorem~\ref{theorem-framing} shows that the category of {\em frames} 
(certain homotopically constant co-$\Delta$-objects)
in any cofibration category is again a cofibration category, 
and that the homotopy category does not change when passing to frames.
Theorem~\ref{thm-frames are Delta-enriched}
shows that the category of frames in a saturated cofibration category 
is always a $\Delta$-cofibration category
(compare Definition~\ref{def-Delta cofibration category}),
the analog of simplicial model category in this context.
In summary, one could say that $\Delta$-sets are to cofibration categories 
what simplicial sets are to model categories.

Here is a summary of the contents of this paper by sections. 
In Section~\ref{sec-topological} we introduce
topological triangulated categories as the homotopy categories of
stable cofibration categories and discuss some basic properties.
Section~\ref{sec-order} is devoted to the concept of $n$-order of
triangulated categories; we show that the existence of a `mod-$n$ reduction'
for a triangulated category implies infinite $n$-order and
apply this to the derived categories of certain structured ring spectra.
Section~\ref{sec-algebraic} discusses algebraic triangulated categories;
we show that every algebraic triangulated category is 
in particular topological and has infinite $n$-order.
Section~\ref{sec-Moore} contains the proof that for every prime~$p$
the $p$-order of the stable homotopy category is at most $p-1$.
In Section~\ref{sec-Delta review} we review basic properties of
$\Delta$-sets. Section~\ref{sec-framings}
develops the theory of {\em framings} in cofibration categories; 
this is needed later to make sense 
of an action of the category of finite $\Delta$-sets on a cofibration category. 
Section~\ref{sec-pointed Delta} introduces and studies coherent
actions of mod-$n$ Moore spaces (i.e., $\Delta$-sets) on
objects in pointed cofibration categories.
In Section~\ref{sec-order in topological} we prove that the $p$-order 
of every topological triangulated category is at least~$p-1$.
Appendix~\ref{app-triangulation}
recalls the basic facts about the homotopy category of a cofibration category
and gives a self-contained proof that the homotopy category
of a stable cofibration category is triangulated. 

Some results of this paper were announced in
the survey article~\cite{sch-leeds}; that paper contains more motivation 
and background, and can serve as a complement for the present article.
In~\cite{sch-leeds}, we based the definition of topological
triangulated categories on Quillen model categories.
Here, however, we use the more general concept of cofibration categories,
so the results here are somewhat more general 
than announced in~\cite{sch-leeds}.

\section{Topological triangulated categories}
\label{sec-topological}

The way we implement the notion of `topological enhancement'
of triangulated categories is to use {\em cofibration categories}.
This notion was first introduced and studied (in the dual formulation) 
by Brown~\cite{brown} under the name `categories of fibrant objects'.
Closely related sets of axioms have been explored by 
various authors, compare Remark~\ref{rk-other cofibration categories}.

\begin{defn}\label{def-cofibration category} 
  A {\em cofibration category} is a category $\cC$ equipped with
  two classes of morphisms, called {\em cofibrations}
  respectively {\em weak equivalences}, that satisfy the following axioms
  (C1)--(C4).
  \begin{itemize}
  \item[(C1)] All isomorphisms are cofibrations and weak equivalences.
    Cofibrations are stable under composition. 
    The category $\cC$ has an initial object and every morphism
    from an initial object is a cofibration.
  \item[(C2)] Given two composable morphisms $f$ and $g$ in $\cC$, 
    then if two of the three morphisms $f,g$ and~$gf$ are weak equivalences, 
    so is the third.
  \item[(C3)] Given a cofibration $i:A\to B$ and any morphism $f:A\to C$,
    there exists a pushout square
    \begin{equation}\begin{aligned}\label{eq-C3pushout}
        \xymatrix{ A \ar[r]^f \ar[d]_i & C \ar[d]^j \\ B\ar[r] & P}
      \end{aligned}\end{equation}
    in $\cC$ and the morphism $j$ is a cofibration. 
    If additionally $i$ is a weak equivalence, then so is $j$.
  \item[(C4)] Every morphism in $\cC$ can be factored as
    the composite of a cofibration followed by a weak equivalence.
  \end{itemize}
\end{defn}

An {\em acyclic cofibration} is a morphism that is both
a cofibration and a weak equivalence.
We note that in a cofibration category a coproduct $B\vee C$ 
of any two objects in $\cC$ exists by (C3) 
with $A$ an initial object, 
and the canonical morphisms from $B$ and $C$ to $B\vee C$ are cofibrations.

A property that we will frequently use is the {\em gluing lemma}.
This starts with a commutative diagram
$$\xymatrix{ A \ar[d]_\sim & B\ar[l]_-i\ar[r]\ar[d]_\sim & C\ar[d]^-\sim\\
 A' & B'\ar[l]^-{i'}\ar[r] & C'}$$
in a cofibration category $\cC$ such that $i$ and $i'$ are cofibrations
and all three vertical morphisms are weak equivalences.
The gluing lemma says that then the induced morphism 
on pushouts $A\cup_BC\to A'\cup_{B'}C'$ is a weak equivalence.  
A proof of the gluing lemma can be found in Lemma~1.4.1~(1) 
of~\cite{radulescu-ABC}.

The {\em homotopy category} of a cofibration category
is a localization at the class of weak equivalences, i.e., 
a functor $\gamma:\cC\to\bHo(\cC)$ that takes all weak
equivalences to isomorphisms and is initial among such functors.
The homotopy category always exists if one is willing to pass to
a larger universe. To get a locally small homotopy category 
(i.e., have `small hom-sets'),
additional assumptions are necessary;
one possibility is to assume that $\cC$ has `enough fibrant objects',
compare Remark~\ref{rk-enough fibrant objects}.
We recall some basic facts about the homotopy category of
a cofibration category in Theorem~\ref{thm-HoC a la Brown}

\begin{rk}\label{rk-other cofibration categories}
The above notion of cofibration category is due to K.\,S.\,Brown~\cite{brown}.
More precisely, Brown introduced `categories of fibrant objects', 
and the axioms (C1)--(C4) are equivalent to
the dual of the axioms (A)--(E) of Part~I.1 in~\cite{brown}. 
The concept of a cofibration category is a substantial
generalization of Quillen's notion of a `closed model category'~\cite{Q}: 
from a Quillen model category one obtains a cofibration category by
restricting to the full subcategory of cofibrant objects and 
forgetting the class of fibrations.

Cofibration categories are closely related to
`categories with cofibrations and weak equivalences'
in the sense of Waldhausen~\cite{waldhausen}.
In fact, a category with cofibrations and weak equivalences
that also satisfies the {\em saturation axiom}~\cite[1.2]{waldhausen}
and the {\em cylinder axiom}~\cite[1.6]{waldhausen} is in particular
a cofibration category as in Definition~\ref{def-cofibration category}.
Further relevant references on closely related axiomatic
frameworks are  Baues' monograph~\cite{baues-algebraic homotopy} and
Cisinski's article~\cite{cisinski-categories derivables}.
Radulescu-Banu's extensive paper~\cite{radulescu-ABC} 
is the most comprehensive source for basic results on cofibration
categories and, among other things, contains a survey of the different kinds of 
cofibration categories and their relationships.
\end{rk}

A cofibration category is {\em pointed} 
if every initial object is also terminal, hence a  zero object, 
which we denote by $\ast$.
In a pointed cofibration category, the factorization axiom~(C4)
provides a {\em cone} for every object $A$,
i.e., a cofibration $i_A:A\to CA$ whose target is weakly equivalent
to the zero object.
The {\em suspension} $\Sigma A$ of $A$ is the quotient
of the `cone inclusion', i.e., a pushout:
$$\xymatrix{ A \ar[r]^-{i_A} \ar[d]& CA \ar[d] \\ \ast \ar[r] & \Sigma A }$$
We recall in Proposition~\ref{prop-suspension} that there is a preferred way 
to make the suspension construction into functor
$\Sigma:\bHo(\cC)\to\bHo(\cC)$ on the level homotopy categories.
In other words: on the level of cofibrations categories
the cone, and hence the suspension, constitute a choice, 
but any set of choices becomes functorial upon passage to the homotopy category.
Moreover, different choices of cones lead to canonically isomorphic
suspension functors, compare Remark~\ref{rk-uniqueness of triangulation}.

Every cofibration $j:A\to B$ in a pointed cofibration category $\cC$
gives rise to a preferred and natural
{\em connecting morphism} $\delta(j):B/A\to \Sigma A$
in $\bHo(\cC)$, see~\eqref{eq:connecting_morphism}.
The {\em elementary distinguished triangle}
 associated to the cofibration $j$ is the triangle
$$ A \ \xra{\ \gamma(j)\ }\ B \ \xra{\ \gamma(q)\ }\ B/A 
\ \xra{\ \delta(j) \ }\ \Sigma A \ ,$$
where $q:B\to B/A$ is a quotient morphism.
A {\em distinguished triangle} is any triangle in the homotopy category
that is isomorphic to the elementary distinguished triangle of a cofibration.  

A pointed cofibration category is {\em stable} if the suspension functor
$\Sigma:\bHo(\cC)\to\bHo(\cC)$ is an autoequivalence.
We recall in Theorem~\ref{thm-HoC is triangulated}
that the suspension functor and the class of distinguished triangles make
the homotopy category $\bHo(\cC)$ of a stable cofibration category
into a triangulated category.
We call the class of triangulated categories arising in this way
the `topological triangulated categories'.
The adjective `topological' does {\em not} mean that the
category or its hom-sets are topologized; rather,
`topological' is supposed to indicate that these examples
are constructed by topological methods, or that they have models
in the spirit of abstract homotopy theory.

\begin{defn}\label{def-topological}
A triangulated category is {\em topological}
if it is equivalent, as a triangulated category,  
to the homotopy category of a stable cofibration category.
\end{defn}

Now we prove a basic closure property of topological triangulated categories.

\begin{prop}\label{prop-topological closed} 
Every full triangulated subcategory of a 
topological triangulated category is topological.
\end{prop}
\begin{proof}
Let $\cC$ be a stable cofibration category and $\cT$ a
full triangulated subcategory of $\bHo(\cC)$. We let $\bar\cC$
denote the full subcategory of $\cC$ consisting of those objects
that are isomorphic in $\bHo(\cC)$ to an object in $\cT$.
We claim that $\bar\cC$ becomes a stable cofibration category
when we restrict the classes of cofibrations and weak equivalences
from $\cC$ to $\bar\cC$. 

Indeed, axioms (C1) and (C2) are directly inherited from $\cC$. 
Concerning axiom (C3) we observe that 
a pushout square~\eqref{eq-C3pushout} in $\cC$
in which $i:A\to B$ is a cofibration
gives rise to two elementary distinguished triangles 
$$ A \xra{\ \gamma(i)\ } B \xra{\qquad} B/A \xra{\ \delta(i)\ } \Sigma A 
\text{\qquad and\qquad}
C \xra{\ \gamma(j)\ } P \xra{\qquad} P/C \xra{\ \delta(j)\ } \Sigma C$$
in $\bHo(\cC)$. So if $A$ and $B$ belong to $\bar\cC$,
then so does the quotient $B/A$. If moreover $C$ belongs to $\bar\cC$,
then so does the pushout $P$ (since $B/A$ and $P/C$ are isomorphic).
Hence the square~\eqref{eq-C3pushout} is in fact a pushout in
the subcategory $\bar\cC$, so (C3) holds in $\bar\cC$.
Since $\bar\cC$ is closed under weak equivalences, a factorization
as in axiom (C4) in the ambient category $\cC$
is also a factorization  in $\bar\cC$,
so $\bar\cC$ inherits property~(C4).

As we just saw, the inclusion $\bar\cC\to\cC$ preserves 
the particular pushouts required by (C3), so it is an exact functor.
Moreover, the functor 
$\bHo(\bar\cC)\to\bHo(\cC)$ induced by the inclusion is fully faithful
(use parts~(i) and~(ii) of Theorem~\ref{thm-HoC a la Brown}).
Since the inclusion is exact, this induced functor is also exact
(by Proposition~\ref{prop-left derived is exact})
and hence an embedding of triangulated categories 
whose image is the original triangulated subcategory~$\cT$ of $\bHo(\cC)$.
\end{proof}

Examples of topological triangulated categories abound.
An important example is the 
{\em stable homotopy category} of algebraic topology which, to my knowledge,
was first published in the form it is used today
by Kan~\cite{kan-semisimplicial}. 
When Kan's paper appeared~1962, 
neither model categories nor cofibration categories had been formalized, 
but it is straightforward to deduce from the results therein
that Kan's `semisimplicial spectra' form a stable cofibration category.
Later Brown showed in~\cite[II, Theorem~5]{brown} that the semisimplicial
spectra even support a model category structure.
By now there is an abundance of models for the stable homotopy category, 
see for example~\cite{BF, EKMM, HSS, MMSS}.
The {\em Spanier-Whitehead category}~\cite{spanier-whitehead}, 
which predates the stable homotopy category, can be obtained from finite
based CW-complexes by formally inverting the suspension functor;
it is equivalent to the full subcategory of compact objects 
in the stable homotopy category,
so it is a topological triangulated category in our sense.

Further examples of topological triangulated categories 
are the homotopy categories of stable model categories,
including the `derived' (i.e., homotopy) categories of structured ring spectra
or spectral categories, 
equivariant and motivic stable homotopy categories,
sheaves of spectra on a Grothendieck site or (Bousfield-)localizations
of the above; a more detailed list of
specific references can be found in~\cite[Sec.\,2.3]{ss-modules}.
In Section~\ref{sec-algebraic} we will discuss the class of {\em algebraic}
triangulated categories and observe that every algebraic
triangulated category is in particular topological.

Triangulated categories that are not topological 
do not come up as frequently; in fact the author's point of view is precisely 
that topological triangulated categories are the ones that come up 
`in nature', and other examples have to be `manufactured by hand'.
The simplest example of a triangulated category that is not
topological (and hence also not algebraic) is the following, 
due to Muro.
The category $\mathcal F(\mZ/4)$ of finitely generated free modules over
the ring $\mZ/4$ has a unique triangulation with
the identity shift functor and such that the triangle
$$ \mZ/4\ \xra{\ 2\ }\ \mZ/4\ \xra{\ 2\ }\ \mZ/4\ \xra{\ 2\ }\ \mZ/4$$
is distinguished.
For details we refer to~\cite{nomodel}, where Muro, Strickland 
and the author discuss an entire family
of `exotic' (i.e., non-topological) triangulated categories
that includes $\mathcal F(\mZ/4)$ as the simplest case. 
I do not know any non-topological triangulated category 
in which~2 is invertible.

\section{Order}\label{sec-order}

As we explain in the next section,
every algebraic triangulated category is topological.
The aim of this paper is to measure to what extent the converse fails.
A prominent example of a non-algebraic, 
but topological triangulated category is the
stable homotopy category; the simple argument to see that, based 
on the mod-2 Moore spectrum, breaks down as soon as~2 is inverted,
for example for the localization
of the stable homotopy category at an odd prime.
So for a deeper understanding of this phenomenon we need different tools.
We now discuss our main concept, the notion of {\em $n$-order}
of an object $X$ in a triangulated category, 
which first appeared in the expository paper~\cite{sch-leeds}.
As a slogan, the $n$-order measures `how strongly'
the relation $n\cdot X=0$ holds. 

For an object $K$ of a triangulated category $\cT$ and a 
natural number $n$ we write $n\cdot K$ for the $n$-fold
multiple of the identity morphism in the abelian group of endomorphisms 
in $\cT$.
We let~$K/n$ denote any cone of $n\cdot K$,
i.e., an object which is part of a distinguished triangle
$$ K \xra{\ n\cdot\ } K \xra{\ \pi\ } K/n \xra{\quad} \Sigma K \ . $$
In the following definition, an {\em extension} of a morphism $f:K\to X$ 
is any morphism $\bar f:K/n\to X$ satisfying $\bar f\pi=f$.

\begin{defn}\label{def-order}
 We consider a triangulated category $\cT$, an object $X$ of $\cT$ 
 and a natural number $n\geq 1$.
 We define the {\em $n$-order} of $X$ inductively.
\begin{itemize}
\item Every object has $n$-order greater or equal to~0.
\item For $k\geq 1$, the object $X$ has $n$-order greater or equal to $k$
if and only if for every object $K$ of $\cT$ and every 
morphism $f:K\to X$ there exists
an extension $\bar f:K/n\to X$ such that some (hence any) cone 
of $\bar f$ has $n$-order greater or equal to $k-1$.
\end{itemize}
The {\em $n$-order} of the triangulated category $\cT$
is the $n$-order of some (hence any) zero object.
\end{defn}

Some comments about the definition are in order.
Since a cone is only well-defined up to
non-canonical isomorphism, we should justify that 
in the inductive definition it makes no difference
whether we ask for the condition for some or for any
cone of the extension~$\bar f$. 
This follows from the observation (which is most easily proved by induction
on $k$) that the property `having $n$-order greater or equal to $k$'
is invariant under isomorphism.

We write $n\ord(X)$ for the $n$-order of $X$, i.e., the largest $k$
(possibly infinite) such that $X$ has $n$-order greater or equal to $k$.
If we need to specify the ambient triangulated category, we
use the notation $n\ord^{\cT}(X)$.
We denote by $n\ord(\cT)$ the $n$-order of the triangulated category~$\cT$.

We record some direct consequences of the definition.

\begin{itemize}
\item The $n$-order for objects is invariant under isomorphism 
and shift.
\item An object $X$ has positive $n$-order if and only if every morphism
$f:K\to X$ has an extension to $K/n$, which is equivalent to
$n\cdot f=0$. So $n\ord(X)\geq 1$ is equivalent to the 
condition $n\cdot X=0$.
\item  
The $n$-order of a triangulated category is one
larger than the minimum of the $n$-orders of all objects of the form~$K/n$. 
\item Let $\cS\subseteq\cT$ be a full triangulated subcategory and
$X$ an object of $\cS$. 
Then we have $n\ord^{\cS}(X)\geq n\ord^{\cT}(X)$.
In the special case of a zero object we get $n\ord(\cS)\geq n\ord(\cT)$.
\item Suppose that $\cT$ is a $\mZ[1/n]$-linear triangulated category,
i.e., multiplication by $n$ is an isomorphism for every object of $\cT$.
Then $K/n$ is trivial for every object~$K$ and thus
$\cT$ has infinite $n$-order. If on the other hand $X$ is non-trivial, then
$n\ord(X)=0$. 
\item 
If every object of $\cT$ has positive $n$-order, then $n\cdot X=0$
for all objects $X$ and so $\cT$ is a $\mZ/n$-linear triangulated category.
Suppose conversely that $\cT$ is a $\mZ/n$-linear triangulated category.
Then induction on $k$ shows that $n\ord(X)\geq k$ for all objects
$X$, and thus every object has infinite $n$-order.
\end{itemize}

As an example of the typical kind of reasoning, 
we give the details for the fourth item.
We let $X$ be an object of the full triangulated subcategory $\cS$ of $\cT$
and show by induction on $k$ that if 
$n\ord^{\cT}(X)\geq k$, then also $n\ord^{\cS}(X)\geq k$.
There is nothing to show for $k=0$, so we may assume $k\geq 1$.
We consider any object $K$ of $\cS$ and morphism $f:K\to X$.
We choose a distinguished triangle for $n\cdot K$ in $\cS$, 
which then serves the same purpose in the larger category $\cT$.
For any extension $\bar f:K/n\to X$ (in $\cS$ or, what is the same, in $\cT$)
any mapping cone $C(\bar f)$ in $\cS$ is also a mapping cone in $\cT$, and
so $n\ord^\cT(C(\bar f))\geq k-1$ since $n\ord^{\cT}(X)\geq k$.
By the inductive hypothesis, $n\ord^\cS(C(\bar f))\geq k-1$,
which shows that the $n$-order of $X$ in the subcategory $\cS$ is at least~$k$.

The last two items above show that the concept of $n$-order
is not interesting for triangulated categories which are
linear over some field. Indeed, if $k$ is a field
and $\cT$ a $k$-linear triangulated category, then the number
$n$ is either invertible or zero in $k$, and hence 
$\cT$ is either $\mZ[1/n]$-linear or $\mZ/n$-linear.
In either case, the $n$-order of $\cT$ is infinite.

\begin{rk}
We can define the co-order of $X$ in a way dual to the order. 
An object $X$ has {\em $n$-co-order} at least $k$ if for every
morphism $g:X\to \Sigma K$ there is a lifting $\bar g:X\to K/n$
such that some (hence any) cone of $\bar g$
has $n$-co-order at least $k-1$. Then $X$ has positive $n$-co-order if and
only $n\cdot X=0$, so the $n$-co-order is another way to measure 
`to what extent' an object $X$ is annihilated by $n$.

The notion of co-order is, however, a special case of order. In fact,
for any triangulated category $\cT$, the opposite category $\cT^\text{op}$
is triangulated by exchanging the role of the shift and its inverse, and
by keeping the exact triangles, but now viewed in the opposite category. 
Then the $n$-co-order of $X$ in $\cT$
is nothing but the $n$-order of $X$ in the opposite category.

If the triangulated category $\cT$ admits a `duality', i.e.,
an exact equivalence $D:\cT\to\cT^\text{op}$ to its opposite category,
then we can compare order and co-order as 
$$  n\ord^{\cT}(X) \ = \ n\ord^{\cT^\text{op}} (DX) \ = \
n\text{-co-ord}^{\cT}(DX) \ .$$
If moreover $X$ is self-dual, i.e., isomorphic to some shift of $DX$, then
the $n$-order and $n$-co-order of~$X$ coincide.
\end{rk}

Now we prove a sufficient criterion, the existence of a `mod-$n$ reduction',
so that a triangulated category has infinite $n$-order. 
This criterion will be used twice below,
namely in Proposition~\ref{prop-reduction of ring spectra} to show 
that the homotopy categories of certain structured ring spectra
have infinite $n$-order, and in the next section  to show 
that every algebraic triangulated category has infinite $n$-order.

\begin{defn}\label{defn-mod n reduction} 
Let $\cT$ be a triangulated category and $n$ a natural number.
A {\em mod-$n$ reduction} of $\cT$ 
consists of a triangulated category $\cT/n$
and an exact functor $\rho_*:\cT\to\cT/n$
that has a right adjoint  $\rho^*$ and
such that for every object $X$ of $\cT$ there exists a distinguished triangle
\begin{equation}\label{eq-reduction triangle}
 X \xra{\ n\cdot } X \xra{\ \eta_X\ } \rho^*(\rho_* X) \xra{\quad} \Sigma X 
\end{equation}
where $\eta_X$ is the unit of the adjunction.
\end{defn}

Note that in the definition of a mod-$n$ reduction,
we do {\em not} require that the category $\cT/n$ is $\mZ/n$-linear.
However, we always have $n\cdot (\rho^*Z)=0$ for all objects
$Z$ of $\cT/n$. Indeed, for $X=\rho^*Z$ the adjunction unit 
$\eta_{\rho^*Z}:\rho^*Z\to\rho^*(\rho_*(\rho^*Z))$ has a left inverse
$\rho^*(\epsilon_Z)$, where $\epsilon_Z:\rho_*(\rho^*Z)\to Z$ is the
adjunction counit. The distinguished triangle~\eqref{eq-reduction triangle} 
thus splits for $X=\rho^*Z$, and so we have $n\cdot(\rho^*Z)=0$ in $\cT$.

\begin{prop}\label{prop-reduction implies infinite order} 
Let $\cT$ be a triangulated category and $n\geq 1$.
If $\cT$ has a mod-$n$ reduction, then for every object $X$ of $\cT$
the object $X/n$ has infinite $n$-order.
Thus the triangulated category~$\cT$ has infinite $n$-order.
\end{prop}
\begin{proof}
We let $(\cT/n,\rho_*)$ be a mod-$n$ reduction and $\rho^*$
a right adjoint of $\rho_*$.
Since $X/n$ is isomorphic to $\rho^*(\rho_*X)$, it is enough to show that
for every object $Z$ of $\cT/n$ and all~$k\geq 0$, the $\cT$-object $\rho^*Z$ 
has $n$-order greater or equal to~$k$.

We proceed by induction on~$k$; for $k=0$ there is nothing to prove.
Suppose we have already shown that every $\rho^*Z$ has $n$-order 
greater or equal to~$k-1$ for some positive~$k$. 
Given a morphism $f:K\to \rho^*Z$ in $\cT$ we form
its adjoint $\hat f:\rho_*K\to Z$ in $\cT/n$; if we apply 
$\rho^*$ we obtain an extension 
$\rho^*(\hat f):\rho^*(\rho_*K)\to\rho^*Z$ of $f$.
We choose a cone of $\hat f$, i.e., a distinguished triangle
$$ \rho_*K \ \xra{\ \hat f\ } \ Z \ \to \ C(\hat f)\ \to\ \Sigma(\rho_*K)$$
in $\cT/n$. Since $\rho^*$ is exact, $\rho^*C(\hat f)$ is a cone
of the extension $\rho^*(\hat f)$ in $\cT$. By induction, $\rho^*C(\hat f)$ 
has $n$-order greater or equal to $k-1$, which proves that $\rho^*Z$
has $n$-order greater or equal to~$k$.
\end{proof}

Now we give examples of topological triangulated categories 
that have mod-$n$ reductions, and thus infinite $n$-order.
The examples which follow are derived categories (or homotopy categories)
of structured ring spectra; for some of these examples I do not know 
whether they are algebraic or not.

For definiteness, we work with symmetric ring spectra~\cite{HSS},
but the following arguments would work just as well for structured ring spectra
in any one of the modern model categories of spectra 
with compatible smash product. A symmetric ring spectrum $R$ 
has a stable model category of left $R$-module spectra~\cite[Cor.\,5.4.2]{HSS}.
We denote by $\bD(R)$ the homotopy category of $R$-module
spectra and refer to it as the {\em derived category} of $R$; 
this is a topological triangulated category. 
For example, for the Eilenberg-Mac\,Lane ring
spectrum~$HA$ of an ordinary ring~$A$,
the derived category $\bD(HA)$ is triangulated equivalent 
to the unbounded derived category of complexes of
$A$-modules (see~\cite[App.\,B.1]{ss-modules}). 
If $R=\mS$ is the sphere spectrum, 
then $\bD(\mS)$ is the homotopy category of symmetric spectra,
hence equivalent to the stable homotopy category.

I owe the following proposition to Tyler Lawson.

\begin{prop}\label{prop-reduction of ring spectra} 
Let $R$ be a commutative symmetric ring spectrum and $n\geq 1$.
Suppose that there exists an $R$-algebra spectrum $B$ that is
a cone of $n\cdot R$ in the following sense:
there is a distinguished triangle 
\begin{equation}\label{eq-algebra triangle}
 R \xra{\ n\cdot } R \xra{\ \eta\ } B \xra{\quad} \Sigma R
\end{equation}
in the derived category $\bD(R)$ of $R$-module spectra,
where $\eta:R\to B$ is the unit morphism of the 
$R$-algebra structure. 
Then for every $R$-algebra spectrum~$A$ the derived category
$\bD(A)$ has a mod-$n$ reduction and thus infinite $n$-order.
\end{prop}
\begin{proof}
We can replace $B$ by a stably equivalent $R$-algebra that is cofibrant 
in the stable model structure of $R$-algebras of~\cite[Cor.\,5.4.3]{HSS}; 
this way we can arrange that $B$ is cofibrant as an $R$-module
(again by~\cite[Cor.\,5.4.3]{HSS}).

The smash product $A\sm_R B$ over $R$ is another $R$-algebra
spectrum equipped with a homomorphism $f=A\sm\eta:A\iso A\sm_RR\to A\sm_R B$ 
of $R$-algebras. Then $f$ gives rise to a Quillen adjoint functor pair
between the associated stable model categories of $A$-modules
and $(A\sm_RB)$-modules:
the right adjoint $f^*$ is `restriction of scalars' and
the left adjoint $f_*=B\sm_R-$ is `extension of scalars'.
This Quillen functor pair descends to a pair  of
adjoint total derived functors 
on the level of triangulated homotopy categories
$$\xymatrix@C=18mm{ \bD(A) \quad \ar@<.4ex>[r]^-{\rho_*=L(f_*)} & 
\ar@<.4ex>[l]^-{\rho^*=R(f^*)} \quad \bD(A\sm_R B)\ . }$$
We claim that this data is a mod-$n$ reduction for $\bD(A)$.
The only missing property is that for every $A$-module $M$
the object $\rho^*(\rho_*M)$ models $M/n$.
For this we can assume without loss of generality
that the $A$-module $M$ is cofibrant.
Since $R$ is central in $A$, we can view the left $A$-module $M$
as an $A$-$R$-bimodule; the functor
$$ M\sm_R - \ : \ R\Mod \ \to \ A\Mod$$
is a left Quillen functor, and so descends to an exact left derived 
functor of triangulated categories
$$ M\sm^L_R - \ \ : \ \bD(R) \ \to \ \bD(A)\ . $$
Hence we obtain a distinguished triangle 
\begin{equation}\label{eq-module triangle}
 M \xra{\ n\cdot\ } M \xra{\, M\sm^L_R \eta\, } 
M\sm_R^L B \xra{\quad} \Sigma M
\end{equation}
in $\bD(A)$ by smashing the triangle~\eqref{eq-algebra triangle} 
over $R$ with $M$ and using the unit isomorphism between $M\sm^L_R R$ and~$M$.

Since $B$ is cofibrant as an $R$-module, the derived smash product
$M\sm_R^L B$ is represented by the pointset level smash product
$M\sm_R B$ which is isomorphic, as a left $A$-module, 
to $f^*(f_*M)=(A\sm_RB)\sm_AM$.
So $\rho^*(\rho_*M)$ is isomorphic
in $\bD(A)$ to $M\sm_R^L B$ in such a way that 
the adjunction unit corresponds to the morphism
$M\sm^L_R\eta:M\to M\sm_R^L B$. 
So the triangle~\eqref{eq-module triangle} completes the proof that
$(\bD(A\sm_R B),\rho_*)$ is a mod-$n$ reduction of $\bD(A)$.
\end{proof}

\begin{eg}
The hypothesis in Proposition~\ref{prop-reduction of ring spectra}
on the commutative symmetric ring spectrum~$R$ can be paraphrased
by saying that the $R$-module spectrum $R/n$ (or rather, some $R$-module
spectrum of this homotopy type) can be given the structure of
an $R$-algebra spectrum.
A theorem of Angeltveit~\cite[Cor.\,3.2]{angeltveit} gives a sufficient
condition for this in terms of the graded ring $\pi_*R$ of homotopy groups
of $R$: if $\pi_*R$ is $n$-torsion free and concentrated in even dimensions,
then $R/n$ admits an $A_\infty$ structure 
compatible with the $R$-module structure; equivalently, 
there is an $R$-algebra spectrum 
whose underlying $R$-module has the homotopy type of $R/n$. 

Some prominent examples of commutative ring spectra that satisfy
Angeltveit's criterion are the complex cobordism spectrum $MU$,
the complex topological $K$-theory spectrum $KU$ and the Lubin-Tate 
spectra $E(k,\Gamma)$ for a formal group law $\Gamma$ of finite
height over a perfect field~$k$. 
So Proposition~\ref{prop-reduction of ring spectra} implies that
for every algebra spectrum $A$ over any of these commutative ring spectra, 
the derived category $\bD(A)$
has infinite $n$-order for every $n\geq 1$.
\end{eg}

\section{Algebraic triangulated categories}\label{sec-algebraic}

An important class of triangulated categories are the
{\em algebraic} triangulated categories, those 
triangulated categories that admit a `differential graded model'. 
In this section we review algebraic triangulated categories
and explain why all algebraic triangulated categories are also topological,
see Proposition~\ref{prop-algebraic is topological}.
Our main new result is then Theorem~\ref{thm-general algebraic}, 
the existence of mod-$n$ reductions for algebraic triangulated categories; 
by Proposition~\ref{prop-reduction implies infinite order} above,
this implies  that algebraic triangulated categories have infinite $n$-order.

The earliest formalization of differential graded models
for triangulated categories seems to be the notion of 
{\em enhanced triangulated category}
of Bondal and Kapranov~\cite[\S 3]{bondal-kapranov}.
We deviate from the standard conventions 
in two minor points. First, we grade complexes {\em homologically}
(as opposed to cohomologically), i.e., differentials decrease the
degree by~1; this is more in tune with grading conventions in topology. 
Second, we use covariant representable functors 
(as opposed to contravariant representable functors),
which amounts to the passage to opposite dg categories.

A {\em differential graded category}, or simply {\em dg category},
is a category $\cB$ enriched in chain complexes of abelian groups.
So a dg category consists of a class of objects,
a chain complex $\cB(X,Y)$ of morphisms for every pair of objects, 
and composition morphisms of chain complexes
$$ \cdot \ : \ \cB(Y,Z) \tensor \cB(X,Y) \ \to \ \cB(X,Z)$$
for every triple of objects.
The composition morphisms have to be associative and have to admit
two-sided unit elements $1_X\in \cB(X,X)_0$ for all objects $X$,
satisfying $d(1_X)=0$.

A {\em $\cB$-module} is a covariant dg functor from $\cB$ to chain complexes
of abelian groups. In more detail, a $\cB$-module $M$ assigns 
to each object $Z$ of $\cB$ a chain complex
$M(Z)$ and to each pair of objects a morphism of chain complexes
$$ \cdot \ : \ \cB(Y,Z) \tensor M(Y) \ \to \ M(Z)\ .$$
This data is required to be associative with respect to the composition
in $\cB$ and the unit cycles have to act as identities. 
A $\cB$-module $M$ is {\em representable} if there exists a pair $(Y,u)$
consisting of an object~$Y$ of $\cB$ and a {\em universal 0-cycle}
$u\in M(Y)_0$ such that for every object $Z$ of $\cB$ the evaluation morphism
\[ \cB(Y,Z)\ \to \ M(Z) \ , \quad \varphi \ \longmapsto \ \varphi\cdot u \]
is an isomorphism of chain complexes. 

A dg category $\cB$ is {\em pretriangulated} if it has a zero object 
and the following two closure properties:
\begin{enumerate}[(a)]
\item (Closure under shifts)
For an object $X$ of a dg category $\cB$ and an integer $n$ 
we define the $\cB$-module $\cB(X,-)[n]$ on objects by
\[ \big(\cB(X,Z)[n]\big)_{n+k}\ = \ \cB(X,Z)_k \]
with differential and action of morphisms by
\[ d(f[n])\ = \ (-1)^n\cdot (df)[n] \text{\qquad respectively\qquad}
 \varphi\cdot(\psi[n]) \ = \ (-1)^{n|\varphi|}\cdot (\varphi\psi)[n]\ .\]
Here we use the notation $f[n]$ when we consider an element $f\in \cB(X,Z)_k$
as an element of $\big(\cB(X,Z)[n]\big)_{n+k}$.
Then for every object $X$ of $\cB$ and every integer $n$, the $\cB$-module
$\cB(X,-)[n]$ is representable.
\item (Closure under cones)
Given a closed morphism $f:X\to Y$ (i.e., a 0-cycle in $\cB(X,Y)$)
we consider the $\cB$-module $M$ defined on objects by
\[ M(Z)_k \ = \  \cB(Y,Z)_k\oplus\cB(X,Z)_{k+1}\]
with differential and action of morphisms by
\[ d(a, b)\ = \ (d(a),\ af-d(b))
\text{\qquad respectively\qquad}
\varphi\cdot (a, b)\ = \ (\varphi a,\ (-1)^{|\varphi|}\cdot \varphi b)\ .\]
Then the $\cB$-module $M$ is representable.
\end{enumerate}

Underlying any dg category $\cB$ is a preadditive `cycle category'
$\cZ(\cB)$ with the same objects as~$\cB$, but with morphisms
given by the 0-cycles in the morphism chain complexes, i.e.,
$\cZ(\cB)(X,Y)=\ker(d:\cB(X,Y)_0\to \cB(X,Y)_{-1})$.
The {\em homology category $\bH(\cB)$} 
(also called the `homotopy category') is a quotient
of the cycle category; it also has the same objects as $\cB$, 
but has as morphism sets the 0-th homology groups
of the homomorphism complexes, $\bH(\cB)(X,Y)=H_0(\cB(X,Y))$.

If a dg category $\cB$ is pretriangulated, 
then the homology category $\bH(\cB)$ 
can be canonically triangulated, as we recall now.
A {\em shift}  of an object $X$ is any object $X[1]$
that represents the module $\cB(X,-)[-1]$. Choices of shifts for all
objects of $\cB$ assemble canonically into an invertible shift
functor $X\mapsto X[1]$ on $\cB$.
The shift functor on $\bH(\cB)$ is induced by this shift functor 
on $\cB$. The distinguished triangles arise from mapping
cone sequences in $\cB$.
Given a closed morphism $f:X\to Y$ we let $Cf$ be 
 {\em mapping cone} of $f$, i.e., 
a representing object for the module in~(b) above.
The cone comes with a universal 0-cycle
\[ (i,u)\ \in \ \cB(Y,Cf)_0\oplus \cB(X,Cf)_1\ = \ M(Cf)_0\ ; \]
the cycle condition means that $d(i)=0$ and $d(u)=if$.
We let $p\in\cB(Cf,X)_{-1}$ be the element characterized by
\[ p\cdot (i,u)\ = \ (0,1_X)\ \in \ 
\cB(Y,X)_{-1}\oplus \cB(X,X)_0\ = \ M(X)_{-1}  \ .\] 
Since $(0,1_X)$ is a cycle, so is $p$, 
which is thus represented by a closed morphism $\bar p:Cf\to X[1]$.
By definition, a triangle in $\bHo(\cB)$ is {\em distinguished}
if it is isomorphic to the image of a triangle of the form
\[    X \ \xra{\ f\ }\ Y \ \xra{\ i\ }\ Cf\ \xra{\ \bar p \ }\  X[1]  \]
for some closed morphism $f$ in $\cB$.
A proof that this really makes $\bH(\cB)$ into a triangulated category
can be found in~\cite[\S 3, Prop.\,2]{bondal-kapranov}.
Alternatively, one can obtain the triangulation by
combining Proposition~\ref{prop-algebraic is topological} 
about the stable cofibration structure on $\cZ(\cB)$
with Theorem~\ref{thm-HoC is triangulated}.
A triangulated category is {\em algebraic} if it is equivalent,
as a triangulated category, to the homology category of
some pretriangulated dg category.

Algebraic triangulated categories can be introduced in
at least two other, equivalent, ways. One way is as 
the full triangulated subcategories of 
homotopy categories of additive categories,
compare Example~\ref{eg-chains in additive} below.
Another way is as the stable categories of exact Frobenius categories.
For the equivalence of these three approaches, and 
for more details, background and references, 
we refer to~\cite{bondal-kapranov},
Keller's ICM article~\cite{keller-differential graded} 
or the paper~\cite{krause-chicago} by Krause.

\medskip

Now we define the structure of a cofibration category
on the cycle category $\cZ(\cB)$ of a pretriangulated dg category $\cB$.
A closed morphism is a {\em weak equivalence} 
if it becomes an isomorphism in $\bH(\cB)$.
A closed morphism $i:A\to B$ is a {\em cofibration} if for every object $Z$
of $\cB$ the induced chain morphism $\cB(i,Z)$
is surjective and the kernel $\cB$-module
\[ Z \ \longmapsto \ \ker\big[ \cB(i,Z):\cB(B,Z)\to\cB(A,Z)\big] \]
is representable. 

We note that if $(C,u)$ represents the kernel module of $\cB(i,-)$,
then the universal 0-cycle is a closed morphism 
$u:B\to C$ such that for every object $Z$ of $\cB$ 
the sequence of cycle groups
\[ 0 \ \to \ \cZ(\cB)(C,Z) \ \xra{\ u^*\ }\ \cZ(\cB)(B,Z)
\ \xra{\ i^*\ } \ \cZ(\cB)(A,Z) \]
is exact (but the map $i^*$ need not be surjective).
This means that $u:B\to C$ is in particular a cokernel
of $i:A\to B$ in the category $\cZ(\cB)$.

\begin{eg}\label{eg-chains in additive} 
Many examples of pretriangulated dg categories
arise from additive categories as follows.
We let $\cA$ be an additive category and denote by
$C(\cA)$ the category of $\mZ$-graded chain complexes of objects in~$\cA$, 
with morphisms the chain maps of homogeneous degree~0.
This category has a natural dg enrichment: for two chain complexes~$X$ and $Y$ 
the chain complex
$\un{C}(\cA)(X,Y)$ of morphisms is given by
\[ \un{C}(\cA)(X,Y)_n \ = \ \prod_{k\in\mZ}\cA(X_k,Y_{k+n})\ ,\]
the abelian group of graded morphisms of homogeneous degree~$n$.
The differential in this complex is given on $f\in\un{C}(\cA)(X,Y)_n$ by
$$ df = d_Y\circ f - (-1)^n  f\circ d_X $$
and the composition morphisms are given by composition.

In this example, the cycle category $\cZ(\un{C}(\cA))$
is simply the category $C(\cA)$, and the homology category $\bH(\un{C}(\cA))$
is the homotopy category $\bK(\cA)$ of complexes modulo chain homotopy.
The weak equivalences specialize to the class of chain
homotopy equivalences, and the cofibrations are those chain maps
that are dimensionwise split monomorphisms.
\end{eg}

An object of a cofibration category $\cC$ is {\em fibrant}
if every acyclic cofibration out of it has a retraction.
Fibrant objects are useful because, for example, every
morphism in $\bHo(\cC)$ into a fibrant object
can be represented by a $\cC$-morphism. For more 
about fibrant objects we refer to Remark~\ref{rk-enough fibrant objects}.

\begin{prop}\label{prop-algebraic is topological} 
Let $\cB$ be a pretriangulated dg category. 
Then the cofibrations and weak equivalences
make the cycle category $\cZ(\cB)$ into a stable cofibration category 
in which every object is fibrant.
Moreover, the homotopy category $\bHo(\cZ(\cB))$
is equivalent, as a triangulated category,
to the homology category $\bH(\cB)$.
In particular,
every algebraic triangulated category is a 
topological triangulated category.
\end{prop}
\begin{proof}
A proof can be combined from the following two results in the literature: 
by~\cite[Sec.\,2.1]{keller-cyclic of exact} 
the cycle category $\cZ(\cB)$ has a preferred structure of 
Frobenius category, and by~\cite[Prop.\,4.19]{cisinski-categories derivables}
every Frobenius category admits a preferred structure of
stable cofibration category. However, it took the author
some effort to combine the arguments from these two references into
a complete proof; so for the convenience of the reader,
we give an independent and self-contained account here.

As a first step we show that every pretriangulated dg category $\cB$
is `closed under extensions' in the following sense.
Given objects $X$ and $Y$ of $\cB$,
a $\cB$-module $M$ and a short exact sequence of $\cB$-modules
\[ 0\ \to\ \cB(X,-)\ \xra{\ j\ }\ M \ \xra{\ p\ }\ \cB(Y,-)\ \to\ 0 \ ,\]
then $M$ is representable.
To see this we let $g\in M(Y)_0$ be a chain with $p_Y(g)=1_Y$.
Then $p_Y(d(g))=d(p_Y(g))=d(1_Y)=0$, so there is 
a unique $f\in\cB(X,Y)_{-1}$
with $j_Y(f)=d(g)$. Then $j_Y(d(f))=d(j_Y(f))=d(d(g))=0$, 
so also $d(f)=0$. 
We define a $\cB$-module $N$ on an object $Z$ in dimension $k$ by
\[ N(Z)_k\ = \ \cB(Y,Z)_k\oplus \cB(X,Z)_k  \ ;\]
the differential and action of $\cB$-morphisms are given by
\[ d(a, b)\ = \ (d(a),\ (-1)^kaf+d(b))
\text{\qquad respectively\qquad}
\varphi\cdot (a, b)\ = \ (\varphi a,\ \varphi b)\ .\]
This functor is representable: if $X[-1]$ is a negative shift of $X$, i.e.,
it represents $\cB(X,-)[1]$, then $f$ corresponds
to a closed morphism $\bar f:X[-1]\to Y$, and any 
mapping cone of this closed morphism represents $N$.

The maps
\[ \Phi(Z)\ : \ N(Z) \ \to \ M(Z)\ , \quad 
(a,b)\ \longmapsto \ a\cdot g\ +\ j_Z(b) \]
constitute a morphism of $\cB$-modules and make the diagram
\[ \xymatrix@C=15mm{ 
0\ar[r] & \cB(X,-) \ar[r]^-{\genfrac{(}{)}{0pt}{}{0}{1}}\ar@{=}[d] & 
N\ar[d]^\Phi\ar[r]^-{(1,0)}&
\cB(Y,-)\ar@{=}[d] \ar[r] & 0\\
0 \ar[r] &\cB(X,-) \ar[r]_-j & M\ar[r]_-p &\cB(Y,-) \ar[r] & 0} \]
commute; since both rows are short exact sequences, 
$\Phi$ is a natural isomorphism.
Altogether this shows that $M$ is representable
(by any mapping cone of $\bar f:X[-1]\to Y$).

Now we prove the axioms of a cofibration category.
Clearly, all isomorphisms are cofibrations and weak equivalences,
and every morphism out of a zero object is a cofibration.
If $i$ and $j$ are two composable cofibrations,
then for every object $Z$ of $\cB$ the chain map $\cB(ji,Z)$
is surjective since  $\cB(j,Z)$ and  $\cB(i,Z)$ are surjective.
Moreover, the sequence of $\cB$-modules
\[ 0\ \to\ \ker[\cB(j,-)] \ \xra{\ \text{incl}\ }\ \ker[\cB(ji,-)]
\ \xra{\ \cB(j,-)\ }\ \ker[\cB(i,-)]\ \to\ 0 \]
is exact.
The sub and quotient modules are representable by hypothesis;
since $\cB$ is closed under extensions, 
the kernel $\cB$-module of $\cB(ji,-)$ is also representable.
Hence $ji$ is again a cofibration, and we have proved axiom~(C1).
The 2-out-of-3 property~(C2) for weak equivalences holds since
isomorphisms in $\bH(\cB)$ have the 2-out-of-3 property.

In axiom (C3) we are given a cofibration 
$i:A\to B$ and a closed morphism $f:A\to C$. 
We define a $\cB$-module $M$ as the pullback of the diagram
of represented modules
\[ \cB(B,-)\ \xra{\cB(i,-)} \ \cB(A,-)\ \xla{\cB(f,-)} \ \cB(C,-)\ . \]
Since $\cB(i,-)$ is objectwise surjective, so is the projection
$M\to \cB(C,-)$; moreover, for every object $Z$
the sequence of $\cB$-modules
\[ 0 \ \to\ \ker[\cB(i,-)] \ \xra{\varphi\mapsto (\varphi,0)}\ M\
\xra{\ \text{proj}\ }\ \cB(C,-)\ \to\ 0 \]
is exact. The module $\ker[\cB(i,-)]$
is representable by hypothesis, so $M$ is representable.
We let~$P$ be a representing object for $M$ and
\[ (g,j)\ \in \ M(P)_0 \ = \ \text{pullback}\big[\
\cB(B,P)_0\ \xra{\cB(i,P)_0} \ \cB(A,P)_0\ \xla{\cB(f,P)_0} \ \cB(C,P)_0\ \big] \]
a universal 0-cycle. Then $g:B\to P$ and $j:C\to P$
are closed morphisms satisfying $gi=jf$, i.e., the square
on the left
\[   \xymatrix@C=12mm{ A \ar[r]^f \ar[d]_i & C \ar[d]^j &&
\cZ(\cB)(P,Z) \ar[r]^-{j^*} \ar[d]_{g^*} & 
\cZ(\cB)(C,Z) \ar[d]^{f^*} \\
B\ar[r]_-g & P &&
\cZ(\cB)(B,Z) \ar[r]_-{i^*} &\cZ(\cB)(A,Z) } \]
commutes. Moreover, the fact that $P$ represents the pullback functor
implies that for every object $Z$ of $\cB$ 
the square of cycle groups on the right
is a pullback. This is precisely the universal property of
a pushout in the cycle category $\cZ(\cB)$; 
so the left square above is a pushout. 

We already argued that  the projection
$M\to \cB(C,-)$, and hence the morphism $\cB(j,-):\cB(P,-)\to\cB(C,-)$, 
is objectwise surjective; the kernel module of $\cB(j,-)$
is isomorphic, via~$\cB(g,-)$, to the 
kernel $\cB$-module of~$\cB(i,-):\cB(B,-)\to\cB(A,-)$.
So the kernel module of~$\cB(j,-)$ is representable, and $j$ is a cofibration.

Now we assume that the cofibration $i$ is also a weak equivalence.
Then for every object~$Z$ of $\cB$,
the chain morphism $\cB(i,Z):\cB(B,Z)\to\cB(A,Z)$
is not only surjective, but also a chain homotopy equivalence.
So the kernel of $\cB(i,Z)$, and hence also the isomorphic kernel of
the epimorphism $\cB(j,Z):\cB(P,Z)\to\cB(C,Z)$, is contractible.
Hence the morphism $\cB(j,Z)$ is a chain homotopy equivalence;
in particular, the map $\bH(j,Z):\bH(P,Z)\to\bH(C,Z)$
is bijective for all objects $Z$. Thus the morphism $j$ is an
isomorphism in $\bH(\cB)$, and so $j$ is a weak equivalence.
This proves axiom~(C3).

For the factorization axiom~(C4) we construct mapping cylinders in $\cB$.
Given a closed morphism $f:X\to Y$, we denote by $Zf$
a mapping cone of the morphism $\genfrac{(}{)}{0pt}{}{1_X}{-f}:X\to X\oplus Y$
(we use that $\cB$ has sums by axiom (C3) with $A$ a zero object).
Then by definition, $Zf$ represents the $\cB$-module given in dimension $k$ by
\[ Z \ \mapsto \  \cB(X,Z)_k\oplus\cB(Y,Z)_k\oplus \cB(X,Z)_{k+1}
\ , \quad d(a,a',b)\ = \ (d(a),\ d(a'),\, a-a'f-d(b))\ ,\]
and with action of morphisms by
\[ \varphi\cdot (a,a', b)\ = \ 
(\varphi a,\,\varphi a',\ (-1)^{|\varphi|}\cdot \varphi b)\ .\]
We let 
\[ (i,j,u)\ \in \ \cB(X,Zf)_0\oplus\cB(Y,Zf)_0\oplus \cB(X,Zf)_1 \]
be a universal 0-cycle.
The cycle condition means that
$d(i)=0$, $d(j)=0$ and $d(u)=i-jf$.
We let $q\in\cB(Zf,Y)_0$ be the element characterized by
\[ q\cdot (i,j,u)\ = \ (f,1_Y,0)\ \in \ 
\cB(X,Y)_0\oplus\cB(Y,Y)_0\oplus \cB(X,Y)_1\ .\] 
Since $(f,1_Y,0)$ is a cycle, so is $q$.
So $q:Zf\to Y$ is a closed morphism and $qi=f$ is a factorization 
of the original morphism $f:X\to Y$ in $\cZ(\cB)$.
The chain map $\cB(i,Z):\cB(Zf,Z)\to\cB(X,Z)$ is isomorphic 
to the projection onto a summand, hence surjective.
The kernel $\cB$-module of the surjection $\cB(i,-)$ is an extension
of $\cB(X,-)[-1]$ and $\cB(Y,-)$, and hence representable.
So the morphism $i:X\to Zf$ is a cofibration.

We let $s\in\cB(Zf,Zf)_1$ satisfy $s\cdot(i,j,u)=(u,0,0)$. The relation
\[ d(s)\cdot(i,j,u)\ = \  d(s\cdot(i,j,u))\ = \ 
d(u,0,0)\ = \ 
(i-jf,0,u) \ = \ (1_{Zf}-jq)\cdot(i,j,u) \]
shows that $d(s)=1_{Zf}-jq$ in $\cB(Zf,Zf)_0$. Together with 
$qj=1_Y$ this shows that $j:Y\to Zf$ is inverse to $q:Zf\to Y$ in $\bH(\cB)$, 
and hence $q$ is a weak equivalence.
This concludes the verification of axiom~(C4), and hence the proof
that $\cZ(\cB)$ is a cofibration category.

In order to show that all objects are fibrant we have to show
that every acyclic cofibration $i:A\to B$ in $\cZ(\cB)$ has a retraction.
Since $i$ is a cofibration and a weak equivalence,
the chain morphism $\cB(i,A):\cB(B,A)\to\cB(A,A)$
is a surjective chain homotopy equivalence.
So the kernel of $\cB(i,A)$ is contractible.
We let $r\in\cB(B,A)_0$ be any 0-chain such that $ri=1_A$.
Then $d(r)$ is a cycle in the contractible kernel complex $\ker[\cB(i,A)]$.
So there is an element $u\in\ker[\cB(i,A)]_0$ with $d(u)=d(r)$.
The element $r-u\in\cB(B,A)_0$ is then a closed
morphism satisfying $(r-u)i=1_A$, i.e., it is the desired retraction.

Now we show that the projection $\gamma:\cZ(\cB)\to\bH(\cB)$ is a localization 
at the class of weak equivalences. 
Since $\gamma$ is the identity on objects and surjective on morphisms,
this amounts to showing that any functor from $\cZ(\cB)$
that inverts weak equivalences takes the same value on homologous morphisms.
So we let $F:\cZ(\cB)\to \cD$ be any functor that takes weak equivalences 
to isomorphisms, and we let $f,g:X\to Z$ be two closed morphisms 
with $f-g=d(h)$ for some $h\in\cB(X,Z)_1$. 
We consider the mapping cylinder $Z1_X$ of the closed morphism~$1_X$
as in~(C4) above. The cycle
\[ (f,g,h)\ \in \ \cB(X,Z)_0\oplus \cB(X,Z)_0\oplus\cB(X,Z)_1 \]
is represented by a closed morphism $H:Z1_X\to Z$
that satisfies $Hi=f$ and $Hj=g$.
By hypothesis, the functor $F$ takes the weak equivalence $q:Z1_X\to X$
to an isomorphism in $\cD$; because $qi=1_X=qj$,
the functor $F$ satisfies $F(i)=F(j)$.
This shows that
\[ F(f)\ = \ F(H)\circ F(i)\ = \ F(H)\circ F(j)\ = \ F(g)\ , \]
and concludes the proof that the homology category $\bH(\cB)$ `is' the
homotopy category of the cofibration structure.

It remains to compare the triangulation from the dg structure
with the triangulation from the cofibration structure 
(see Appendix~\ref{app-triangulation}).
For an object $X$ of $\cB$ we define an object $CX$ as the mapping cylinder, 
as in the proof of (C4) above,  
of the unique morphism $X\to 0$ to a zero object.
By the above, the morphism $i:X\to CX$ is a cofibration and
$CX$ is weakly contractible; hence we have obtained 
a cone of $X$ in the sense of cofibration categories.
The shift~$X[1]$ is a cokernel of $i:X\to CX$;
so the shift functor is a functorial lift 
of the suspension functor in $\bHo(\cZ(\cB))=\bH(\cB)$.
Since the shift functor is invertible (already in $\cZ(\cB)$),
the suspension functor of $\bHo(\cZ(\cB))$ is invertible and
the cofibration structure is stable.

Finally, we let $f:A\to B$ be a cofibration, $q:B\to B/A$ a cokernel
and $Cf$ a mapping cone of~$f$. Then the diagram
\[
\xymatrix{ A \ar@{=}[d] \ar[r]^-{\gamma(f)} &  
B \ar@{=}[d] \ar[r]^-{\gamma(i)} & 
Cf \ar[d]_\iso^{\gamma(0\cup q)} \ar[r]^-{\gamma(p)} & 
A[1] \ar@{=}[d] \\
A \ar[r]_-{\gamma(f)} & B \ar[r]_-{\gamma(q)}  & 
B/A \ar[r]_-{\delta(f)}  & A[1]}\]
commutes in $\bH(\cB)$. The upper triangle is a prototypical distinguished 
triangle in $\bH(\cB)$ arising from the dg structure, and the
lower triangle is distinguished in the triangulation from the
cofibration structure. So every `distinguished dg triangle'
is also a `distinguished cofibration triangle', and hence the two
classes of distinguished triangles coincide.
\end{proof}

The passage from pretriangulated dg categories to
stable cofibration categories described 
in Proposition~\ref{prop-algebraic is topological} 
respects the structure preserving functors. 
More precisely, we let $F:\cB\to\cB'$ be a dg functor between pretriangulated
dg categories. Then one can show that the restriction
$\cZ(F):\cZ(\cB)\to\cZ(\cB')$ 
of $F$ to the cycle categories is automatically exact (in the sense
of cofibration categories).

A well-known example of a topological triangulated category
that is not algebraic is the stable homotopy category.
To see this, we exploit that for every object $X$ of an
algebraic triangulated category the object $X/n$ 
(a cone of multiplication by $n$ on $X$) is annihilated by~$n$;
indeed, this observation is a special case of the much more general
Theorem~\ref{thm-general algebraic} below, 
but it also has a simple direct proof,
see for example~\cite[Prop.\,1]{sch-leeds}.
On the other hand, the mod-2 Moore spectrum in the stable homotopy category
is of the form $\mS/2$ for $\mS$ the sphere spectrum, and
it is well-known that $\mS/2$ is {\em not} annihilated by~2.
An account of the classical argument using Steenrod operations
can be found in~\cite[Prop.\,4]{sch-leeds}.

We will now show that in algebraic triangulated categories
the relation $n\cdot X/n=0$ holds `in a very strong sense',
i.e., the $n$-order of every object of the form $X/n$ is infinite.

\begin{theorem}\label{thm-general algebraic}
Let $\cT$ be an algebraic triangulated category and $n\geq 1$.
Then $\cT$ has a mod-$n$ reduction and thus has infinite $n$-order.
\end{theorem}
\begin{proof} We may suppose that $\cT=\bH(\cB)$
is the homology category of a pretriangulated dg category $\cB$.
We consider the new dg category $\mZ[e]\tensor\cB$
where $\mZ[e]$ is the dg category with a single object
whose endomorphism dg ring is the exterior algebra, over the integers, 
on a 1-dimensional class~$e$ with differential~$d(e)=n$. 
In more detail, the dg category $\mZ[e]\tensor\cB$ 
has the same objects as $\cB$,
the morphism complexes are defined by
\[ (\mZ[e]\tensor\cB)(X,Y) \ = \ \mZ[e]\tensor\cB(X,Y) \ , \]
and composition is given by
\[(1\tensor\varphi\ + \ e\tensor\psi)\cdot (1\tensor\varphi'\ + \ e\tensor\psi')
\ = \  1\tensor\varphi\varphi'\ + \ 
e\tensor(\psi\varphi'+(-1)^{|\varphi|}\varphi\psi') \ .  \]
A dg functor $J:\cB\to\mZ[e]\tensor\cB$ is given by $J(X)=X$
on objects and by $J(\varphi)=1\tensor\varphi$ on morphisms.

The new dg category $\mZ[e]\tensor\cB$ is typically not
closed under mapping cones, hence not pretriangulated.
We form the pretriangulated envelope $\cB/n=(\mZ[e]\tensor\cB)^{\text{pre}}$,
see~\cite[\S 1]{bondal-kapranov} (where the objects are called 
{\em twisted complexes}), 
\cite[2.2 (d)]{keller-cyclic of exact}
(where this is called the {\em exact envelope})
or~\cite[4.5]{keller-differential graded}
(where this is called the {\em pretriangulated hull}\,).
The envelope $\cB/n$ contains $\mZ[e]\tensor\cB$ as a full
dg subcategory and the inclusion $\iota:\mZ[e]\tensor\cB\to\cB/n$ 
is an initial example of a dg functor from $\mZ[e]\tensor\cB$
to a pretriangulated dg category.
The composite morphism of dg categories
\[ i= \iota\circ J\ : \  \cB \ \to \ \mZ[e]\tensor\cB \ \to \
(\mZ[e]\tensor\cB)^{\text{pre}}=\cB/n   \]
descends to an exact functor of triangulated homology categories
$$ \rho_*=\bH(i)\ : \ \bH(\cB) \ \to \ \bH(\cB/n) \ .$$
We will show that $(\bH(\cB/n),\rho_*)$ is a mod-$n$ reduction.

We define a dg functor $t:\mZ[e]\tensor\cB\to\cB$ that will eventually
give rise to a right adjoint to~$\rho_*$.
Given $X$ in $\cB$ we let $t(X)$ be a mapping cone of the closed
morphism $n\cdot 1_X$. So by definition,~$t(X)$ represents the $\cB$-module
given in dimension $k$~by
\[ Z \ \mapsto \  \cB(X,Z)_k\oplus\cB(X,Z)_{k+1}
\ , \quad d(a,b)\ = \ (d(a),\,na-d(b))\ .\]
The object $t(X)$ comes with a universal 0-cycle
\[ (\eta_X,g_X)\ \in \  \cB(X,t(X))_0\oplus\cB(X,t(X))_1\ . \]
The cycle condition means that $d(\eta_X)=0$ and $d(g_X)=n\cdot\eta_X$.
On morphism complexes we let
\[ t \ : \  (\mZ[e]\tensor \cB)(X,Y) =\mZ[e]\tensor\cB(X,Y)
\ \to\ \cB(t(X),t(Y)) \]
be the unique chain map characterized by the relations
\[ t(1\tensor\varphi\ + \ e\tensor\psi)\cdot(\eta_X,g_X) \ = \ 
(\eta_Y\varphi+g_Y\psi,\ g_Y\varphi) 
\ \in \  \cB(X,t(Y))_k\oplus\cB(X,t(Y))_{k+1}\]
for all $\varphi\in\cB(X,Y)_k$ and $\psi\in\cB(X,Y)_{k-1}$.
Compatibility with units and composition is straightforward,
so we have really defined a dg functor $t:\mZ[e]\tensor\cB\to\cB$.
The universal property of the pretriangulated envelope 
then provides a dg functor $t':\cB/n\to\cB$ that extends $t$.
We let 
\[ \rho^*\ = \ \bH(t')\ : \ \bH(\cB/n)\ \to \ \bH(\cB) \]
denote the induced exact functor on homology categories. 

We will now make $\rho^*$ into a right adjoint of $\rho_*$.
The relation $t(1\tensor\varphi)\eta_X=\eta_Y\varphi$ 
means that the homology classes of the 
closed morphisms $\eta_X:X\to t(X)=(t'(i(X))$
constitute a natural transformation 
$\eta:\Id_{\bH(\cB)}\to \bH(t'i)=\rho^*\rho_*$. 
We claim that $\eta$ is the unit of 
an adjunction between $\rho_*$ and $\rho^*$;
so we need to show that the composite
\begin{equation}\label{eq:unit_candidate}
   \bH(\cB/n)(\rho_*X,\,Z) \ \xra{\ \rho^*\ }\ 
\bH(\cB)(\rho^*(\rho_* X),\,\rho^*Z) \ \xra{\bH(\cB)(\eta_X,\rho^*Z)}
\ \bH(\cB)(X,\,\rho^*Z)  
\end{equation}
is bijective for all $X$ in $\cB$ and all $Z$ in $\cB/n$. 
In the special case when $Z=\rho_*Y$ for some $Y$ in $\cB$ 
we have $\rho^*(\rho_*Y)=t(Y)$ and 
the composite~\eqref{eq:unit_candidate} 
is the effect on $H_0$ of the chain map
\[u \ : \  \mZ[e]\tensor \cB(X,Y) = (\cB/n)(i(X),\,i(Y)) 
\ \to\ \cB(X,\,t(Y))  \]
given by
\[ u(1\tensor\varphi \ + \ e\tensor \psi) \ =\
\eta_Y\varphi \ +\ g_Y\psi\ .\]
We can describe the inverse of $u$ explicitly, as follows.
We let $r\in\cB(t(X),X)_0$ and  $p\in\cB(t(X),X)_{-1}$ be the elements
characterized by the relations
\[ r\cdot(\eta_X,g_X)\ = \ (1_X,0) \text{\qquad respectively\qquad}
 p\cdot(\eta_X, g_X)\ = \ (0,1_X) \ . \]
Then $(\eta_Xr+g_Xp)(\eta_X,g_X)=(\eta_X,g_X)$, so we must have
$\eta_Xr+g_Xp=1_{t(X)}$.
We define
\[ v \ : \  \cB(Y,t(X))\ \to  \ (\cB/n)(i(Y),i(X)) \text{\qquad by\qquad} 
v(a)\ = \ 1\tensor ra - e\tensor pa\ ,\]
and direct calculation shows that $uv$ and $vu$ are the identity maps.
This shows that the map~\eqref{eq:unit_candidate}
is bijective in the special case $Z=\rho_*Y$.

We consider the class of objects $Z$ of $\bH(\cB/n)$ such that
the map~\eqref{eq:unit_candidate}
is bijective for all~$X$ in~$\bH(\cB)$.
Since $\rho_*$ and $\rho^*$ are exact functors, 
this class forms a triangulated subcategory of~$\bH(\cB/n)$. 
Moreover, the class contains all objects of the form $\rho_*Y$, 
by the last paragraph.
The homology category $\bH(\cB/n)$ of the pretriangulated
envelope of $\mZ[e]\tensor\cB$ is generated, as a triangulated category, by 
the objects of $\bH(\mZ[e]\tensor\cB)$, and these are precisely the
objects of the from $\rho_*Y$. So the map~\eqref{eq:unit_candidate}
is always bijective, and that shows that  $\rho^*$ 
is right adjoint to $\rho_*$ with $\eta$ as adjunction unit.

For all $X$ in $\cB$, the object $\rho^*(\rho_*X)=t(X)$ 
is a mapping cone of multiplication by $n$ on $X$, by construction.
Moreover, the homology class of $\eta_X:X\to t(X)=\rho^*(\rho_* X)$ 
is the cone inclusion.
Since mapping cone sequences give rise to exact triangles in $\bH(\cB)$,
this proves the existence of a distinguished 
triangle~\eqref{eq-reduction triangle}.
\end{proof}

\begin{rk}
The construction of the mod-$n$ reduction $\cT/n$ 
in Theorem~\ref{thm-general algebraic} depends on the
choice of dg model for $\cT$, and not just on the triangulated
category $\cT$. For a specific example to illustrate this,
we can take $\cT$ as the category of $\mF_2$-vector spaces,
with identity shift functor and
the exact sequences as distinguished triangles.
There are two well-known dg models,
the dg categories of acyclic complexes of
projective modules over the rings $\mF_2[\epsilon]$ respectively $\mZ/4$.
We remark without proof that the construction 
of Theorem~\ref{thm-general algebraic} applied to these
two dg categories yields two mod-2 reductions that are
not equivalent (even as categories).
\end{rk}

\begin{rk}
I expect that rationally there is no difference whatsoever
between algebraic and topological triangulated categories.
In other words: every topological triangulated category whose
morphism groups are uniquely divisible ought to be algebraic.

There are various pieces of evidence for this claim.
On the one hand side, all invariants I know to distinguish algebraic from
topological triangulated categories vanish rationally.
For example, the $n$-order is rationally useless 
since $\mQ$-linear triangulated categories have infinite $n$-order 
for all natural numbers~$n$.
Similarly, the action of the Spanier-Whitehead category
of finite CW-complexes on a topological triangulated category 
(that one can construct from the unstable action 
of Remark~\ref{rk-Ho CW action}) is no extra information 
for $\mQ$-linear triangulated categories since
the chain functor from the Spanier-Whitehead category
to the bounded derived category of finitely generated abelian groups
is rationally an equivalence (both sides are in fact rationally equivalent
to the category of finite dimensional graded $\mQ$-vector spaces). 

Moreover, under certain technical assumptions and cardinality restrictions,
$\mQ$-linear topological triangulated categories are known to be algebraic.
More precisely, a theorem of Shipley~\cite[Cor.\,2.16]{shipley-DGAs and HZ}
says that every $\mQ$-linear spectral model category 
(a stable model category enriched over
the stable model category of symmetric spectra) with a set
of compact generators is Quillen equivalent to dg modules over
a certain differential graded $\mQ$-category. 
\end{rk}

\section{The order of Moore spectra}
\label{sec-Moore}

In the previous section we showed that every algebraic triangulated category
has infinite $n$-order for every number $n$.
In this section we show that for every prime $p$ the $p$-order 
of the stable homotopy category of finite spectra is at most $p-1$. 
This shows in particular that the stable homotopy category 
is not algebraic. The stable homotopy category is a topological
triangulated category, and in Section~\ref{sec-order in topological} below 
we will show that every topological triangulated category 
has $p$-order at least $p-1$.

In this section we work in the homotopy category of so called 
`finite spectra', i.e., the full triangulated subcategory category  $\bSH^c$ 
of compact objects in the stable homotopy category.
This category is equivalent to the 
{\em Spanier-Whitehead category}~\cite{spanier-whitehead}, 
obtained from finite based CW-complexes 
by formally inverting the suspension functor.
We denote by $\mS$ the sphere spectrum as an object of 
the stable homotopy category.
The {\em mod-$n$ Moore spectrum} is a cone of multiplication by $n$ 
on the sphere spectrum, i.e., it is part of a distinguished triangle
\begin{equation}\label{eg-triangle defining Moore} 
\mS \ \xra{\,n\cdot\,} \ \mS \ \to \ \mS/n \ \to \ \Sigma \mS \ . 
\end{equation}
Alternatively, $\mS/n$ can be defined as a suitably desuspended suspension
spectrum of a mod-$n$ Moore space.
The mod-$n$ Moore spectrum is characterized up to 
isomorphism in the category $\bSH^c$ by the property
that its integral spectrum homology is concentrated in dimension zero 
where it is isomorphic to $\mZ/n$.
For a prime~$p$ the mod-$p$ cohomology of $\mS/p$ is
one-dimensional in dimensions~0 and~1, and trivial otherwise,
and the Bockstein operation is non-trivial from dimension~0 to
dimension~1.

The morphism $2\cdot \mS/2$ is nonzero in the stable homotopy category,
so the triangulated category~$\bSH^c$ and any triangulated category 
which contains it is not algebraic by Theorem~\ref{thm-general algebraic}.
However, for odd primes~$p$ we have $p\cdot\mS/p=0$
and we use the concept of $p$-order to
show the homotopy category of finite $p$-local spectra is not algebraic.
We denote by~$\bSH^c_{(p)}$ the category of 
{\em finite $p$-local spectra}, i.e., the compact objects in the 
triangulated category of $p$-local spectra. The finite $p$-local spectra
are precisely the $p$-localizations of objects in $\bSH^c$,
but they are not in general compact in the larger category $\bSH$
(in other words, $\bSH^c_{(p)}$ is larger than the intersection of 
$\bSH^c$ and $\bSH_{(p)}$).

\begin{theorem}\label{thm-order of Moore}
Let $p$ be a prime. Then the mod-$p$ Moore
spectrum $\mS/p$ has $p$-order at most $p-2$ in the triangulated
category $\bSH^c_{(p)}$ of finite $p$-local spectra. 
Hence the category $\bSH^c_{(p)}$ has $p$-order at most $p-1$, 
and is not algebraic.
\end{theorem}

The category $\bSH^c_{(p)}$ of finite $p$-local spectra 
is contained in the stable homotopy category $\bSH$. 
Since the stable homotopy category is a topological triangulated
category, its $p$-order and that of any subcategory is at least
$p-1$ by Theorem~\ref{thm-general topological}. 
By Theorem~\ref{thm-order of Moore} the $p$-order of $\bSH^c_{(p)}$, 
and hence of every triangulated category which contains it, is at most $p-1$. 
So we conclude that any triangulated category
which sits between $\bSH^c_{(p)}$ and $\bSH$ has $p$-order exactly $p-1$.

\begin{proof}
For $p=2$ the theorem just rephrases the fact that $2\cdot\mS/2\ne 0$. 
So we assume for the rest of the proof that $p$ is an odd prime.
Readers familiar with the proof of the rigidity theorem 
for the stable homotopy category will recognize the following arguments 
as a key step in~\cite[Thm.\,3.1]{schwede-rigid}.
We will use mod-$p$ cohomology operations
and some knowledge about stable homotopy groups of spheres.
We adopt the standard abbreviation $q=2p-2$.
Then the Steenrod operation $P^i$ has degree $iq$.
Below we will use the Adem relation
\begin{equation}\label{Adem relation}
P^pP^{(j-1)p} \ = \ j\cdot P^{jp}\ +\ P^{jp-1}P^1
\end{equation}
which holds for all positive~$j$.

We say that a finite $p$-local spectrum $X$
{\em satisfies condition $(C_i)$} if its 
mod-$p$ cohomology is 1-dimensional in dimensions
$jpq$ and $jpq+1$, connected by a Bockstein operation, for $j=0,\dots,i$ 
and trivial in all other dimensions and the Steenrod operation 
$P^{ip}:H^0(X,\mF_p)\to H^{ipq}(X,\mF_p)$ is an isomorphism. 

For the course of this proof we write
$$\mS^i_{(p)} \ = \ \Sigma^i\, \mS_{(p)}$$
for the $i$-dimensional $p$-local sphere spectrum.\medskip

Step 1: Let $X$ be a finite $p$-local spectrum
which satisfies condition $(C_i)$ for some $i$ between~0 and $p-1$.
Then there exists a morphism 
$$ f\ :\ \mS^{(i+1)pq-1}_{(p)}\ \to\ X$$ 
which is detected by the operation $P^p$, i.e., such that in the
mod-$p$ cohomology of any mapping cone $C(f)$ the operation
$P^{p}:H^{ipq}(C(f),\mF_p)\to H^{(i+1)pq}(C(f),\mF_p)$ is an isomorphism. 

For the proof of this step we let $X^{(n)}$ denote a stable 
$p$-local $n$-skeleton of $X$. 
Since $X/X^{(ipq-1)}$ has the mod-$p$ cohomology of the suspended Moore
spectrum $\Sigma^{ipq}\mS/p$, it is isomorphic to $\Sigma^{ipq}\mS/p$ in 
the stable homotopy category. 
Hence there exists a morphism 
$$\tilde\beta_1 \ :\ \mS^{(i+1)pq-1}_{(p)} \ \to\ X/X^{(ipq-1)}$$ 
which is detected by $P^p$, see for example page~60
of~\cite[\S 5]{toda-realizing} (the notation indicates that $\tilde\beta_1$
can be chosen so that the composite with the appropriate shift of
the pinch map $\mS/p\to \mS^1_{(p)}$ 
is the class~$\beta_1$ that generates the $p$-component
of the stable stem of dimension $pq-2$).
Now we claim that $\tilde\beta_1$ can be lifted 
to a morphism $f:\mS^{(i+1)pq-1}_{(p)}\to X$.

The obstruction to lifting $\tilde\beta_1$ to $X$ is the composite morphism
$$  \mS_{(p)}^{(i+1)pq-1} \ \xra{\ \tilde\beta_1\ }\ X/X^{(ipq-1)} 
\  \to \  \Sigma X^{(ipq-1)} \ , $$
where the second map is the connecting morphism.
Since $\Sigma X^{(ipq-1)}$ has stable cells in dimensions
$jpq+1$ and $jpq+2$ for $j=0,\dots,i-1$,
the obstructions lie in the $p$-local stable stems
of dimension $jpq-3$ and $jpq-2$ for $j=2,\dots,i+1$;
by  serious calculations it is known that indeed for $j=2,\dots,p$, 
the $p$-components of the stable stems
of dimension $jpq-3$ and $jpq-2$ are trivial.
We give detailed references for this calculation in 
the proof of Theorem~3.1 of~\cite{schwede-rigid}.
Since the obstruction group vanishes, there exists a morphism
$f:\mS_{(p)}^{(i+1)pq-1}\to X$ as above.\medskip

Step 2: We show that there does not exist any finite $p$-local
spectrum $X$ which satisfies condition $(C_{p-1})$.
Suppose to the contrary that such an $X$ exists. 
By Step~1 there is a morphism $f:\mS_{(p)}^{p^2q-1}\to X$ which is 
detected by the operation $P^p$, and we consider its mapping cone $C(f)$.
Since $X$ satisfies $(C_{p-1})$, the composite Steenrod operation
$$ P^pP^{(p-1)p} \ : \ H^0(C(f);\mF_p) \ \to \
H^{p^2q}(C(f);\mF_p)  $$
is non-trivial.
On the other hand, for $i=p$ the Adem relation~\eqref{Adem relation}
becomes $P^pP^{(p-1)p}=P^{p^2-1}P^1$.
Since the operation $P^1$ is trivial in the
cohomology of $C(f)$ for dimensional reasons,
we arrive at a contradiction, which means that
there is no object $X$ satisfying condition $(C_{p-1})$.\medskip

Step 3: We show by downward induction on $i$ that 
an object $X$ of the category~$\bSH^c_{(p)}$ 
which satisfies condition $(C_i)$ has $p$-order at most $p-i-2$.
We start the induction with $i=p-1$, where the conclusion 
`$X$ has $p$-order at most $-1$' means that
there does not exist such an $X$, which was shown in Step~2.

For the inductive step we assume known that for some 
$i=0,\dots,p-2$ every object which satisfies condition $(C_{i+1})$ 
has $p$-order at most $p-i-3$. Let $X$ be an object 
which satisfies condition $(C_{i})$. We consider a morphism 
$f:\mS_{(p)}^{(i+1)pq-1}\to X$ that is detected by the operation~$P^p$,
and which exists by Step~1. We claim that any mapping cone $C(\bar f)$
of any extension $\bar f:\Sigma^{(i+1)pq-1}\mS/p\to X$ 
satisfies condition $(C_{i+1})$, so it has $p$-order at most $p-i-3$.
This proves that $X$ has $p$-order at most $p-i-2$.

Indeed, by attaching a copy of the suspended mod-$p$ Moore spectrum 
$\Sigma^{(i+1)pq-1}\mS/p$ to~$X$, the mod-$p$ cohomology increases 
by one copy of $\mF_p$ in dimensions $(i+1)pq$ and $(i+1)pq+1$, 
connected by a Bockstein operation, and it remains unchanged 
in all other dimensions.
So the mapping cone $C(\bar f)$ has its mod-$p$ cohomology in the
right dimensions. Since $f$ is detected by~$P^p$, 
the composite Steenrod operation
$$ P^pP^{ip} \ : \ H^0(C(\bar f);\mF_p) \ \to \
H^{(i+1)pq}(C(\bar f);\mF_p)  $$
is an isomorphism.
By the Adem relation~\eqref{Adem relation} and since $P^1$ acts trivially
for dimensional reasons, $P^{(i+1)p}$  acts as a unit multiple of $P^pP^{ip}$,
and thus as an isomorphism. This proves that $C(\bar f)$
satisfies condition $(C_{i+1})$ and finishes Step~3.\medskip

Now we draw the final conclusion. The mod-$p$ Moore spectrum $\mS/p$
satisfies condition $(C_0)$, so by Step~3 for $i=0$ it has $p$-order
at most $p-2$.
\end{proof}

\section{\texorpdfstring{Review of $\Delta$-sets}{Review of Delta-sets}}
\label{sec-Delta review}

In this section we recall $\Delta$-sets and review some of their properties. 
In the later sections we will use actions of
$\Delta$-sets on cofibration categories to establish 
lower bounds for the $p$-order, $p$ any prime, in topological 
triangulated categories.
A general reference for $\Delta$-sets is the paper~\cite{rourke-sanderson}
by Rourke and Sanderson.

We let $\Delta$ denote the category whose objects
are the totally ordered sets $[n]=\{0<1<\dots<n\}$ for $n\geq 0$,
and whose morphisms are the injective monotone maps.
A {\em $\Delta$-set} (sometimes called a {\em semisimplicial set}
or a {\em presimplicial set})
is a contravariant functor from the category
$\Delta$ to the category of sets; a morphism of $\Delta$-sets is a
natural transformation of functors. 
We write $K_n=K([n])$ for the value of a $\Delta$-set
$K:\Delta^\text{op}\to\text{(sets)}$ and call the elements
of this set the {\em $n$-simplices} of $K$. For a morphism
$\alpha:[m]\to [n]$ in $\Delta$ and an $n$-simplex $x$ of $K$
we write
\[ x\alpha \ = \ K(\alpha)(x) \ \in \ K_m \]
for the effect of the map induced by $\alpha$. The functor
property then becomes the relation $(x\alpha)\beta=x(\alpha\beta)$.
A $\Delta$-set $K$ is {\em finite} if the disjoint union of all
the sets $K_n$ is finite. Equivalently, $K$ is finite if each $K_n$ is finite 
and almost all $K_n$ are empty.

For $0\leq i\leq n$ we denote by $d_i:[n-1]\to[n]$ 
the unique morphism in $\Delta$ whose image does not contain~$i$.
A $\Delta$-set can be defined by specifying the sets of simplices
and the {\em face maps}, i.e., the effect of the morphisms  $d_i$.
These maps have to satisfy the relations
\[ xd_jd_i \ = \ xd_id_{j-1} \text{\qquad for all $i<j$.}\]
We will often specify the faces of an $n$-simplex $x$
in a compact way by writing 
$$\partial x\ =\ (xd_0,\,xd_1,\dots,xd_n)\ . $$
The {\em geometric realization} of a $\Delta$-set $K$ is the topological space
\begin{equation}\label{eq:Delta_set_realization}
   |K| \ = \ \bigcup_{n\geq 0} \ K_n\times\nabla^n / \sim \ .
\end{equation}
Here $\nabla^n$ is the topological $n$-simplex (the convex hull of
the standard basis vectors in $\mR^{n+1}$), and
the equivalence relation is generated by 
$$ (x\alpha,t) \ \sim \ (x,\alpha_*t)$$
for all $x\in K_m$, $t\in\nabla^n$ and $\alpha:[n]\to[m]$,
where $\alpha_*(t_0,\dots,t_n)=(s_0,\dots,s_m)$ with
$s_i=\sum_{\alpha(j)=i}t_j$.
A morphism of $\Delta$-sets is a {\em weak equivalence}
if it becomes a homotopy equivalence after geometric realization.
A $\Delta$-set is {\em weakly contractible}
if its geometric realization is contractible.
We emphasize that although the category of $\Delta$-sets
has useful notions of cofibrations (the monomorphisms)
and weak equivalences,  $\Delta$-sets do {\em not} 
form a cofibration category because the factorization axiom~(C4) fails.

\begin{eg} Important examples are the representable $\Delta$-sets 
$\Delta[n]=\Delta(-,[n])$, for $n\geq 0$.  
For any $\Delta$-set $K$, the Yoneda lemma says that 
evaluation at the unique $n$-simplex $\Id_{[n]}$ of $\Delta[n]$ is 
a natural bijection from the morphism set $\Delta\text{-set}(\Delta[n],K)$
to the set $K_n$ of $n$-simplices of~$K$.
The maps
\[ \nabla^n \ \to \ |\Delta[n]|\ , \quad t\ \longmapsto\ [\Id_{[n]},t] 
\text{\qquad and\qquad}
|\Delta[n]|\ \to \ \nabla^n \ , \quad [\alpha,t]\ \longmapsto\ \alpha_*(t) \]
are mutually inverse homeomorphisms between the topological $n$-simplex 
and the geometric realization of $\Delta[n]$.

Important sub-$\Delta$-sets of $\Delta[n]$
are the {\em boundary} $\partial\Delta[n]=\Delta[n]-\{\Id_{[n]}\}$
and the {\em horns} 
\[ \Lambda^i[n]\ = \ \partial\Delta[n]-\{d_i\} \ =\ \Delta[n]-\{\Id_{[n]},d_i\} \]
for $0\leq i\leq n$. The geometric realization of $\partial\Delta[n]$
maps homeomorphically onto the boundary of the topological simplex $\nabla^n$;
the geometric realization of $\Lambda^i[n]$
maps homeomorphically onto the $i$-horn of $\nabla^n$, i.e., the
boundary with the interior of the $i$-th face removed.
\end{eg}

\begin{defn}\label{def-elementary expansion}
Let $K$ be a sub-$\Delta$-set of $L$. The inclusion
$K\to L$ is an {\em elementary expansion} of dimension~$n$ if there is
an $n$-simplex $e\in L_n-K_n$ and an $i\in\{0,\dots,n\}$ such that
$L$ is the disjoint union of $K$ and $\{e, ed_i\}$, 
and $ed_j\in K$ for all $j\ne i$.  
\end{defn}

Elementary expansions can be characterized as pushouts of horn inclusions: 
$K\to L$ is an elementary expansion of dimension $n$ if 
and only if there is a pushout
\begin{equation}\begin{aligned}\label{eq:expansion_pushout}
\xymatrix{ \Lambda^i[n] \ar[r]\ar[d] & \Delta[n]\ar[d]\\
K\ar[r] &L}
\end{aligned}\end{equation}
for some $i\in\{0,\dots,n\}$. The simplex $e$ in the definition of
elementary expansion is then the image of
the $n$-simplex of $\Delta[n]$.

Geometric realization has a right adjoint, so it commutes with pushouts. 
So if $K\to L$ is an elementary expansion of dimension~$n$, then the realization
$|L|$ is obtained from $|K|$ by attaching a topological simplex
along a horn, so the inclusion $|K|\to |L|$ is a homotopy equivalence.
Hence every elementary expansion of $\Delta$-sets is a weak equivalence.

\begin{eg}\label{eg-cone}
For a $\Delta$-set $K$ we denote by $CK$ the {\em cone} of $K$, defined by
$(CK)_0=K_0\amalg\{\ast\}$ and
\[
(CK)_n \ = \   K_n\ \amalg\ \{\sigma x\ |\ x\in K_{n-1} \} \ ,
\]
for $n\geq 1$. The face operators are determined by requiring that
the inclusion of $K_n$ as the first summand of $(CK)_n$ makes
$K$ a sub-$\Delta$-set of $CK$, and by the formulas
$$ (\sigma x)d_i \ = \ 
\begin{cases}
\sigma(xd_i) & \text{ for $0\leq i< n$ and}\\
\quad x & \text{ for $i=n$,}
\end{cases}$$
with the interpretation $(\sigma x)d_1=*$ for $x\in K_0$.
For example, the unique morphism $\Delta[n+1]\to C\Delta[n]$ is an isomorphism.
The geometric realization $|CK|$ is homeomorphic to the cone of $|K|$,
hence contractible.
If $K'$ is a sub-$\Delta$-set of $K$ and $K-K'$ consists of a single
$n$-simplex $x$, then the inclusion $CK'\to CK$ has complement
$\{x,\,\sigma x\}$ and is an elementary expansion of dimension $n+1$.
So if $K$ is finite, then the inclusion $\{*\}\to CK$ of the cone
point is a composite of elementary expansions, one for each simplex of $K$.
\end{eg}

Categorical products of $\Delta$-sets are not homotopically well behaved; 
this is one of the reasons why simplicial sets are preferable 
for many purposes.
For example, the geometric realization 
of the categorical product $\Delta[1]\times\Delta[1]$ 
is not even connected. However, there is another construction,
the {\em geometric product} $K\tensor L$ of two $\Delta$-sets $K$ and $L$,
defined as follows.
An $n$-simplex of $K\tensor L$ is an equivalence class of triples
$(x,y;\,\varphi)$ where $x\in K_i$, $y\in L_j$ and 
$\varphi:[n]\to[i]\times[j]$ is an injective monotone map.
The equivalence relation is generated by
\begin{equation}
  \label{eq:generate_relation_geometric}
 (x\alpha,y\beta;\,\varphi) \ \sim \ 
(x,y;\,(\alpha\times\beta)\varphi)    
\end{equation}
for morphisms $\alpha$ and $\beta$ in the category $\Delta$.
Every equivalence class has a preferred representative,
namely the unique triple $(x,y;\,\varphi)$ where both components 
$\varphi^1:[n]\to[i]$ and $\varphi^2:[n]\to[j]$ of~$\varphi$ 
are surjective; however, we will not use this.
A morphism $\nu:[m]\to[n]$ of~$\Delta$ acts on the
equivalence class of a triple by
\[  [x,y;\,\varphi]\nu \ = \  [x,y;\,\varphi\nu] \ . \]
For example, every $n$-simplex of $\Delta[i]\tensor\Delta[j]$ 
has a unique representative of the form $(\Id_{[i]},\Id_{[j]};\,\varphi)$
for an injective monotone map $\varphi:[n]\to[i]\times[j]$,
so $\Delta[i]\tensor\Delta[j]$ is isomorphic to the $\Delta$-set
of injective monotone maps into $[i]\times[j]$.

The geometric product is symmetric monoidal with unit object $\Delta[0]$.
The unit isomorphism sends an
$n$-simplex $[x,\Id_{[0]};\,\varphi]$ of $K\tensor\Delta[0]$ to $x\in K_n$. 
The associativity isomorphism 
$(K\tensor L)\tensor M\iso K\tensor(L\tensor M)$ sends an $m$-simplex
$[[x,y;\,\varphi],z;\,\psi]$ to
\[ [x,[y,z;\,\varphi'];\,\psi'] \ \in \ ( K\tensor(L\tensor M))_m\ , \]
where $\varphi':[l]\to[j]\times[k]$ and $\psi':[m]\to[i]\times[l]$
are the unique monotone injections that make the square 
\[\xymatrix@C=18mm{ [m]\ar[r]^-{\psi} \ar[d]_{\psi'} & 
[n]\times[k]\ar[d]^{\varphi\times[k]}\\
[i]\times[l]\ar[r]_-{[i]\times\varphi'}& [i]\times[j]\times[k]} \]
commute. The symmetry isomorphism 
$K\tensor L\iso L\tensor K$ sends 
$[x,y;\,\varphi]$ to $[y,x;\,\tau\varphi]$,
where $\tau:[i]\times[j]\to[j]\times[i]$ interchanges the factors.

The product $\tensor$ is called `geometric' because geometric realization 
takes it to cartesian product of topological spaces.
More precisely, a continuous natural map
$p_K:|K\tensor L|\to|K|$ is defined by sending the equivalence class of 
$$ ([x,y;\,\varphi],\, t) \ \in \ (K\tensor L)_n\times\nabla^n$$
to the class of $(x,\varphi^1_*(t))\in K_i\times\nabla^i$,
where $\varphi^1:[n]\to[i]$ is the first component 
of $\varphi:[n]\to[i]\times[j]$. One should beware that $p_K$ is not
induced by a morphism of $\Delta$-sets from $K\tensor L$ to $K$.
There is an analogous map for the second factor $L$, and the 
combined map
\begin{equation}\label{eq:geometric_versus_cartesian_product}
 (p_K,p_L) \ : \ |K\tensor L| \ \to \ |K|\times |L|  
\end{equation}
is a homeomorphism (when the product is given the
compactly generated topology).
It follows that the functor $K\tensor-$ preserves weak equivalences
of $\Delta$-sets.

A general injective weak equivalence between finite $\Delta$-sets 
is {\em not} a sequence of elementary expansions.
Nevertheless, it is well known that the localization 
of the category of $\Delta$-sets
at the class of elementary expansions agrees with the
localization of the category of $\Delta$-sets
at the class of weak equivalences, and the resulting homotopy category
is equivalent to the homotopy category of CW-complexes,
compare Section~I.4 of~\cite{buonchristiano-rourke-sanderson}.
For the convenience of the reader we recall a proof of the first of
these two facts inside of the category of {\em finite} $\Delta$-sets.

\begin{prop}\label{prop-expansions generate equivalences}
Let $F$ be a functor defined on the category of finite $\Delta$-sets
that takes elementary expansions to isomorphisms.
Then $F$ takes all weak equivalences to isomorphisms.
\end{prop}
\begin{proof}
We let $f:K\to L$ be a weak equivalence between finite $\Delta$-sets.
By `filling all horns' (cf.\,\cite[p.\,334]{rourke-sanderson}) 
we obtain a sequence of elementary expansions
$$ K=K^0 \subseteq K^1\subseteq K^2\subseteq \cdots \subseteq K^n\subseteq 
\cdots  $$
such that the union $K^\infty=\cup_{n\geq 0}K^n$ is a Kan $\Delta$-set, i.e.,
every morphism of $\Delta$-sets from a horn~$\Lambda^i[n]$ to $K^\infty$
can be extended to the simplex $\Delta[n]$.
Of course, $K^\infty$ is no longer finite.

We let $i_0, i_1:K\to K\tensor\Delta[1]$ be the `front and back inclusion', 
i.e., the morphisms defined on $x\in K_n$  by 
\[ i_0(x) \ = \  [x,d_1;\, \psi_n] \text{\qquad respectively\qquad} 
i_1(x) \ = \  [x,d_0;\, \psi_n]\ , \]
where $\psi_n:[n]\to[n]\times[0]$ is the unique monotone injection.
We form the mapping cylinder $Mf=K\tensor\Delta[1]\cup_fL$, i.e.,
a pushout of the diagram:
$$\xymatrix@C=15mm{K \ar[d]_-{i_1} \ar[r]^f &  L\ar[d]^j \\
K\tensor\Delta[1]  \ar[r]_-g & Mf}$$ 
The back inclusion $i_1$ is a sequence of elementary expansions, 
hence so is its cobase change $j:L\to Mf$. The geometric realization
of $Mf$ is homeomorphic to the topological mapping cylinder of
$|f|:|K|\to |L|$ and the monomorphism $gi_0:K\to Mf$
becomes the front inclusion after realization. Hence
$|gi_0|:|K|\to |Mf|$ is a homotopy equivalence since~$f$ is, and
$|K|$ is a retract of $|Mf|$. By the `generalized extension property'
of Kan $\Delta$-sets~\cite[Cor.\,5.4]{rourke-sanderson} there is
a morphism of $\Delta$-sets $\varphi:Mf\to K^\infty$ such that
$\varphi gi_0:K\to K^\infty$ is the inclusion.
Since $K$ and $L$ are finite $\Delta$-sets, so is 
the mapping cylinder $Mf$. The image of $\varphi:Mf\to K^\infty$
is thus contained in $K^n$ for some $n\geq 0$.

Now let $F$ be a functor to some category $\cD$, 
defined on the category of finite $\Delta$-sets, 
that takes elementary expansions to isomorphisms. 
We apply $F$ to the various morphisms 
of finite $\Delta$-sets and obtain a commutative diagram in $\cD$:
$$\xymatrix@C=18mm{F(K) \ar[d]_{F(i_1)}^{\iso} \ar[r]^-{F(f)} & 
F(L) \ar[d]^-{F(j)}_-\iso \\
F(K\tensor\Delta[1]) \ar[r]_{F(g)}& 
F(Mf) \ar[d]^{F(\varphi)} \\
F(K)\ar[r]_{\iso}  \ar[u]^{F(i_0)}_{\iso} & F(K^n)}$$
The morphisms decorated with symbols `$\iso$'
are isomorphisms since they arise from morphisms
of $\Delta$-sets that are sequences of elementary expansions.
So the diagram shows that $F(f)$ has a left inverse in $\cD$
and $F(\varphi)$ has a right inverse in $\cD$.

Now we apply the same argument to the weak equivalence $\varphi:Mf\to K^n$
instead of $f$. We deduce that $F(\varphi)$ 
has a left inverse in $\cD$. Since  $F(\varphi)$ 
has a left inverse and a right inverse, it is an isomorphism. Hence the
morphism $F(f)$ is also an isomorphism.
\end{proof}

\section{Frames in cofibration categories}
\label{sec-framings}

In this section we develop the technique of `framings', or
`$\Delta$-resolutions' in cofibration categories. Framings 
are a way to construct homotopically meaningful pairings with $\Delta$-sets,
compare the notion of `$\Delta$-cofibration category' in
Definition~\ref{def-Delta cofibration category}.
In Section~\ref{sec-pointed Delta} we need such a pairing 
to talk about actions of Moore spaces on objects of a cofibration category. 

In the context of Quillen model categories, 
the theory of framings is well established 
and goes back to Dwyer and Kan~\cite[4.3]{DK-function}, 
who use the terminology {\em (co-)simplicial resolutions}. 
In our more general setup of cofibration categories we cannot work with
cosimplicial objects -- the lack of fibrations and matching objects
in a cofibration category does not allow the construction
of codegeneracy maps in a frame. The solution is not to ask
for codegeneracy morphisms, i.e., to work
with co-$\Delta$-objects instead of cosimplicial objects.
In other words: $\Delta$-sets are to cofibration categories 
what simplicial sets are to Quillen model categories.

\medskip

As before we let $\Delta$ denote the category whose objects
are the totally ordered sets $[n]=\{0<1<\dots<n\}$ for $n\geq 0$,
and whose morphisms are the injective monotone maps.
A {\em co-$\Delta$-object} in a category~$\cC$ is a
covariant functor $A:\Delta\to\cC$. Morphisms of co-$\Delta$-objects are
natural transformations of functors.
In a co-$\Delta$-object $A$ we typically write $A^n$ for $A([n])$.

Now we discuss how to pair co-$\Delta$-objects with $\Delta$-sets.
If $A$ a co-$\Delta$-object in $\cC$ and $Z$ an object of $\cC$,
then the composite functor
$$ \Delta^\text{op} \ \xra{\ A\ } \ \cC^\text{op} \ 
\xra{\ \cC(-,Z)\ } \ \text{(sets)}$$
is a $\Delta$-set that we denote by $\cC(A,Z)$.
If $K$ is a $\Delta$-set, then we denote by $K\cap A$
a representing object, if it exists, of the functor
$$ \cC \ \to \ \text{(sets)}\ , \quad Z \longmapsto\ 
\Delta\text{-set}(K,\cC(A,Z)) \ .$$
In more detail, $K\cap A$ is an initial example of a $\cC$-object
equipped with morphisms $x_*:A^n\to K\cap A$ for
every $n\geq 0$ and every $n$-simplex $x$ of $K$, such that
for all morphisms $\alpha:[m]\to [n]$ in the category $\Delta$ the composite
$$ A^m \ \xra{\ \alpha_*\ } \ A^n \ \xra{\ x_*\ }\
K\cap A$$
is equal to $(x\alpha)_*$. 
The universal morphisms $x_*$ are
part of the data, but we will often omit them from the notation.

The defining property of $K\cap A$ can be rephrased in several other ways.
For example, $K\cap A$ is a colimit of the composite functor 
\[ s(K) \ \xra{(n,x)\mapsto[n]} \ \Delta \ \xra{\quad A\quad }\ \cC\ . \]
Here $s(K)$ is the simplex category of $K$:
objects are pairs $(n,x)$ with $x\in K_n$,
and morphisms from  $(n,x)$ to $(m,y)$ are those $\Delta$-morphisms
$\alpha:[n]\to[m]$ that satisfy $y\alpha=x$.
If $\cC$ has coproducts, then $K\cap A$ is also a coend of the functor
\[ \Delta^\text{op}\times \Delta \ \to \ \cC \ , \quad
([n],[m]) \  \longmapsto \  K_n\times A^m \ , \]
where $K_n\times A^m$ denotes a coproduct 
of copies  of $A^m$ indexed by the set~$K_n$.

\begin{eg}
By the Yoneda lemma, the object $A^n$ represents the functor
$\Delta\text{-set}(\Delta[n],\cC(A,-))$. 
So for every co-$\Delta$-object $A$ we can -- and will -- take
$\Delta[n]\cap A=A^n$ with respect to the structure morphisms of $A$.

The product of a co-$\Delta$-object with the boundary $\partial\Delta[n]$
of a simplex will play an important role in the following, 
so we use the special notation
\[ \partial^n\!A \ = \ \partial\Delta[n]\cap A \]
and refer to this object as the $n$-th {\em latching object}  of 
a co-$\Delta$-object $A$. The simplex category~$s(\partial\Delta[n])$
is the full subcategory of the over category $\Delta\downarrow[n]$ 
with objects all $\alpha:[i]\to[n]$ for~$i<n$.
So $\partial^n\!A$, if it exists, is a colimit 
of the functor $s(\partial\Delta[n])\to\cC$ that sends
$\alpha:[i]\to[n]$ to~$A^i$.
For example, $\partial^0\!A$ is an initial object and 
$\partial^1\!A$ is a coproduct of two copies of $A^0$.
\end{eg}

The pairing $(K,A)\mapsto K\cap A$ extends to a functor
in two variables in a rather formal way, whenever the representing objects
exist. Indeed, if $\lambda:K\to L$ is a morphism of $\Delta$-sets
and $f:A\to B$ a morphism of co-$\Delta$-objects, then
precomposition with $\lambda$ and $f$ is a natural transformation
of set-valued functors on $\cC$
$$\Delta\text{-set}(\lambda,\cC(f,-)) \ : \
 \Delta\text{-set}(L,\cC(B,-))\ \to \ \Delta\text{-set}(K,\cC(A,-)) \ . $$
If $L\cap B$ respectively $K\cap A$ represent these two functors,
then the Yoneda lemma provides a unique $\cC$-morphism 
$\lambda\cap f:K\cap A\to L\cap B$ that represents the transformation
$\Delta\text{-set}(\lambda,\cC(f,-))$.

\begin{defn}\label{def-Reedy} 
A co-$\Delta$-object $A$ in  a cofibration category $\cC$ 
is {\em cofibrant} if for every $n\geq 0$
the latching object $\partial^n\!A$ exists and 
the canonical morphism $\nu^n:\partial^n\!A\to A^n$ is a cofibration.
A morphism $f:A\to B$ of cofibrant co-$\Delta$-objects
is a {\em cofibration} if for every $n\geq 0$ the morphism 
$$ f^n\cup\nu^n : \  A^n\cup_{\partial^n\!A}\partial^n\!B \ \to \ B^n$$
is a cofibration.
A morphism $f:A\to B$ of co-$\Delta$-objects is a  {\em level equivalence}
if $f^n:A^n\to B^n$ is a weak equivalence for all $n\geq 0$.
\end{defn}

In the definition of cofibrations,
$A^n\cup_{\partial^nA}\partial^nB$ denotes a pushout of the diagram
$$ A^n \ \xla{\ \nu^n \ }\ \partial^n\!A \ \xra{\ \partial^n\!f\ } \ \partial^n\!B \ ,$$
which exists since $\nu^n$ is a cofibration.

\begin{prop}\label{prop-represent}
Let $K$ be a finite $\Delta$-set and  $\cC$ a cofibration category.
Then for every cofibrant co-$\Delta$-object $A$ in $\cC$
the object $K\cap A$ exists in $\cC$. Moreover, the functor
$K\cap-$ takes cofibrations between cofibrant co-$\Delta$-objects
to cofibrations in $\cC$, and it takes acyclic cofibrations 
between cofibrant co-$\Delta$-objects 
to acyclic cofibrations in $\cC$.
\end{prop}
\begin{proof} 
We argue by induction over the dimension of $K$.
If $K$ is empty (i.e., $(-1)$-dimensional), 
then any initial object of $\cC$ represents the functor 
$\Delta\text{-sets}(K,\cC(A, -))$, and can thus be taken as $\emptyset\cap A$.
Hence for every morphism $f:A\to B$ of co-$\Delta$-objects,
the morphism $\emptyset\cap f$ is an isomorphism, thus an acyclic cofibration.

Now suppose that $n\geq 0$ and we have established the proposition 
for all finite $\Delta$-sets of dimension less than~$n$.
We claim first that for every morphism $f:A\to B$ 
between cofibrant co-$\Delta$-objects that is a 
cofibration and a level equivalence, the morphism
$f^n\cup\nu^n:A^n\cup_{\partial^n\!A}\partial^n\!B\to B^n$ 
is an acyclic cofibration.
Since $\partial^n\!A=\partial\Delta[n]\cap A$ and 
$\partial\Delta[n]$ has dimension $n-1$, 
the morphism $\partial^n\!f:\partial^n\!A\to\partial^n\!B$
is an acyclic cofibration by induction. 
So its cobase change
$A^n\to A^n\cup_{\partial^n\!A}\partial^n\!B$ is an acyclic cofibration. 
Since $f^n:A^n\to B^n$ is a weak equivalence,
the morphism $f^n\cup\nu^n:A^n\cup_{\partial^n\!A}\partial^n\!B\to B^n$ is 
a weak equivalence, hence an acyclic cofibration.

Now suppose that $K$ is $n$-dimensional.
We can write $K$ as a pushout
$$\xymatrix{ K_n\times \partial\Delta[n] \ar[r]\ar[d] & 
K_n\times \Delta[n] \ar[d]\\
K' \ar[r] & K}$$
where $K'$ is the $(n-1)$-skeleton of $K$.
By induction there is a representing object $K'\cap A$
for the functor $\Delta\text{-sets}(K',\cC(A, -))$.
The latching object $\partial^n\!A$ represents the functor 
$\Delta\text{-sets}(\partial\Delta[n],\cC(A, -))$
and $A^n$  represents the functor 
$\Delta\text{-sets}(\Delta[n],\cC(A, -))$.
Moreover, the morphism $\nu^n:\partial^n\!A\to A^n$ is a cofibration 
since $A$ is cofibrant, and hence so is a finite coproduct of copies of
$\nu^n$.
So any pushout in $\cC$
$$\xymatrix@C=15mm{ K_n\times\partial^n\!A \ar[r]^-{K_n\times \nu^n}\ar[d] &
K_n\times A^n \ar[d]\\
K'\cap A \ar[r] & K\cap A}$$
can serve as the object $K\cap A$.
Here, and in the rest of the proof, we write $K_n\times X$ 
for a coproduct, indexed by the finite set $K_n$, of copies of an object $X$.
 
Now we let $f:A\to B$ be a cofibration between cofibrant co-$\Delta$-objects.
The morphism $K\cap f:K\cap A\to K\cap B$ is obtained by passage to 
horizontal pushouts from the commutative diagram:
$$\xymatrix@C=18mm{ K_n\times A^n\ar[d]_{K_n\times f^n} & 
K_n\times\partial^n\!A\ar[r]\ar[l]_-{K_n\times\nu^n}
\ar[d]^{K_n\times \partial^n\!f} & K'\cap A \ar[d]^{K'\cap f}\\
K_n\times B^n & K_n\times\partial^n\! B\ar[r]\ar[l]^-{K_n\times \nu^n} & 
K'\cap B}$$
This induced map on pushouts factors as the composite
\begin{align*}
 K_n\times A^n\cup_{K_n\times\partial^n\!A}(K'\cap A) \ 
&\xra{\ \Id\cup (K'\cap f)\ } \
K_n\times A^n\cup_{K_n\times\partial^n\!A}(K'\cap B)\\ 
&\xra{(K_n\times f^n)\cup\Id } \
K_n\times B^n\cup_{K_n\times\partial^n\!B}(K'\cap B)\ .
\end{align*}
The first map is a cobase change of $K'\cap f$, 
which is a cofibration (and a weak equivalence if $f$ 
is also a level equivalence) by induction.
The second map is a cobase change of a finite coproduct of copies of
$f^n\cup\nu^n:A^n\cup_{\partial^n\!A}\partial^n\!B\to B^n$,
which a cofibration by hypothesis on $f$.
If $f$ is also a level equivalence, then $f^n\cup\nu^n$,
and hence the second map, is an acyclic cofibration by the above claim.
So the composite map $K\cap f$ is a cofibration
(respectively acyclic cofibration) as the composite of two cofibrations
(respectively acyclic cofibrations).
\end{proof}

Another purely formal consequence of the definition is that the 
$\cap$-pairing preserves colimits in both variables.
We emphasize that the next proposition does not claim the existence
of any kind of colimits; it only says that if a certain
colimit exists in the category of cofibrant co-$\Delta$-objects
(respectively  the category of  finite $\Delta$-sets), 
then the functor $K\cap-$
(respectively $-\cap A$) preserves the colimit.

\begin{prop}\label{prop-colimits} 
For every finite $\Delta$-set $K$ the functor
$K\cap-$ takes colimits in the category of cofibrant co-$\Delta$-objects
to colimits in $\cC$.
For every cofibrant co-$\Delta$-object $A$ the functor
$-\cap A$ takes colimits in the category of finite $\Delta$-set 
to colimits in $\cC$.
\end{prop}
\begin{proof}
We show the claim for the functor $-\cap A$,
the other case being completely analogous.
Let~$I$ be a small category and $F$ a functor from $I$ to
the category of finite $\Delta\text{-sets}$ that has a colimit $\colim_IF$
(inside the category of finite $\Delta$-sets).
We denote by $F\cap A:I\to\cC$ the functor  given by capping $F$ 
objectwise with $A$.
The representability property of the cap product and 
universal property of a colimit combine into natural isomorphisms 
\begin{align*}
 \cC((\colim_IF)\cap A,-) \ &\iso \ 
\Delta\text{-set}(\colim_IF,\cC(A,-)) \ \iso \ 
{\lim}_I\ \Delta\text{-set}(F,\cC(A,-)) \\ 
&\iso \ {\lim}_I\ \cC(F\cap A,-) \ .
\end{align*}
So $(\colim_IF)\cap A$ has the universal property of a colimit of $F\cap A$.
\end{proof}

\begin{prop}\label{prop-SM7} 
  Let $i:K\to L$ be a monomorphism between finite $\Delta$-sets
  and $j:A\to B$ a cofibration between cofibrant co-$\Delta$-objects in $\cC$.
  \begin{enumerate}[\em (i)]
  \item The pushout of the diagram
    $$L\cap A \ \xla{\ i\cap A\ } \ K\cap A 
    \ \xra{\ K\cap j\ } \ K\cap B$$
    exists in $\cC$ and the pushout product morphism
    $$(L\cap j)\cup(i\cap B)\ : \ 
    (L\cap A) \cup_{(K\cap A)} (K\cap B)  \ \to \  L\cap B$$
    is a cofibration.
  \item If moreover $j$ is a level equivalence, then 
    $(L\cap j)\cup(i\cap B)$ is a weak equivalence in $\cC$.
  \end{enumerate}
\end{prop}
\begin{proof}
(i) The morphism $K\cap j$ is 
a cofibration by Proposition~\ref{prop-represent},
so the pushout exists by axiom~(C3). 
For the proof that $(L\cap j)\cup(i\cap B)$ is a cofibration 
we argue by induction on the number of simplices of $L$ that are not in the
image of $i$. If $i$ is bijective, then 
$(L\cap j)\cup(i\cap B)$ is an isomorphism. Otherwise we 
choose a sub-$\Delta$-set~$L'$ of $L$ such that $i(K)\subseteq L'$ and
such that~$L'$ has one simplex less than $L$. The morphism
$(L\cap j)\cup(i\cap B)$ then factors as the composite
$$ (L\cap A) \cup_{(K\cap A)} (K\cap B) \ \xra{\ (L\cap A)\cup(i\cap B)} \ 
(L\cap A) \cup_{(L'\cap A)} (L'\cap B)
\xra{\ (L\cap j)\cup(\text{incl}\cap B)} \  L\cap B\ . $$
The first map is a cobase change of
$$ (L'\cap j)\cup(i\cap B)\ : \ 
(L'\cap A) \cup_{(K\cap A)} (K\cap B)
\ \to \  L'\cap B\ ,$$
which is a cofibration by induction. The second map is
a cobase change of the cofibration
$f^n\cup\nu^n:A^n\cup_{\partial^n\!A}\partial^n\!B\to B^n$,
where $n$ is the dimension of the simplex not in $L'$.
So $(L\cap j)\cup(i\cap B)$ is the composite of two cofibrations,
hence a cofibration itself.

(ii) The morphism $K\cap j:K\cap A\to K\cap B$
is an acyclic cofibration by Proposition~\ref{prop-represent}. 
Hence its cobase change $\psi:L\cap A\to (L\cap A) \cup_{(K\cap A)} (K\cap B)$ 
is an acyclic cofibration. Also by Proposition~\ref{prop-represent}, 
the morphism $L\cap j:L\cap A\to L\cap B$
is an acyclic cofibration. Since 
$((L\cap j)\cup(i\cap B))\circ\psi=L\cap j$,
the pushout product map is a weak equivalence by the 2-out-of-3 property.
\end{proof}

\begin{defn}
A {\em frame} in  a cofibration category is a cofibrant 
co-$\Delta$-object $A$ that is homotopically constant, i.e.,
for every morphism $\alpha:[n]\to[m]$ in the category $\Delta$ 
the morphism $\alpha_*:A^n\to A^m$ is a weak equivalence.
\end{defn}

A cofibration category $\cC$ is {\em saturated} if every morphism that
becomes an isomorphism under the localization functor $\gamma:\cC\to\bHo(\cC)$
is already a weak equivalence. Saturation is no serious restriction since
the weak equivalences in any cofibration category can  be saturated
without changing the cofibrations or the homotopy category 
(see~\cite[Prop.\,3.16]{cisinski-categories derivables}).
If $A$ is a frame, then the functor $-\cap A$ tries hard 
to turn weak equivalences of finite $\Delta$-sets
into weak equivalences in $\cC$. This does not work in complete
generality, but the next proposition shows, among other things,
that $-\cap A$ has this property in all saturated cofibration categories.

\begin{prop}\label{prop-cap with frame}
Let $A$ be a frame in a cofibration category $\cC$.
\begin{enumerate}[\em (i)]
\item For every elementary expansion
$K\subset L$ the morphism $K\cap A\to L\cap A$
induced by the inclusion is an acyclic cofibration.
\item For every weak equivalence $f$ between finite $\Delta$-sets,
the induced morphism $f\cap A$ becomes an isomorphism in $\bHo(\cC)$. 
\item If the cofibration category $\cC$ is saturated, then
the functor $-\cap A$ takes weak equivalences between
finite $\Delta$-sets to weak equivalences in $\cC$.
\end{enumerate}
\end{prop}
\begin{proof}
(i) We proceed by induction on the dimension $n$ of the elementary expansion.
For $n=1$ the $\Delta$-set $\Lambda^i[1]$ is isomorphic to $\Delta[0]$
and the inclusion $\Lambda^i[0]\to \Delta[1]$ corresponds to
the face map $d_{1-i}:\Delta[0]\to\Delta[1]$.
So the map in question is isomorphic to $(d_{1-i})_*:A^0\to A^1$,
hence an acyclic cofibration.

Now we suppose that $n\geq 2$. We start with the special (but universal)
cases of the horn inclusions $\Lambda^i[n]\to\Delta[n]$ for $0\leq i\leq n$.
We let $v:[0]\to[n]$ be the map with $v(0)=0$.
The inclusion $\{v\}\to\Lambda^i[n]$ is the composite of a sequence 
of elementary expansions of dimensions strictly less than $n$.
So the induced morphism 
$\{v\}\cap A\to\Lambda^i[n]\cap A$ is a weak equivalence by induction. 
The composite of this weak equivalence with the morphism
$\Lambda^i[n]\cap A\to\Delta[n]\cap A$
is isomorphic to the structure morphism $v_*:A^0\to A^n$ and hence 
a weak equivalence. So the map $\Lambda^i[n]\cap A\to\Delta[n]\cap A$ 
is also a weak equivalence. This map is also a cofibration by
Proposition~\ref{prop-represent}, hence an acyclic cofibration.
If $K\subset L$ is a general elementary expansion of dimension~$n$, then the
pushout square~\eqref{eq:expansion_pushout} caps with $A$ to a pushout
square in $\cC$:
\[ \xymatrix{ \Lambda^i[n]\cap A \ar[r]\ar[d] & \Delta[n]\cap A\ar[d]\\
K\cap A\ar[r] &L\cap A} \]
As a cobase change of an acyclic cofibration, 
the morphism $K\cap A\to L\cap A$ is itself an acyclic cofibration. 

(ii) By part~(i) the functor  $\gamma(-\cap A)$
from the category of finite $\Delta$-sets to $\bHo(\cC)$
takes elementary expansions to isomorphisms.
So it takes weak equivalences to isomorphisms by 
Proposition~\ref{prop-expansions generate equivalences}.

Claim~(iii) is a combination of part~(ii) and the saturation property.
\end{proof}

The following proposition is a special case of
Theorem 9.2.4 (1a) of~\cite{radulescu-ABC} where the indexing
category is the direct category $\Delta$.
More precisely, part~(i) is Radulescu-Banu's axiom (CF4) 
saying that every morphism $f:A\to B$ from a cofibrant object 
can be factored as $f=pi$ with~$i$ a cofibration 
and $p$ a pointwise weak equivalence.

\begin{prop}\label{prop-co-Delta factorization}
  Let $\cC$ be a cofibration category.
  \begin{enumerate}[\em (i)]
  \item Let $f:A\to Z$ be a morphism of co-$\Delta$-objects 
    in $\cC$ such that $A$ is cofibrant.
    Then there exists a cofibrant co-$\Delta$-object $B$ and
    a factorization $f=pi$ as a cofibration 
    $i:A\to B$ followed by a level equivalence $p: B\to Z$.
  \item The cofibrations and level equivalences make the category
    of cofibrant co-$\Delta$-objects into a cofibration category.
  \end{enumerate}
\end{prop}

The next proposition is the key step in the proof that the homotopy category
of frames in $\cC$ is equivalent to the homotopy category of $\cC$,
see Theorem~\ref{theorem-framing} below.

\begin{prop}\label{prop-App2 for frames} 
Let $A$ be a frame in a cofibration category $\cC$, $Z$ an object of $\cC$
and $\varphi:A^0\to Z$ a $\cC$-morphism.
Then there is a homotopically constant co-$\Delta$-object $Y$, 
a morphism $y:A\to Y$ of co-$\Delta$-objects 
and a weak equivalence
$p:Z\to Y^0$ such that $y^0=p\varphi$.
\end{prop}
\begin{proof}
We denote by $P(n)$ the $n$-dimensional $\Delta$-set with a unique
$i$-simplex $x_i$ for $i=0,\dots,n$ and we let $C(n)$ be the cone
(see Example~\ref{eg-cone}) of $P(n)$.
We start by choosing a pushout:
$$\xymatrix{ A^0 \ar[r]^-{\varphi} \ar[d]_{(x_0)_*}^\sim & Z \ar[d]^p_\sim\\
C(0)\cap A \ar[r]_-{\lambda_0} & Y^0}$$
The inclusion $\{x_0\}\to C(0)$ is an elementary expansion,
so $(x_0)_*:A^0\iso\{x_0\}\cap A\to C(0)\cap A$ is an acyclic cofibration
by Proposition~\ref{prop-cap with frame}~(i).
Then we proceed by induction on~$n$ and define~$Y^n$ as a pushout:
$$\xymatrix@C=15mm{ 
C(n-1)\cap A \ar[r]^-{\lambda_{n-1}} \ar[d]_\sim & Y^{n-1} \ar[d]^{i_n}\\
C(n)\cap A \ar[r]_-{\lambda_n} & Y^n}$$
The inclusion $C(n-1)\to C(n)$ is an elementary expansion, so
the left vertical map is an acyclic cofibration, again 
by Proposition~\ref{prop-cap with frame}~(i).
For $0\leq i\leq n$, we define the structure map $d_j:Y^{n-1}\to Y^n$
as the cobase change $i_n$ of $C(n-1)\cap A\to C(n)\cap A$.
So $d_j$ is independent of $j$ and an acyclic cofibration.
The co-$\Delta$-object $Y$ is thus homotopically constant
(but usually not cofibrant).
Because $x_{n+1}d_j=x_n$ the composite maps 
$$ A^n\ \xra{\ (x_n)_*\ } \  C(n)\cap A\ \xra{\ \lambda_n\ }
\ Y^n$$
constitute a morphism of co-$\Delta$-objects $y:A\to Y$.
Moreover, we have $y^0=\lambda_0\circ (x_0)_*=p\varphi$.
\end{proof}

Now we can prove the main result of this section.

\begin{theorem}\label{theorem-framing} Let $\cC$ be a cofibration category. 
Then the category $f\cC$ of frames in $\cC$ forms a cofibration category
with respect to cofibrations and level equivalences. 
The functor $f\cC\to \cC$ that evaluates in dimension~0
is exact and its derived functor is an equivalence of categories
from $\bHo(f\cC)$ to $\bHo(\cC)$. 
\end{theorem}
\begin{proof}
Frames are closed under level equivalences
within the category of cofibrant co-$\Delta$-objects.
Also, the initial co-$\Delta$-object is a frame, and
for every pushout square~\eqref{eq-C3pushout}
in which $A, B$ and $C$ are frames, the pushout 
again homotopically constant by the gluing lemma.
So the class of frames is also closed under pushouts along cofibrations,
and hence form a cofibration category. 

In order to show that the derived functor 
of evaluation in dimension~0
is an equivalence of categories we use the criterion
given by the 
`approximation theorem'~\cite[Theorem~3.12]{cisinski-categories derivables}.
The necessary hypotheses are satisfied: 
a morphism of $f:A\to B$ 
of homotopically constant co-$\Delta$-objects is a level equivalence if and
only if $f^0:A^0\to B^0$ is a weak equivalence, i.e., the evaluation functor
satisfies the approximation property (AP1) 
of~\cite[3.6]{cisinski-categories derivables}. 
The second condition~(AP2) demands that for every frame $A$,
every $\cC$-object $Z$ and every morphism $\varphi:A^0\to Z$
there should be a frame $B$, a cofibration $i:A\to B$ 
and weak equivalences
$\psi:B^0\to Y^0$ and $p:Z\to Y^0$ such that $\psi i^0=p\varphi$.
Indeed, Proposition~\ref{prop-App2 for frames} provides
a morphism $y:A\to Y$ that we can factor,
using Proposition~\ref{prop-co-Delta factorization}~(i),
as $y=qi$ for a cofibration $i:A\to B$
followed by a level equivalence $q:B\to Y$.
Then $B$ is cofibrant (since $A$ is cofibrant and $i$ is a cofibration)
and homotopically constant  (since $Y$ is homotopically constant 
and $q$ is a level equivalence). So $B$ is a frame.
Moreover, the morphism $\psi=q^0$ is a weak equivalence and satisfies
\[ \psi i^0 \ = \ q^0i^0 \ = \ y^0 \ = \ p\varphi \ .\qedhere\]
\end{proof}

\begin{rk}\label{rk-Ho CW action} 
As a combination of the previous results 
we have effectively constructed a natural
action of the homotopy category of finite CW-complexes on
the homotopy category of any cofibration category.
In more detail, the composite functor
$$ \Delta\text{-sets}^\text{fin} \times f\cC 
\ \xra{\ \cap\ }\ \cC \ \xra{\ \gamma\ }\ \bHo(\cC) $$
takes weak equivalences of finite $\Delta$-sets and
level equivalences of frames to isomorphisms
(by Proposition~\ref{prop-cap with frame}~(ii)
and Proposition~\ref{prop-SM7}~(ii)).
So the functor factors over the localization of the left hand side
through a unique functor
\begin{equation}\label{eq:cap^L}
 \cap^L \ : \ \bHo(\Delta\text{-sets}^\text{fin}) \times \bHo(f\cC) 
\ \to \ \bHo(\cC) \ .\end{equation}
Here we denote by $\bHo(\Delta\text{-sets}^\text{fin})$
a localization of the category of finite $\Delta$-sets
at the class of weak equivalences.
One should beware though that finite $\Delta$-sets do {\em not} form a
cofibration category (the factorization axiom (C4) fails),
but such a localization can be constructed `by hand',
for example by setting
$$\bHo(\Delta\text{-sets}^\text{fin})(K,L) \ = \ [|K|,|L|]\ ,$$
the set of homotopy classes of continuous maps between
the geometric realizations of $K$ and~$L$.
In fact, the homotopy category $\bHo(\Delta\text{-sets}^\text{fin})$
is equivalent, via the functor of geometric realization,
to the homotopy category (in the traditional sense) of finite CW-complexes,
see for example~\cite[Ch.\,I, Thm.\,4.3]{buonchristiano-rourke-sanderson}.

The functor~\eqref{eq:cap^L} is not quite an action of
$\bHo(\Delta\text{-sets}^\text{fin})$ on $\bHo(\cC)$ yet,
but we can fix that by choosing a `framing' of $\cC$, i.e., an inverse
$\cF:\bHo(\cC)\to\bHo(f\cC)$ to the equivalence 
of Theorem~\ref{theorem-framing}. One can then show that
the composite functor
$$  \bHo(\Delta\text{-sets}^\text{fin}) \times \bHo(\cC) 
\ \xra{\ \Id\times \cF\ }\ \bHo(\Delta\text{-sets}^\text{fin}) \times \bHo(f\cC) 
\ \xra{\ \cap^L\ } \ \bHo(\cC) $$
is coherently associative and unital with respect to the 
(derived) geometric product of $\Delta$-sets, 
and it is natural for exact functors of cofibration categories.
We shall not elaborate on this point since we don't need it 
in the present paper.
\end{rk}

There is now a standard way of extending the $\cap$-pairing
to a pairing $(K,A)\mapsto K\tensor A$
that takes a finite $\Delta$-set $K$ and a co-$\Delta$-object $A$
to another co-$\Delta$-object $K\tensor A$.
In dimension~$n$ we set
\begin{equation}\label{eq:define_tensor}
   (K\tensor A)^n \ = \ (\Delta[n]\tensor K)\cap A \ ,
\end{equation}
where the tensor symbol on the right hand side is the geometric product
of $\Delta$-sets. The structure maps arise via the functoriality in $\Delta[-]$.

The tensor product construction comes with
coherent natural isomorphisms
\begin{equation}\label{eq:cap_tensor_iso}
    K\cap (L \tensor A) \ \iso \  (K\tensor L) \cap A   \text{\qquad and\qquad}
 K\tensor (L \tensor A) \ \iso \ (K\tensor L) \tensor A   
\end{equation}
in the category $\cC$ respectively 
the category of co-$\Delta$-objects in $\cC$.
Indeed, given a simplex  $x\in K_i$ we let $\bar x:\Delta[i]\to K$ be the
morphism with $\bar x(\Id_{[i]})=x$. As $i$ and $x$ vary, the morphisms
\[ x_*= (\bar x\tensor L)\cap A\ : \ 
(L\tensor A)^i = (\Delta[i]\tensor L)\cap A 
\ \to \ (K\tensor L) \cap A \]
are compatible, so the universal property of $K\cap(L\tensor A)$
provides a morphism
\[ \psi\ : \ K\cap (L\tensor A) \ \to \ (K\tensor L) \cap A \ .\]
In the other direction we consider a triple  $(x,y;\,\varphi)$
with $x\in K_i$, $y\in L_j$ and $\varphi:[n]\to[i]\times[j]$
a monotone injection. Then $[\Id_{[i]},y;\,\varphi]$
is an $n$-simplex of $\Delta[i]\tensor L$, so we can form the composite 
\[  A^n \ \xra{[\Id_{[i]},y;\,\varphi]_*} \ 
(\Delta[i]\tensor L)\cap A = (L\tensor A)^i\ \xra{\ x_*\ }\
K\cap(L\tensor A) \ ,\]
which in fact only depends on the equivalence class 
$[x,y;\,\varphi]$ in $(K\tensor L)_n$. As the class $[x,y;\,\varphi]$ varies,
the maps are again compatible, so the universal property of 
$(K\tensor L)\cap A$ provides a morphism
\[ \bar\psi\ : \ (K\tensor L) \cap A  \ \to \  K\cap (L\tensor A)\ .\]
The morphisms $\psi$ and $\bar\psi$ are natural in all three variables
and are mutually inverse isomorphisms.
The second isomorphism in~\eqref{eq:cap_tensor_iso} 
is then given in dimension~$n$ by the composite
\begin{align*}
 (K\tensor (L \tensor A))^n \ &= \ 
 (\Delta[n]\tensor K)\cap(L\tensor A) \ \iso \
 ((\Delta[n]\tensor K)\tensor L)\cap A \\ 
&\iso \
  (\Delta[n]\tensor(K\tensor L)) \cap  A \ = \
  ((K\tensor L) \tensor A)^n \ 
\end{align*}
that combines the isomorphism $\psi$ with the associativity isomorphism
$(\Delta[n]\tensor K)\tensor L\iso\Delta[n]\tensor(K\tensor L)$
of the geometric product of $\Delta$-sets.

\begin{defn} \label{def-Delta cofibration category}
A {\em $\Delta$-cofibration category} is a cofibration category $\cC$ 
equipped with a pairing
$$ \Delta\text{-sets}^\text{fin}\times \cC \ \to \ \cC \ ,
\quad (K,X) \ \longmapsto \ K\tensor X$$
that is coherently associative and unital with respect to the
geometric product of $\Delta$-sets and
satisfies the following  properties:
\begin{itemize}
\item For every finite $\Delta$-set $K$ the functor
$K\tensor-:\cC\to\cC$ is exact.
\item For every object $A$ of $\cC$ the functor
$-\tensor A:\Delta\text{-sets}^\text{fin}\to\cC$ is exact.
\item Let $i:K\to L$ be a monomorphism of finite $\Delta$-sets
and $j:A\to B$ a cofibration in $\cC$. Then the pushout product
morphism
$$  (i\tensor B)\cup(L\tensor j) 
\ : \ K\tensor B\cup_{K\tensor A}L\tensor A\ \to \ L\tensor B$$
is a cofibration.
\end{itemize}
\end{defn}

\begin{rk}\label{rk-PPP weak equivalence} 
Suppose that $\cC$ is a $\Delta$-cofibration category,
$i:K\to L$ a monomorphism and $j:A\to B$ a cofibration 
as in the pushout product property.
If in addition $i$ or $j$ is a weak equivalence, then the morphism
$$  (i\tensor B)\cup(L\tensor j) 
\ : \ K\tensor B\cup_{K\tensor A}L\tensor A\ \to \ L\tensor B$$
is also a weak equivalence. Indeed, if $i$ is also weak equivalence,
then so is $i\tensor A$ (since $-\tensor A$ is exact), and similarly
for $i\tensor B$. So $i\tensor A:K\tensor A\to L\tensor A$ is
an acyclic cofibration, hence so is its cobase change
$K\tensor B\to K\tensor B\cup_{K\tensor A}L\tensor A$.
The composite of this weak equivalence with 
$(i\tensor B)\cup(L\tensor j)$ is the weak equivalence $i\tensor B$,
so $(i\tensor B)\cup(L\tensor j)$ is a weak equivalence, as claimed. 
The argument is analogous when $j$ is a weak equivalence.
\end{rk}

\begin{theorem}\label{thm-frames are Delta-enriched} 
Let $\cC$ be a saturated cofibration category.
Then the $\tensor$-product~\eqref{eq:define_tensor} makes the 
category of frames $f\cC$  into a $\Delta$-cofibration category.
\end{theorem}
\begin{proof}
We start by showing that for every finite $\Delta$-set $K$
and every frame $A$ the co-$\Delta$-object $K\tensor A$
is again a frame. 
Since $A$ is cofibrant, the object $(\partial\Delta[n]\tensor K)\cap A$
exists; by~\eqref{eq:cap_tensor_iso}, this object is isomorphic to
\[ \partial\Delta[n]\cap (K\tensor A)\ = \  
\partial^n(K\tensor A) \ , \]
so the $n$-th latching object of $K\tensor A$ exists.
The latching morphism $\partial^n(K\tensor A)\to (K\tensor A)^n$
is induced by the inclusion $\partial\Delta[n]\to\Delta[n]$, and thus
isomorphic to $(\text{incl}\tensor K)\cap A:(\partial\Delta[n]\tensor K)\cap A\to(\Delta[n]\tensor K)\cap A$;
this is a cofibration by Proposition~\ref{prop-SM7}~(i).
In other words, if~$A$ is cofibrant, then so is $K\tensor A$.  

If $A$ is a frame, then $K\tensor A$ is also homotopically constant:
for every morphism $\alpha:[n]\to [m]$ in the category $\Delta$
the induced morphisms $\alpha_*:\Delta[n]\to\Delta[m]$ and
$\alpha_*\tensor K:\Delta[n]\tensor K\to \Delta[m]\tensor K$
are sequences of elementary expansions,
so by Proposition~\ref{prop-cap with frame}~(i) the induced map
$$ (K\tensor A)^n \ = \   (\Delta[n]\tensor K)\cap A
\ \to \ (\Delta[m]\tensor K)\cap A \ = \ (K\tensor A)^m$$
is an acyclic cofibration. 

The $\tensor$-pairing preserves all existing colimits in both variables,
in particular initial objects and pushouts along cofibrations.
Indeed, colimits in functor categories are objectwise, so we must show that
for every $n\geq 0$ the functor 
$(K,A)\mapsto(\Delta[n]\tensor K)\cap A$
takes all colimits in $K$ and $A$ to colimits in $\cC$. 
Since the geometric product $\Delta[n]\tensor -$
preserves colimits of finite $\Delta$-sets, this follows from
the fact that the $\cap$-pairing preserves colimits in both variables,
see Proposition~\ref{prop-colimits}.

For the pushout product property we consider an inclusion
$i:K\to L$ of finite $\Delta$-sets and a cofibration 
$j:A \to B$ between frames in $\cC$. 
To shorten the notation we write
\[ \td{i,j} = (i\tensor B)\cup(L\tensor j) \ : \ 
K\tensor B\cup_{K\tensor A}L\tensor A \ \to \ L\tensor B \]
for a pushout product map. We then have to show that $\td{i,j}$ is
a cofibration of co-$\Delta$-objects, and that in turn means showing
that the morphism
$$  \td{i,j}^n\cup\nu^n \ : \
(K\tensor B\cup_{K\tensor A}L\tensor A)^n \cup_{\partial^n(K\tensor B\cup_{K\tensor A}L\tensor A)}
\partial^n(L\tensor B) \ \to \ (L\tensor B)^n$$
is a cofibration in $\cC$.
The isomorphism~\eqref{eq:cap_tensor_iso} and a rearranging of pushouts
translates this into the claim that the morphism
\begin{align*}
 \td{k,j}\colon
(\Delta[n]\tensor K\cup\partial\Delta[n]\tensor L)\cap B
\ \cup_{(\Delta[n]\tensor K\cup\partial\Delta[n]\tensor L)\cap A} &(\Delta[n]\tensor L)\cap A \ 
\to \  (\Delta[n]\tensor L)\cap B   
\end{align*}
is a cofibration in $\cC$,  where 
$$ k\ :\ 
\Delta[n]\tensor K\,\cup\,\partial\Delta[n]\tensor L\
\to\  \Delta[n]\tensor L  $$
is the inclusion (the union in the source is along the
intersection $\partial\Delta[n]\tensor K$).
So $\td{k,j}$ is a cofibration in $\cC$ 
by Proposition~\ref{prop-SM7}~(i). This completes the proof of the
pushout product property.

In the special case when $K$ is the empty $\Delta$-set the pushout
product property shows that the functor $L\tensor-$ preserves cofibrations.
Since $(L\tensor A)^n=(\Delta[n]\tensor L)\cap A$, 
Proposition~\ref{prop-SM7}~(ii) 
shows that $L\tensor -$ preserves level equivalences.

In the special case when $A$ is the initial object the pushout
product property shows that the functor $-\tensor B$ takes monomorphisms 
to cofibrations.
If $f:K\to L$ is a weak equivalence, then so is 
$\Delta[n]\tensor f:\Delta[n]\tensor K\to \Delta[n]\tensor L$.
So the morphism $(f\tensor B)^n=(\Delta[n]\tensor f)\cap B$ 
becomes an isomorphism in $\bHo(\cC)$ 
by Proposition~\ref{prop-cap with frame}~(ii).
Since $\cC$ is saturated, $(f\tensor B)^n$ is a weak equivalence for
every $n\geq 0$, i.e., $-\tensor B$ takes weak equivalences 
to level equivalences. 
\end{proof}

\begin{rk}
If $\cC$ has arbitrary coproducts and colimits of sequences of cofibrations, 
then the functor $\Delta\text{-sets}(K,\cC(A,-))$
is representable for every $\Delta$-set $K$.
If coproducts and sequential colimits are suitably compatible
with cofibrations and weak equivalences, then most of the result
of this section carry over from finite to arbitrary $\Delta$-sets.
\end{rk}

\section{\texorpdfstring{Pointed $\Delta$-cofibration categories}
{Pointed Delta-cofibration categories}}
\label{sec-pointed Delta}

In this section we introduce and study actions of based $\Delta$-sets 
on objects of a pointed $\Delta$-cofibration category.
A {\em based $\Delta$-set} is a contravariant functor
from the category $\Delta$ to the category of based sets.
So a based $\Delta$-set is a $\Delta$-set equipped with 
a distinguished basepoint in every dimension, preserved under the face maps. 
One should beware that, in contrast to the world
of simplicial sets, the 0-simplex $\Delta[0]$ is {\em not} a terminal
object in the category of $\Delta$-sets. 
So specifying a vertex in a $\Delta$-set does {\em not} 
determine a morphism from the terminal $\Delta$-set (which has
exactly one simplex in every dimension), and does {\em not}
make a $\Delta$-set based.
A based $\Delta$-set is non-empty in every dimension,
so it is never finite dimensional, and never finite.

Now we discuss how to pair objects in a pointed 
$\Delta$-cofibration category $\cC$ 
with based $\Delta$-sets.
If $X$ and $Z$ are objects of $\cC$, we can define a $\Delta$-set
$\map(X,Z)$ as  the composite functor
$$ \Delta^\text{op} \ \xra{\ [n]\mapsto\Delta[n]\ } \ 
(\Delta\text{-sets})^\text{op} \ 
\xra{\ \cC(-\tensor X,Z)\ } \ \text{(sets)}\ . $$
Since $\cC$ is pointed, this $\Delta$-set is canonically based:
the basepoint of $\map(X,Z)_n$ is the zero map from $\Delta[n]\tensor X$ to $Z$.
If $K$ is a based $\Delta$-set, then we denote by $K\sm X$
a representing object, if it exists, of the functor
$$ \cC \ \to \ \text{(sets)}\ , \quad Z \longmapsto\ 
\Delta\text{-set}_*(K,\map(X,Z)) \ ,$$
where the right hand side is the set of morphisms of
based $\Delta$-sets from $K$ to $\map(X,Z)$.

The smash product $K\sm X$ plays the role of a quotient of
$K\tensor X$ by $\ast\tensor X$ where $\ast\subseteq K$ 
denotes the sub-$\Delta$-set consisting only of the various basepoints
in all dimensions. 
However, this is not literally true, because $K$ and $\ast$ 
are infinite $\Delta$-sets, and so the expressions
$K\tensor X$ and $\ast\tensor X$ do not individually make sense
in a $\Delta$-cofibration category. The following concepts of
`essentially finite' based $\Delta$-sets and `finite approximation'
allow us to deal with the fact that based $\Delta$-sets are never finite.

\begin{defn} A based $\Delta$-set is {\em essentially finite} if it
has only finitely many non-basepoint simplices.
A {\em finite presentation} of a based $\Delta$-set $K$
is a morphism $r:R\to K$ of (unbased) $\Delta$-sets such that $R$ is finite
(in the absolute sense) and every non-basepoint simplex of $K$ has exactly one
preimage under~$r$.
\end{defn}

In other words, a finite presentation is an 
`isomorphism away from the basepoints'.
So if a based $\Delta$-set has a finite presentation, then it must be
essentially finite. 
A based $\Delta$-set $K$ is essentially finite if and only if
each $K_n$ is finite  and almost all~$K_n$ consist only of the basepoint.
So if $m$ is the maximum of the dimensions of the non-basepoint simplices,
then the inclusion $K^{(m)}\to K$ of the $m$-skeleton is a finite presentation.
(This particular finite presentation is also injective, but that is not
required of finite presentations in general.) 
We conclude that a based $\Delta$-set has a finite presentation 
if and only if it is essentially finite.

\begin{prop}\label{prop-finite presentation}
Let $r:R\to K$ be a finite presentation of a based $\Delta$-set
and denote by~$r^{-1}(\ast)$ the sub-$\Delta$-set of $R$ 
consisting of all preimages of the respective basepoints. 
Then for every object~$X$ of a pointed $\Delta$-cofibration category
every cokernel of the morphism 
\[  \iota\tensor X\ :\ r^{-1}(\ast)\tensor X\ \to\ R\tensor X\]
is an object $K\sm X$, where $\iota:r^{-1}(\ast)\to R$ is the inclusion.
In particular, the object $K\sm X$ exists 
for every essentially finite based $\Delta$-set~$K$.
\end{prop}
\begin{proof} We let $Y$ be another based $\Delta$-set.
Then precomposition with $r$ induces a bijection from the set
$\Delta\text{-set}_*(K,Y)$ to the subset of $\Delta\text{-set}(R,Y)$
consisting those morphisms $\varphi:R\to Y$ 
that send the entire sub-$\Delta$-set $r^{-1}(\ast)$ to the basepoints. 
In the special case of the based $\Delta$-set
$Y=\map(X,Z)$ this shows that $\Delta\text{-set}_*(K,\map(X,Z))$
is the fiber, over the constant morphism with zero map values, 
of the restriction map
\[ \Delta\text{-set}(R,\map(X,Z))\ \to \ 
\Delta\text{-set}(r^{-1}(\ast),\map(X,Z)) \ .\]
This restriction map is isomorphic to
\[  \cC(\iota\tensor X,Z)\ : \ \cC(R\tensor X,Z)\ \to \ 
\cC(r^{-1}(\ast)\tensor X,Z) \ ,\]
so its fiber is represented by any cokernel 
of the morphism $\iota\tensor X$.
In other words, any cokernel represents the functor 
$\Delta\text{-set}_*(K,\map(X,-))$ and is thus a possible choice of $K\sm X$.
Since every essentially finite based $\Delta$-set has a finite presentation,
the object $K\sm X$ always exists. 
\end{proof}

The same kind of arguments as for the $\cap$-pairing in Section~\ref{sec-framings}
show that the smash product pairing $K\sm X$ 
canonically extends to a functor of two variables,
and that it preserves colimits in each variable.

\begin{prop}\label{prop-smash with equivalence gives weak equivalence}
Let $f:K\to L$ be a weak equivalence between
essentially finite based $\Delta$-sets
and $X$ an object of a pointed $\Delta$-cofibration category.
Then the morphism
$$ f\sm X \ : \  K\sm X \ \to \ L\sm X$$
is a weak equivalence.
\end{prop}
\begin{proof}
We let $m$ be any even number at least as large as the maximum dimension of 
a non-basepoint simplex of $K$.
We claim that the inclusion $K^{(m)}\to K$ of the $m$-skeleton
is a weak equivalence of $\Delta$-sets. Indeed, the geometric realization
of $K^{(m)}$ is the $m$-skeleton of the canonical CW-structure on $|K|$.
Since $K$ consists only of basepoints above dimension $m$,
from $|K^{(m)}|$ to $|K^{(m+2)}|$ two cells of dimension
$m+1$ and $m+2$ are attached; 
since $m$ is even the $(m+2)$-cell is attached to the
$(m+1)$-cell by a map of degree~1.
So for all large enough even $m$ the skeleton inclusion
$|K^{(m)}|\to|K^{(m+2)}|$ is a homotopy equivalence.
Hence the skeleton inclusion
$|K^{(m)}|\to|K|$ is also a homotopy equivalence.

Now we let $m$ be even and at least 
large  as the maximum dimension of 
a non-basepoint simplex of both $K$ and $L$.
Since $f$ is a weak equivalence, so is 
$f^{(m)}:K^{(m)}\to L^{(m)}$, by the previous paragraph.
These skeleta are finite $\Delta$-sets (in the absolute sense). 
We let $P$ denote a $\Delta$-set with a unique simplex in each dimension
up to dimension~$m$, and no simplices above dimension~$m$.
We let $j:P\to K^{(m)}$ be the morphism that hits the basepoints.
In  the commutative diagram
\[ \xymatrix@C015mm{ \ast\ar@{=}[d] & 
P\tensor X \ar@{=}[d] \ar[r]^-{j\tensor X}\ar[l] &
K^{(m)} \tensor X \ar[d]^{f^{(m)}\tensor X}_\sim \\
\ast & P\tensor X \ar[r]_-{f^{(m)}j\tensor X} \ar[l] &
L^{(m)} \tensor X } \]
the right vertical morphism is a weak equivalence since $f^{(m)}$ is.
Since $f\sm X:K\sm X\to L\sm X$ can be obtained from this
diagram by passage to horizontal pushouts 
(by Proposition~\ref{prop-finite presentation}), it is a weak equivalence
by the gluing lemma.
\end{proof}

\begin{eg}\label{eg-smash with S^1}
The smash product with certain based $\Delta$-sets
$I$ respectively $S^1$ provides
functorial cones and suspensions in any pointed $\Delta$-cofibration
category $\cC$.
We let $I$ be the essentially finite based $\Delta$-set with
\[ I_1=\{z,\ast\}\ , \quad 
 I_0=\{zd_0,zd_1\} \]
and with $I_k$ consisting only of the basepoint for
$k\geq 2$. The basepoint in dimension~0 is the vertex $zd_0$.
The morphism $r:\Delta[1]\to I$ that hits the
1-simplex $z$ is a finite presentation with $r^{-1}(\ast)=\{d_0\}$.
Proposition~\ref{prop-finite presentation} provides a pushout square:
\[ \xymatrix@C=15mm{ 
X \ar[r]^-{(d_0)_*} \ar[d] & \Delta[1]\tensor X\ar[d]^q\\
\ast\ar[r] & I\sm X }\]
Since $(d_0)_*:X\to\Delta[1]\tensor X$ is an acyclic cofibration,
the object $I\sm X$ is weakly contractible. 

The object $\partial\Delta[1]\tensor X$ is a coproduct of two copies of $X$,
and the composite 
\[ \partial\Delta[1]\tensor X\ \xra{\text{incl}\tensor X} 
\ \Delta[1]\tensor X \ \xra{\quad q\quad }\ 
I\sm X \]
is zero on the copy of $X$ indexed by $d_0$.
So the commutative square
\[\xymatrix@C=15mm{
\partial\Delta[1]\tensor X\ar[r]^-{\text{incl}\tensor X}\ar[d]_p & 
\Delta[1]\tensor X\ar[d]^q\\
X\ar[r]_-{i_X} & I\sm X }\]
is a pushout, where $i_X=q\circ(d_1)_*$, and where $p$ 
is the morphism such that $p\circ(d_0)_*$ is zero and $p\circ(d_1)_*=\Id_X$.
So $i_X:X\to I\sm X$ is a cofibration with weakly contractible
target, i.e., a functorial cone of $X$.

We let $S^1$ be the based $\Delta$-set with a unique non-basepoint simplex
$z$ of dimension~1. The morphism $r:\Delta[1]\to S^1$ 
that hits $z$ is a finite presentation, 
and it satisfies $r^{-1}(\ast)=\partial\Delta[1]$.
By Proposition~\ref{prop-finite presentation}, the object $S^1\sm X$ is then
a cokernel of the morphism $\partial\Delta[1]\tensor X\to\Delta[1]\tensor X$,
hence also a cokernel of the cone inclusion $i_X:X\to CX$.
So $S^1\sm X$ is isomorphic in $\bHo(\cC)$ to the suspension of $X$.

We define the {\em mapping cone} $Cf$ of a morphism $f:X\to Y$ as a pushout:
\begin{equation}\begin{aligned}\label{eq:mapping_cone}
\xymatrix{X\ar[d]_f 
\ar[r]^-{i_X} & I\sm X  \ar[d]\\
Y \ar[r]_-j & Cf}
\end{aligned}\end{equation}
The pushout exists because the cone inclusion $i_X$ is a cofibration.
We can compare the elementary distinguished triangles of the
two cofibrations $i_X$ and $j$ in $\bHo(\cC)$:
\[ \xymatrix@C=15mm{ 
X \ar[r]^-{\gamma(i_X)} \ar[d]_{\gamma(f)} &  
I\sm X \ar[r] \ar[d] & 
(I\sm X)/X \ar[d]^\iso \ar[r]^-{\delta(i_X)}_-\iso & 
\Sigma X  \ar[d]^{\Sigma\gamma(f)}\\
Y \ar[r]_-{\gamma(j)} & Cf\ar[r]& 
Cf/Y \ar[r]_-{\delta(j)}&  \Sigma Y}    \]
Since the square~\eqref{eq:mapping_cone} is a pushout, 
the induced map $(I\sm X)/X\to Cf/Y$ is an isomorphism in~$\cC$.
The cone $I\sm X$ is weakly contractible, hence the connecting morphism 
$\delta(i_X)$ is an isomorphism in $\bHo(\cC)$.
In the lower distinguished triangle we can thus replace 
$Cf/Y$ by the isomorphic object 
$\Sigma X$ and obtain a distinguished triangle
\[  Y \xra{\ \gamma(j)\ } \ Cf \ \xra{\qquad} \ 
\Sigma X  \ \xra{\ \Sigma\gamma(f)\ } \
 \Sigma Y  \ .  \]
We rotate this triangle to the left and compensate for the sign
by changing the unnamed morphism into its negative; the result is
a distinguished triangle
\begin{equation}\label{eq:mapping_cone_triangle}
   X    \ \xra{\ \gamma(f)\ } \ Y\ \xra{\ \gamma(j)\ } \ 
Cf\ \xra{\qquad} \  \Sigma X  \ .
\end{equation}
\end{eg}

For $n\geq 2$ we define an essentially finite based $\Delta$-set $S\td{n}$ by
$$  S\td{n}_0 \ = \ \{e_i\ |\ i\in\mZ/n\}  \text{\qquad and\qquad}
S\td{n}_1 \ = \ \{f_i\ |\ i\in\mZ/n\}\cup\{\ast\}   \ ,
$$
and with $S\td{n}_k$ consisting only of the basepoint for $k\geq 2$.
We take the vertex $e_0$ as the basepoint in $S\td{n}_0$.
The face maps are given by $\partial f_i=(e_i,e_{i+1})$ (to be read modulo $n$).
As in Example~\ref{eg-smash with S^1} we let $S^1$ be the based $\Delta$-set
with exactly one non-basepoint simplex $z$ of dimension~1.
We now define morphisms of based $\Delta$-sets 
\[ \psi_i \ : S\td{n}\ \to \ S^1 
\text{\qquad and\qquad} 
\nabla\ :\ S\td{n}\ \to\  S^1\ .   \]
The morphism $\psi_i$ is determined by $\psi_i(f_i)=z$
and $\psi_i(f_k)=*$ for $i\not\equiv k$ modulo $n$.
The morphism~$\nabla$ is determined by $\nabla(f_i)=z$ for all $i\in\mZ/n$.
The square
\[ \xymatrix@C=15mm{S\td{n}-\{f_i\}\ar[d]_{\text{incl}}\ar[r]^-{\psi_i} &
\ast \ar[d] \\
S\td{n} \ar[r]_-{\psi_i} & S^1 } \]
is a pushout of based $\Delta$-sets
and both $S\td{n}-\{f_i\}$ and $\ast$ are weakly contractible,
so $\psi_i$ is a weak equivalence.

\begin{prop}\label{prop-psi_i homotopic}
Let $X$ be an object in a $\Delta$-cofibration category $\cC$.
\begin{enumerate}[\em (i)]
\item 
The relation $\gamma(\psi_i\sm X)=\gamma(\psi_{i+1}\sm X)$
holds as morphism from $S\td{n}\sm X$ to $S^1\sm X$ 
in $\bHo(\cC)$.  
\item If $\cC$ is stable, then the relation
\[ \gamma(\nabla\sm X) \ = \ n\cdot \gamma(\psi_1\sm X)  \]
holds as morphisms from $S\td{n}\sm X$ to $S^1\sm X$ 
in $\bHo(\cC)$.
\end{enumerate}
\end{prop}
\begin{proof} 
(i) We let $\hat S^1$ be the extension of the $\Delta$-set $S^1$ given by
\[ \hat S^1_0=\{\ast\}\ , \quad \hat S^1_1=\{\ast, z, g, g'\} 
\text{\qquad and\qquad} \hat S^1_2=\{\ast, c, c'\} \] 
and $\hat S^1_k=\{\ast\}$ for $k\geq 3$.
The faces of the additional 2-simplices are 
$\partial c = (\ast,g,z)$ and $\partial c'=(z,g',\ast)$.
The inclusion $j:S^1\to\hat S^1$ is a sequence 
of two elementary expansions, hence a weak equivalence.
The morphism $\gamma(j\sm X):S^1\sm X\to\hat S^1\sm X$ is then an
isomorphism in $\bHo(\cC)$.

We let $S\subseteq S\td{n}$ be the (unbased) 1-dimensional sub-$\Delta$-set
consisting of the simplices $e_i$ and~$f_i$ for all $i\in\mZ/n$.
The inclusion $\iota:S\to S\td{n}$ is a finite presentation
with~$\iota^{-1}(\ast)=\{e_0\}$, the basepoint in dimension~0.
By Proposition~\ref{prop-finite presentation}, the object $S\td{n}\sm X$ 
is then a cokernel of the cofibration 
$\iota\tensor X:\{e_0\}\tensor X\to S\tensor X$,
and we claim that the projection $q:S\tensor X\to S\td{n}\sm X$
becomes a split epimorphism in the homotopy category $\bHo(\cC)$.
Indeed, the cofibration $\iota\tensor X$
gives rise to an elementary distinguished triangle in $\bHo(\cC)$: 
\[ \{e_0\}\tensor X\ \xra{\gamma(\iota\tensor X)}\ 
S\tensor X \ \xra{\ \gamma(q)\ }\  S\td{n}\sm X \ 
\xra{\delta(\iota\tensor X)}
\ \Sigma(\{e_0\}\tensor X)\]
The composite of the inclusion $\{e_0\}\to S$ 
with the cone inclusion $i_S:S\to CS$ (compare Example~\ref{eg-cone})
is a weak equivalence, hence the composite 
\[  \{e_0\}\tensor X\ \xra{\gamma(\iota\tensor X)}\ 
S\tensor X \ \xra{\gamma(i_S\tensor X)}\  CS\tensor X  \]
is an isomorphism. So $\gamma(\iota\tensor X)$ is a split monomorphism,
and so  $\gamma(q)$ is a split epimorphism.

We will now define a combinatorial homotopy, i.e., a morphism
of $\Delta$-sets $H:S\tensor\Delta[1]\to \hat S^1$.
We let $\varphi,\varphi':[2]\to [1]\times[1]$ be the two monotone
injective maps defined by
\[ \varphi(0)=\varphi'(0)=(0,0)\ , \quad 
 \varphi(1)=(0,1)\ , \quad \varphi'(1)=(1,0)\text{\quad and\quad}
 \varphi(2)=\varphi'(2)=(1,1)\ . \]
The $\Delta$-set $S\tensor\Delta[1]$ is generated by the 2-simplices
\[ A_k \ = \ [f_k,\Id_{[1]};\, \varphi] \text{\qquad and\qquad}
 B_k \ = \ [f_k,\Id_{[1]};\, \varphi'] \]
for $k\in\mZ/n$, subject only to the relations
\[ A_kd_1 \ = \ B_kd_1 \text{\qquad and\qquad}
A_kd_2 \ = \ B_{k+1}d_0   \]
(to be read modulo $n$).
The homotopy $H:S\tensor\Delta[1]\to\hat S^1$ is determined by 
\[ H(A_k) \ = \  H(B_k) \ = \
\begin{cases}
\  c&\text{\ if $k\equiv i$ modulo $n$,}\\
\  c'&\text{\ if $k\equiv i+1$ modulo $n$, and}\\
\  \ast&\text{\ else.}
\end{cases} \]

We let $i_0, i_1:S\to S\tensor\Delta[1]$ be the `front and back inclusion', 
i.e., the morphisms  determined by 
\[ i_0(f_k) \ = \ B_kd_2 \text{\qquad respectively\qquad} 
i_1(f_k) \ = \  A_kd_0\ . \]
We let $P$ be a 2-dimensional $\Delta$-set with a
unique simplex of dimension~0, 1 and 2. 
The unique morphism $u:\Delta[1]\to P$ satisfies 
$(S\tensor u)i_0=(S\tensor u)i_1$ as morphisms $S\to S\tensor P$.
So we also have
$\gamma(S\tensor u\tensor X)\circ \gamma(i_0\tensor X)=\gamma(S\tensor u\tensor X)\circ \gamma(i_1\tensor X)$.
Since $P$ is weakly contractible, the morphism $u$, and hence also
the morphism $S\tensor u\tensor X:S\tensor\Delta[1]\tensor X\to S\tensor P\tensor X$ 
is a weak equivalence. 
So $\gamma(S\tensor u\tensor X)$ is invertible and we conclude that 
$\gamma(i_0\tensor X)=\gamma(i_1\tensor X)$.

The homotopy $H$ has image in the 2-skeleton $(\hat S^1)^{(2)}$ of $\hat S^1$.
The morphisms $Hi_0$ and $Hi_1$ are lifts of the based morphism $j\psi_i$ 
respectively $j\psi_{i+1}$ 
to morphisms of unbased $\Delta$-sets, i.e., the squares 
\[ 
 \xymatrix@C=19mm{ 
S\ar[r]^-{Hi_0}\ar[d]_{\text{incl}} &
(\hat S^1)^{(2)}\ar[d]^{\text{incl}} &
S\ar[r]^-{Hi_1}\ar[d]_{\text{incl}} &
(\hat S^1)^{(2)}\ar[d]^{\text{incl}} \\
S\td{n}\ar[r]_-{j\psi_i} & \hat S^1 &
S\td{n}\ar[r]_-{j\psi_{i+1}} & \hat S^1 } \]
commute. Since the vertical maps are finite presentations,  the squares
\[ \xymatrix@C=17mm{
S\tensor X\ar[d]_q  \ar[r]^-{Hi_0\tensor X} &
(\hat S^1)^{(2)}\tensor X \ar[d]^-{\text{proj}}&
S\tensor X\ar[d]_q  \ar[r]^-{Hi_1\tensor X} &
(\hat S^1)^{(2)}\tensor X \ar[d]^-{\text{proj}}
\\
S\td{n}\sm X \ar[r]_-{j\psi_i\sm X} & \hat S^1\sm X  &
S\td{n}\sm X \ar[r]_-{j\psi_{i+1}\sm X} & \hat S^1\sm X  
} \]
commute in $\cC$. So we conclude that
\begin{align*}
   \gamma(j\sm X)\circ\gamma(\psi_i\sm X)\circ\gamma(q) \ &= \ 
   \gamma(\text{proj})\circ\gamma(H\tensor X)\circ\gamma(i_0\tensor X) \\ 
&= \ \gamma(\text{proj})\circ\gamma(H\tensor X)\circ\gamma(i_1\tensor X) \\ 
&= \ \gamma(j\sm X)\circ\gamma(\psi_{i+1}\sm X)\circ\gamma(q) \ .
\end{align*}
Since $\gamma(j\sm X)$ is an isomorphism and $\gamma(q)$ 
is a split epimorphism, this implies the desired relation 
$\gamma(\psi_i\sm X)=\gamma(\psi_{i+1}\sm X)$.

(ii) We define a morphism of based $\Delta$-sets
$\kappa:S\td{n}\to\bigvee_{j=1}^n  S^1$
by sending $f_i$ to the $i$-th copy of $z$.
This morphism makes the diagram
\[ \xymatrix@C=20mm{ 
& S\td{n}\sm X \ar[ld]_{\nabla\sm X} \ar[d]_-{\kappa\sm X} 
\ar[dr]^-(.7){\psi_i\sm X} \\
S^1\sm X &
\bigvee_{j=1}^n S^1\sm X \ar[l]^-{\text{fold}} \ar[r]_-{p_i} &
S^1\sm X } \]
commute in $\cC$, where $p_i$ is the projection to the $i$-th wedge summand.
Since coproducts in $\cC$ become sums
in the homotopy category of $\cC$, the fold map occurring in the diagram
becomes the sum of the $n$ projections in $\bHo(\cC)$. 
So we obtain the desired relation
\[   \gamma(\nabla\sm X)\ = \ 
\sum_{i=1}^n\ \gamma(p_i) \circ \gamma(\kappa\sm X) 
\ = \ \sum_{i=1}^n\ \gamma(\psi_i) \ = \ n\cdot \gamma(\psi_1)\ . \qedhere \]
\end{proof}

Our next aim is to define `actions' of a mod-$n$ Moore space 
(i.e., $\Delta$-set) on objects of a pointed $\Delta$-cofibration category.
For the arguments below we need a $\Delta$-set model for
a Moore space whose reduced homology
is concentrated in an {\em even} dimension. The easiest way to construct one
is to cone off a suspension of the morphism $\nabla:S\td{n}\to S^1$.
To make this precise we need the following `reduced' version of the geometric
product for based $\Delta$-sets.
A based $\Delta$-set $K$ admits a unique based morphism $\ast\to K$
from the terminal $\Delta$-set.
Given two based $\Delta$-sets~$K$ and $L$, we define their
{\em geometric smash product} $K\sm L$ as the pushout
\[ \xymatrix@C=15mm{ (K\tensor\ast) \cup_{(\ast\tensor\ast)}(\ast\tensor L)
\ar[r]^-{\text{incl}} \ar[d] & K\tensor L \ar[d] \\
\ast \ar[r]& K\sm L} \]
If both $K$ and $L$ are essentially finite,
then the smash product $K\sm L$ is again essentially finite.
The associativity isomorphism in a pointed $\Delta$-cofibration
category then passes to coherent isomorphisms
\[  (K\sm L)\sm X \ \iso\ K\sm (L\sm X) \ .  \]

From Example~\ref{eg-smash with S^1} we recall 
the essentially finite based $\Delta$-set $I$ with
\[ I_1=\{z,\ast\} \ , \quad  I_0=\{zd_0,zd_1\}\]
and with $I_k=\{\ast\}$ for $k\geq 2$. The vertex $zd_0$
is the basepoint in dimension~0. For every based $\Delta$-set $K$
the smash product $I\sm K$ is weakly contractible
and comes with a monomorphism $i_K:K\to I\sm K$.
Now we can define an essentially finite based $\Delta$-set $M$ as the pushout:
\begin{equation}\begin{aligned}\label{eq-define M}
\xymatrix@C=25mm{ S^1\sm S\td{n} 
\ar[d]_{S^1\sm\nabla} 
\ar[r]^-{i_{S^1\sm S\td{n}}} & 
I\sm S^1\sm S\td{n}\ar[d]\\
S^2 \ar[r]_-\iota &  M }  
\end{aligned}\end{equation}
where we use the abbreviation $S^2=S^1\sm S^1$.
The geometric realization of $M$ is then a mod-$n$ Moore space, i.e.,
a simply connected CW-complex whose reduced homology is $\mZ/n$ 
concentrated in dimension 2.
Moreover, the inclusion $\iota:S^2\to M$
induces an epimorphism in integral homology.

If we smash the pushout square~\eqref{eq-define M}
with an object $X$ of a pointed $\Delta$-cofibration category~$\cC$,
we obtain a pushout in~$\cC$:
\[ \xymatrix@C=20mm{ 
S^1\sm S\td{n}\sm X \ar[d]_-{S^1\sm \nabla\sm X} 
\ar[r]^-{i_{S^1\sm S\td{n}\sm X}} & 
I\sm S^1\sm S\td{n}\sm X\ar[d] \\
S^2\sm X \ar[r]_-{\iota\sm X}& M\sm X } \]
This shows that $M\sm X$ is the mapping cone, 
as defined in~\eqref{eq:mapping_cone}, of the morphism
$S^1\sm \nabla\sm X:S^1\sm S\td{n}\sm X\to S^2\sm X$.
The distinguished triangle~\eqref{eq:mapping_cone_triangle}
then becomes a distinguished triangle
\[     S^1\sm S\td{n}\sm X    \ \xra{\ \gamma(S^1\sm \nabla\sm X)\ } 
\ S^2\sm X\ \xra{\ \gamma(\iota\sm X)\ } \ 
M\sm X\ \xra{\qquad} \  \Sigma(S^1\sm S\td{n}\sm X ) \ .  \]
We use the isomorphism
$\gamma(S^1\sm\psi_1\sm X):S^1\sm S\td{n}\sm X\to S^1\sm S^1\sm X=S^2\sm X$
in $\bHo(\cC)$
to replace the first and last objects in this triangle.
Since $\gamma(\nabla\sm X) = n\cdot \gamma(\psi_1)$
by Proposition~\ref{prop-psi_i homotopic}~(ii), 
the morphism $\gamma(S^1\sm\nabla\sm X)$
then turns into the $n$-fold multiple of the identity.
The upshot is a distinguished triangle
\begin{equation}\label{eq-smash M is mod n} 
S^2\sm X \ \xra{n\cdot(S^2\sm X)}\  S^2\sm X \
\xra{\ \gamma(\iota\sm X)\ } \ M\sm X \ \xra{\qquad} \ 
\Sigma(S^2\sm X)  \ .
\end{equation}

\begin{defn}\label{def-symmetric powers}
We denote by $P^i$ the $i$-th {\em symmetric power} of the
based $\Delta$-set $M$,
\[ P^i \ = \ M^{\sm i}/\Sigma_i \ . \] 
Here the symmetric group $\Sigma_i$ permutes the factors
of the $i$-th geometric smash power.
The associativity isomorphism 
$M^{\sm i}\sm M^{\sm j}\iso M^{\sm(i+j)}$
passes to a quotient morphism of symmetric powers
\begin{align}\label{eq-canonical projection}
 \mu_{i,j} \ : \ P^i \sm P^j  \ = \  
(M^{\sm i}/\Sigma_i)\, &\sm\, (M^{\sm j}/\Sigma_j)\
\to \ M^{\sm (i+j)}/\Sigma_{i+j} =  P^{i+j} 
\end{align}
which we refer to as the {\em canonical projection}.
The canonical projections are associative in the sense that the following 
diagram commutes
\[\xymatrix@C=25mm{
P^i  \sm  P^j  \sm P^k \ar_{\mu_{i,j} \, \sm\, P^k}[d]
\ar^-{P^i\, \sm\,  \mu_{j,k}}[r] &
P^i  \sm P^{j+k} \ar^{\mu_{i,j+k}}[d] \\
P^{i+j}  \sm P^k \ar_{\mu_{i+j,k}}[r] & P^{i+j+k}
}\]
for all $i,j,k\geq 1$.
\end{defn}

The Moore `space' $M$ is an essentially finite based $\Delta$-set, 
hence all its smash powers $M^{\sm i}$ and its symmetric powers $P^i$ 
are essentially finite. 
So it makes sense to smash $P^i$ with objects in any pointed
$\Delta$-cofibration category.

\begin{defn}\label{def-coherent action} 
Let $\cC$ be a pointed $\Delta$-cofibration category. 
An {\em $M$-module} $X$ consists of an infinite sequence 
\[ X_{(1)}, X_{(2)}, \dots, X_{(k)}, \dots  \] 
of objects of $\cC$, together with morphisms in $\cC$
\[ \alpha_{i,j} \ : \  P^i\sm X_{(j)} \ \to \ X_{(i+j)}\] 
for $i,j\geq 1$ such that the associativity diagram
\[\xymatrix@C=20mm{
 P^i  \sm  P^j  \sm X_{(k)}\ar^-{\mu_{i,j}\sm X_{(k)}}[r]
\ar_{P^i\sm \alpha_{j,k}}[d] &
P^{i+j}  \sm X_{(k)} \ar^{\alpha_{i+j,k}}[d] \\
P^i \sm X_{(j+k)} \ar_{\alpha_{i,j+k}}[r] & X_{(i+j+k)}
}\]
commutes for all $i,j,k\geq 1$.
The {\em underlying object} of an $M$-module $X$ 
is the object $X_{(1)}$ of $\cC$. 
A {\em morphism} $f:X\to Y$ of $M$-modules consists
of $\cC$-morphisms $f_{(j)}:X_{(j)}\to Y_{(j)}$ for
$j\geq 1$, such that the diagrams
\[\xymatrix@C=20mm{  P^i \sm X_{(j)} \ar[d]_{\alpha_{i,j}} 
\ar[r]^{P^i\sm f_{(j)}} &
 P^i \sm Y_{(j)} \ar[d]^{\alpha_{i,j}} \\
X_{(i+j)} \ar[r]_-{f_{(i+j)}} & Y_{(i+j)} }\]
commute for $i,j\geq 1$.
\end{defn}

\begin{eg}\label{eg-tautological action} 
For every object $K$ of $\cC$ we define the {\em free $M$-module}
$M\triangle K$ generated by $K$  as follows. 
The $j$-th term of this $M$-module is
$$ (M\triangle K)_{(j)} \ = \  P^j\sm K$$ 
and the structure map 
$\alpha_{i,j}:  P^i\sm (M\triangle K)_{(j)}\to(M\triangle K)_{(i+j)}$
is the morphism
$$  \mu_{i,j}\sm K\ : \ P^i\sm P^j\sm K \ \to \ P^{i+j}\sm K \ .$$
The associativity condition is then a consequence of the associativity of
the projection maps.

Now we suppose that $X$ is any $M$-module and $f:K\to X_{(1)}$ 
a morphism in~$\cC$.
We denote by $X_{\bullet+1}$ the $M$-module
obtained from $X$ by forgetting the object $X_{(1)}$ and reindexing the
rest of the data, i.e.,  $(X_{\bullet+1})_{(i)}=X_{(i+1)}$, 
and similarly for the action maps.
As $i$ varies, the action maps $\alpha_{i,1}$ of $X$
assemble into a morphism of $M$-modules
\[ \alpha_{\bullet,1} \ : \ M\triangle X_{(1)} \ \to \ X_{\bullet+1} \ .\]
The {\em free extension} of $f:K\to X_{(1)}$
is the composite morphism of $M$-modules 
\begin{equation}\label{eq:free_extension}
 M\triangle K \ \xrightarrow{M\triangle f} \ 
M\triangle X_{(1)} \ \xrightarrow{\alpha_{\bullet,1}} \ X_{\bullet+1} \ .  
\end{equation}
\end{eg}

Now we prove a purely topological fact, namely that
for every prime~$p$ and for all~$i$ that are strictly less than $p$,
the $i$-th reduced symmetric power of a mod-$p$ Moore space 
is again a mod-$p$ Moore space.
The following proposition ought to be well-known,
but the author was unable to find a reference.

\begin{prop}\label{prop-symmetric power of Moore} 
Let $p$ be a prime and $(N,A)$ be a based CW-pair 
such that $A$ and $N$ are simply connected,
the reduced integral homology of $A$ and $N$ is concentrated in dimension~2
and the map $H_2(A,\mZ)\to H_2(N,\mZ)$ induced by the inclusion 
$A\subset N$ is isomorphic to the reduction map $\mZ\to\mZ/p$. 
Then for all $2\leq i\leq p-1$ the composite
$$ A\sm (N^{\sm (i-1)}/\Sigma_{i-1}) \ \xra{\text{\em incl}} \ 
N\sm (N^{\sm (i-1)}/\Sigma_{i-1}) \ \xra{\text{\em proj}} \ N^{\sm i}/\Sigma_i $$
is a weak equivalence.
\end{prop}
\begin{proof}
The space $N^{\sm i}/\Sigma_i$ is simply connected because it is the quotient 
of the simply connected $\Sigma_i$-CW-complex 
$N^{\sm i}$ with $\Sigma_i$-fixed basepoint.
Indeed, any element in the fundamental group of $N^{\sm i}/\Sigma_i$
can be represented by a closed path of 1-cells. Such a path can be lifted 
to a path of 1-cells in $N^{\sm i}$, which a priori need not close up. 
However, since the basepoint in $N^{\sm i}$ is $\Sigma_i$-fixed, 
it is the only preimage of the basepoint in  $N^{\sm i}/\Sigma_i$, so
any closed path at the basepoint in $N^{\sm i}/\Sigma_i$
lifts to a closed path in the simply connected space $N^{\sm i}$,
and is thus null-homotopic.

Since the source and target of the map 
$A \sm (N^{\sm i-1}/\Sigma_{i-1})\to N^{\sm i}/\Sigma_i$ 
in question are simply connected,
we may prove the claim by showing that the map induces an
isomorphism in cohomology with coefficients in any field~$k$.
We calculate this with the help of Bredon cohomology~\cite{Bredon-SLN}.
We denote by $\cO$ the orbit category of~$\Sigma_i$, with objects
the cosets $\Sigma_i/K$ for all subgroups $K$ of~$\Sigma_i$ and with
morphisms of maps of left $\Sigma_i$-sets.
In Chapter I, (10.4), of~\cite{Bredon-SLN} Bredon sets up the
{\em universal coefficient spectral sequence}
$$ E_2^{s,t}\ = \ \Ext_\cO^s(\mathbf H_t(N^{\sm i}),\underline k) 
\quad \Longrightarrow \quad \tilde H^{s+t}_{\Sigma_i}(N^{\sm i},\underline k) \
\iso \ \tilde H^{s+t}(N^{\sm i}/\Sigma_i,k) \ .$$ 
abutting to the reduced Bredon cohomology 
$\tilde H^*_{\Sigma_i}(N^{\sm i},\underline k)$ of the 
$\Sigma_i$-CW-complex  $N^{\sm i}$ with coefficients in $\underline k$, 
the constant functor 
$\cO^{\text{op}}\to (k\text{-vector spaces})$ with value~$k$.
Here $\mathbf H_t(N^{\sm n})$ is the
contravariant functor on the orbit category with values
$\mathbf H_t(N^{\sm n})(\Sigma_i/K)=\tilde H_t((N^{\sm i})^K,k)$,
the reduced $k$-homology of the $K$-fixed points.
The Ext groups are formed in the abelian category of coefficient systems, i.e.,
functors $\cO^{\text{op}}\to (k\text{-vector spaces})$.
In the identification of the abutment we use that Bredon cohomology 
with coefficients in a constant functor
is isomorphic to the cohomology of the quotient space.

For every subgroup $K$ of $\Sigma_i$ the fixed point subspace 
$(N^{\sm i})^K$ is homeomorphic to a smash product 
of at most $i$ copies of the mod-$p$ Moore space $N$.
If the field $k$ has characteristic different from $p$,
the functor $\mathbf H_t(N^{\sm i})$ is thus identically zero.
So the $E^2$-term of the universal coefficient spectral sequence
vanishes and we conclude that the reduced cohomology groups
$\tilde H^t(N^{\sm i}/\Sigma_i,k)$ all vanish if the characteristic of $k$ 
is different from $p$.

It remains to discuss the case $k=\mF_p$.
If $F:\cO^\text{op}\to (k\text{-vector spaces})$ is any functor,
then natural transformations from $F$ to the constant functor 
$\underline {\mF}_p$ are in bijective correspondence with 
the $\Sigma_i$-invariants of the  $\mF_p$-dual of $F(\Sigma_i/\{1\})$.
Since $i<p$, taking invariants by a $\Sigma_i$-action is exact,
so the constant functor $\underline \mF_p$ is an injective object.
Hence the Ext groups vanish in positive dimensions
and the universal coefficient spectral sequence collapses at
the $E_2$-term to isomorphisms
$$ \Hom_\cO(\mathbf H_t(N^{\sm i}),\underline \mF_p) 
\ \iso\ \tilde H^t(N^{\sm i}/\Sigma_i,\mF_p) \ . $$ 
The left hand side in turn is isomorphic to
$\tilde H^*(N^{\sm i},\mF_p)^{\Sigma_i}$,
the $\Sigma_i$-invariants of the reduced mod-$p$ cohomology of $N^{\sm i}$.

The reduced $\mF_p$-cohomology of $N$ has a basis
consisting of a 2-dimensional class $x$ and its 
\mbox{mod-$p$} Bockstein $\beta(x)$.
By the Kunneth theorem the $\mF_p$-cohomology of $N^{\sm i}$ 
is isomorphic to the \mbox{$i$-fold} tensor product of $\mF_p\{x,\beta(x)\}$. 
The $\Sigma_i$-action on cohomology introduces a sign
whenever two instances of the odd dimensional class $\beta(x)$ are permuted.
So the $\Sigma_i$-invariants of $\tilde H^*(N^{\sm i},\mF_p)$ 
are 2-dimensional, with basis consisting of $x^{\tensor i}$ of dimension $2i$ and
\[ \beta(x^{\tensor i})\ = \ 
\sum_{k=1}^i\ x^{\tensor (k-1)}\tensor \beta(x) \tensor x^{\tensor{i-k}}\]
of dimension $2i+1$. The class $x$ restricts to a generator $\bar x$
of $H^2(A,\mF_p)$, and so the mod-$p$ cohomology of  
$A\sm (N^{\sm (i-1)}/\Sigma_{i-1})$ has an $\mF_p$-basis 
consisting of the classes
\[ \bar x\tensor x^{\tensor(i-1)}\text{\qquad and\qquad}
 \beta(\bar x\tensor x^{\tensor(i-1)})\ ;\]
we can thus conclude that the map
$A\sm (N^{\sm (i-1)}/\Sigma_{i-1})\to N^{\sm i}/\Sigma_i$
in question is indeed an $\mF_p$-cohomology isomorphism.
\end{proof}

\begin{defn} An $M$-module $X$ is {\em $k$-coherent}
for a natural number $k\geq 1$, if the composite
\[ S^2\sm X_{(j-1)} \ \xrightarrow{\ \iota\,\sm X_{(j-1)}\ } \ 
M \sm X_{(j-1)} \ 
\ \xrightarrow{\ \alpha_{1,j-1}\ }\ X_{(j)} \]
is a weak equivalence for all $2\leq j\leq k$
(where we use that $M=P^1$).
\end{defn}

For example, every $M$-module $X$ is 1-coherent, 
and if $X$ is 2-coherent, then in the homotopy category of $\cC$ 
the composite map
\[  M \sm X_{(1)} \ \xrightarrow{\ \gamma(\alpha_{1,1})\ } 
\ X_{(2)}  \ \xleftarrow{\ \gamma(\alpha_{1,1}\circ(\iota\sm X_{(1)}))^{-1}\ } \ 
S^2\sm X_{(1)} \] 
is a retraction to $\gamma(\iota\sm X_{(1)}):S^2\sm X_{(1)}\to M\sm X_{(1)}$. 
So if the cofibration category is stable, then the identity map of $X_{(1)}$ 
is annihilated by $n$ in the group $[X_{(1)},X_{(1)}]_{\bHo(\cC)}$,
by the distinguished triangle~\eqref{eq-smash M is mod n}.

Now we can prove a key result, namely that every free $M$-module $M\triangle K$
is $(p-1)$-coherent.
Since the geometric realization of the terminal $\Delta$-set $\ast$
is {\em not} a point 
(but rather an infinite dimensional contractible CW-complex), 
the geometric realization does {\em not}
take the geometric smash product to the smash product of spaces. 
However, we can define the {\em reduced realization}
of a based $\Delta$-set $K$ as
\[ |K|_{\bullet} \ = \ |K|/|\ast| \ ,\]
where $\ast$ denotes the sub-$\Delta$-space of $K$ consisting of
the basepoints in the various dimensions.
Since a terminal $\Delta$-set is weakly contractible,
the projection $|K|\to |K|_\bullet$ is a homotopy equivalence.
Thus a morphism  of based $\Delta$-sets is a weak equivalence
if and only if it induces a weak equivalence of reduced realizations.
The homeomorphism~\eqref{eq:geometric_versus_cartesian_product}
from $|K\tensor L|$ to $|K|\times |L|$ passes to a homeomorphism
\[ |K\sm L|_{\bullet} \ \iso \  |K|_{\bullet}\sm |L|_{\bullet}\ . \]
For example, the reduced realization of the based $\Delta$-set
$S^1$ is a circle, so $|S^1\sm L|_\bullet$
is homeomorphic to the reduced suspension of $|L|_\bullet$.

\begin{prop}\label{prop-powers of Moore}
Let $p$ be an odd prime. Then for $2\leq i\leq p-1$ 
and every object $K$ of a pointed $\Delta$-cofibration category
the composite map
\[ S^2 \sm  P^{i-1}\sm K\ \xrightarrow{\ \iota\sm P^{i-1}\sm K\ } 
M\sm P^{i-1}\sm K \ \xrightarrow{\ \mu_{1,i-1}\sm K\ } \  P^i\sm K \] 
is a weak equivalence. In other words, 
the free $M$-module $M\triangle K$ is $(p-1)$-coherent.
\end{prop}
\begin{proof} 
The reduced geometric realization of the finite based $\Delta$-set $M$
is a mod-$p$ Moore space. So the CW-pair $(|M|_\bullet, |S^2|_\bullet)$
satisfies the hypothesis of Proposition~\ref{prop-symmetric power of Moore}.
The composite
\[  |S^2|_\bullet\sm (|M|_\bullet^{\sm (i-1)}/\Sigma_{i-1}) \ \xra{\text{incl}} \ 
|M|_\bullet\sm (|M|_\bullet^{\sm (i-1)}/\Sigma_{i-1}) \ \xra{\text{proj.}} \ |M|_\bullet^{\sm i}/\Sigma_i \]
is thus a weak equivalence. Reduced geometric realization preserves
smash products and commutes with orbits by group actions,
so the composite above is homeomorphic to the reduced realization
of the composite morphism of based $\Delta$-sets
\[  S^2\sm P^{i-1} \ \xra{\iota\sm P^{i-1}} \ 
M\sm P^{i-1} \ \xra{\ \mu_{1,i-1}\ } \ P^i\ . \]
So Proposition~\ref{prop-smash with equivalence gives weak equivalence}
ensures that smashing this morphism with $K$ results in a weak equivalence. 
\end{proof}

\begin{construction}
The mapping cone construction of Example~\ref{eg-smash with S^1}
gives a way to make new $M$-modules from old ones.
Given a morphism $f:X\to Y$ of $M$-modules we construct
another $M$-module $Cf$, the {\em mapping cone} of $f$ as follows.
For $j\geq 1$ we set  $(Cf)_{(j)}=C(f_{(j)})$, i.e.,
the $j$-th object $(Cf)_{(j)}$
is the mapping cone of the morphism $f_{(j)}:X_{(j)}\to Y_{(j)}$.
The action map~$\alpha_{i,j}:P^i\sm Cf_{(j)}\to Cf_{(i+j)}$ 
is obtained by taking horizontal pushouts in the commutative diagram:
\[\xymatrix@C=20mm{
P^i \sm Y_{(j)} \ar_{\alpha_{i,j}}[d] & 
P^i \sm X_{(j)} \ar[l]_-{P^i\sm f_{(j)}} \ar_{\alpha_{i,j}}[d]\ar[r]^-{P^i\sm i_{X_{(j)}}} &
P^i\sm CX_{(j)} \ar[d]^{\bar\alpha_{i,j}} \\
Y_{(i+j)}  & X_{(i+j)} \ar[l]^-{f_{(i+j)}}\ar[r]_-{i_{X_{(i+j)}}} &
CX_{(i+j)} }\]
where the right vertical map is the composite
\[  P^i\sm I\sm X_{(j)}\ 
\xra{\text{symmetry}\sm X_{(j)}} \ 
I\sm P^i\sm X_{(j)}\ \xra{\ I\sm \alpha_{i,j}\ } \ 
I\sm X_{(i+j)} \ .\]
Associativity is inherited from associativity of the actions on $X$ and $Y$.
\end{construction}

\begin{prop}\label{prop-cones}
Let $\cC$ be a pointed $\Delta$-cofibration category.
\begin{enumerate}[\em (i)]
\item Let $f:X\to Y$ be a morphism of $M$-modules.
If $X$ and $Y$ are $k$-coherent, then the mapping cone $Cf$
is again $k$-coherent.
\item Let $p$ be a prime and $X$ a $(k+1)$-coherent $M$-module 
for some $1\leq k\leq p-1$. Then for every $\cC$-morphism $f:K\to X_{(1)}$ 
the mapping cone 
of the free extension~\eqref{eq:free_extension}
$\hat f:M\triangle K\to X_{\bullet +1}$ is $k$-coherent.
\end{enumerate}
\end{prop}
\begin{proof}
(i) The morphism 
$\alpha_{1,j-1}(\iota\sm Cf_{(j-1)}):S^2\sm Cf_{(j-1)}\to Cf_{(j)}$
is obtained by passage to horizontal pushouts in the commutative diagram:
$$\xymatrix@C=18mm{
 S^2\sm Y_{(j-1)} \ar[d]_{\alpha_{1,j-1}(\iota\sm Y_{(j-1)})} & 
S^2\sm X_{(j-1)} \ar[d]^{\alpha_{1,j-1}(\iota\sm X_{(j-1)})}\ar[l]_-{S^2\sm f_{(j-1)}}\ar[r]^-{S^2\sm i_{X_{(j-1)}}} &
S^2\sm CX_{(j-1)}\ar[d]^{\bar\alpha_{1,j-1}(\iota\sm CX_{(j-1)})} \\
 Y_{(j)} &  X_{(j)} \ar[l]^-{f_{(j)}}\ar[r]_-{i_{X_{(j)}}} & 
CX_{(j)} }$$
The two left horizontal morphisms are cofibrations, and
for $j\leq k$ all  vertical morphisms are weak equivalences.
So by the gluing lemma, the induced map on pushouts is a weak equivalence for
$j\leq k$.

(ii) If $X$ is $(k+1)$-coherent, then the shifted $M$-module $X_{\bullet+1}$
is $k$-coherent. The free module $M\triangle K$ is $(p-1)$-coherent 
by Proposition~\ref{prop-powers of Moore}, hence $k$-coherent. 
So the mapping cone of the free extension
$\widehat f:M\triangle K\to X_{\bullet +1}$ is $k$-coherent by part~(i).
\end{proof}

We come to a final useful property of $M$-modules, needed in the next section.

\begin{lemma}\label{lemma-extension over acyclic}
 For every $M$-module $X$ and every acyclic cofibration
$\varphi_{(1)}:X_{(1)}\to Z_{(1)}$ there exists an $M$-module $Z$ 
with underlying object $Z_{(1)}$
and a morphism of $M$-modules $\varphi:X\to Z$
that extends $\varphi_{(1)}$ and such that every component
$\varphi_{(j)}:X_{(j)}\to Z_{(j)}$ is an acyclic cofibration.
\end{lemma}
\begin{proof}
For $j\geq 2$ we define the object  $Z_{(j)}$
and the morphism $\varphi_{(j)}$ as the pushout:
\[\xymatrix@C=20mm{ 
P^{j-1}\sm X_{(1)} \ar[r]^-{P^{j-1}\sm\varphi_{(1)}} \ar[d]_{\alpha_{j-1,1}}  & 
P^{j-1}\sm Z_{(1)} \ar[d]^{\alpha_{j-1,1}} \\ 
X_{(j)} \ar[r]_{\varphi_{(j)}} & \ Z_{(j)} }\]
Since the morphism $\varphi_{(1)}$ is an acyclic cofibration,
so is $P^{j-1}\sm\varphi_{(1)}$;
hence the pushout exists and the morphism $\varphi_{(j)}$ 
is an acyclic cofibration. 
The structure maps $\alpha_{i,j}:P^i\sm Z_{(j)}\to Z_{(i+j)}$
are induced on pushouts by the commutative diagram:
\[\xymatrix@C=25mm{ 
P^i\sm X_{(j)} \ar[d]_{\alpha_{i,j}} &
P^i\sm P^{j-1}\sm X_{(1)} \ar[r]^-{P^i\sm P^{j-1}\sm\varphi_{(1)}} 
\ar[l]_-{P^i\sm\alpha_{j-1,1}} \ar[d]^{\mu_{i,j-1}\sm X_{(1)}} & 
P^i\sm  P^{j-1}\sm Z_{(1)} \ar[d]^{\mu_{i,j-1}\sm Z_{(1)}} \\ 
X_{(i+j)} &  P^{i+j-1}\sm X_{(1)} 
\ar[r]_-{P^{i+j-1}\sm\varphi_{(1)}} \ar[l]^-{\alpha_{i+j-1,1}}  & 
P^{i+j-1}\sm Z_{(1)} }\]
The associativity condition follows.
\end{proof}

\section{\texorpdfstring{The $p$-order in topological triangulated categories}{The p-order in topological triangulated categories}}
\label{sec-order in topological}

In this section we prove that every topological triangulated category
has $p$-order at least~$p-1$, for any prime~$p$,
see Theorem~\ref{thm-general topological}.
The proof relies on the techniques developed in the last two sections.
A key ingredient is the following link between 
the concepts of coherent $M$-action
(which take place in a cofibration category) 
and the notion of $p$-order 
(which takes place in the triangulated homotopy category).

In any pointed $\Delta$-cofibration category 
the functor $S^2\sm -:\cC\to \cC$ is exact and so it
descends to an exact functor of triangulated categories
on $\bHo(\cC)$ (see Proposition~\ref{prop-left derived is exact}).
We denote the derived functor again by $S^2\sm-$; it is
naturally isomorphic to the double suspension functor.
We define an {\em $M$-extension} of a morphism $f:K\to X$ in $\bHo(\cC)$
to be a morphism $\Phi:M\sm K\to S^2\sm X$ in $\bHo(\cC)$ such that 
$\Phi\circ\gamma(\iota\sm K)=S^2\sm f:S^2\sm K\to S^2\sm X$.

\begin{prop}\label{prop-order and coherence} 
Let $\cC$ be a stable $\Delta$-cofibration category and $p$ a prime.
Let $X$ be the underlying object of a $(k+1)$-coherent $M$-module 
for some $k< p$.
\begin{enumerate}[\em (i)]
\item For every object $K$ of $\cC$ and every morphism
$f:K\to X$ in $\bHo(\cC)$ there exists an $M$-extension
$\Phi:M\sm K\to S^2\sm X$ of  $f$ and a distinguished triangle
$$ M\sm K \ \xra{\ \Phi\  } \  
S^2\sm X \ \xra{\qquad } \ C \ \xra{\qquad } \ \Sigma(M\sm K)$$
such that the object $C$ is underlying a $k$-coherent $M$-module.
\item The $p$-order of $X$ is at least $k$.
\end{enumerate}
\end{prop}
\begin{proof} (i) By assumption there is a $(k+1)$-coherent $M$-module
$Y$ such that $Y_{(1)}=X$. The morphism $f:K\to X=Y_{(1)}$ 
in the homotopy category can be written as a fraction 
$f=\gamma(s_{(1)})^{-1}\gamma(a)$ where $a:K\to Z_{(1)}$ 
and $s_{(1)}:Y_{(1)}\to Z_{(1)}$ are morphisms in $\cC$ 
and $s$ is an acyclic cofibration. 
By Lemma~\ref{lemma-extension over acyclic} the object $Z_{(1)}$ 
can be extended to an $M$-module $Z$ and $s$ can be extended to a morphism
$s:Y\to Z$ of $M$-modules all of whose components $s_{(i)}:Y_{(i)}\to Z_{(i)}$
are acyclic cofibrations. Since $Y$ is $(k+1)$-coherent, $Z$ is
then also $(k+1)$-coherent.

We let $\hat a_{(1)}:M\sm  K\to Z_{(2)}$ be the first component
of the free extension~\eqref{eq:free_extension} of $a$, i.e., the composite
\[
M\sm K \ \xrightarrow{M\sm a} \ M\sm Z_{(1)} \
\xrightarrow{\alpha_{1,1}} \ Z_{(2)} \ .  \]
The square in~$\cC$
$$\xymatrix@C=13mm{ S^2\sm K \ar[r]^-{S^2\sm a} \ar[d]_{\iota\sm K} &
S^2\sm Z_{(1)} \ar[d]^{\alpha_{1,1}(\iota\sm Z_{(1)})}_{\sim} \\
M\sm K \ar[r]_-{\hat a_{(1)}} & Z_{(2)}}$$
commutes and the right vertical morphism is a weak equivalence.
The composite in $\bHo(\cC)$
$$ M\sm K \ \xra{\ \gamma(\hat a_{(1)})\ } \ Z_{(2)} \ 
\xra{\gamma(\alpha_{1,1}(\iota\sm s_{(1)}))^{-1}} \ S^2\sm X $$
is thus an $M$-extension $\Phi$ of $f$. The diagram
$$\xymatrix@C=15mm{ M\sm K \ar[r]^-{\Phi} \ar@{=}[d] & 
S^2\sm X \ar[d]_\iso^{\gamma(\alpha_{1,1}(\iota\sm s_{(1)}))}\ar[r] &  
C\hat a_{(1)} \ar@{=}[d] \\
 M\sm K \ar[r]_-{\gamma(\hat a_{(1)})}  & Z_{(2)} \ar[r] &  C\hat a_{(1)}}$$
commutes in $\bHo(\cC)$; since the lower row is part 
of a distinguished triangle by~\eqref{eq:mapping_cone_triangle}, 
so is the upper row.
The mapping cone $C\hat a_{(1)}$ is the underlying object
of the $M$-module $C\hat a$ which is $k$-coherent 
by Proposition~\ref{prop-cones}~(ii). 

(ii) We proceed by induction on $k$. For $k=0$ there is
nothing to show, so we assume $k\geq 1$.
Given any object $K$ of $\cC$ we let
$$ K \xra{\ p\cdot\ } K \xra{\ \pi\ } K/p \xra{\quad } \Sigma K $$
be a distinguished triangle that provides a cone of multiplication
by $p$ on $K$. We obtain a distinguished triangle
by smashing this triangle from the left with $S^2$, and another one 
from~\eqref{eq-smash M is mod n} (with $n=p$ and with $K$ instead of $X$). 
So there is an isomorphism  
$\psi:M\sm K\iso S^2\sm (K/p)$ in $\bHo(\cC)$ such that 
$S^2\sm \pi=\psi\circ\gamma(\iota\sm K)$
as morphisms $S^2\sm K\to S^2\sm(K/p)$. 

Now we let $f:K\to X$ be any morphism in $\bHo(\cC)$.
Part~(i) provides an $M$-extension $\Phi:M\sm K\to S^2\sm X$ of 
$f$ and a cone $C$ of $\Phi$ 
that admits a $k$-coherent $M$-action. 
By induction, this cone $C$ has $p$-order at least $k-1$. 
Since smashing with $S^2$ is isomorphic
to double suspension, and thus an equivalence of categories,
the map
$$ \bHo(\cC)(K/p,X) \ \to \  \bHo(\cC)(M\sm K,S^2\sm X) \ ,\quad
\varphi \ \longmapsto \ (S^2\sm \varphi)\circ\psi $$
is a bijection; so there is a unique morphism $\bar f:K/p\to X$ 
such that $(S^2\sm\bar f)\circ\psi=\Phi$, and the morphism $\bar f$
is an extension of $f$.
Finally, if $C\bar f$ is a cone of the extension $\bar f$,
then $S^2\sm C\bar f$ is a cone of 
$\Phi=(S^2\sm\bar f)\circ\psi:M\sm K\to S^2\sm X$.
The $p$-order is invariant under suspension and isomorphism, 
so the cone of $\bar f$ has $p$-order at least $k-1$.
Hence $X$ has $p$-order at least $k$.
\end{proof}

Now we can prove our main result.

\begin{theorem}\label{thm-general topological}
Let $\cT$ be a topological triangulated category and $p$ a prime.
Then for every object~$X$ of $\cT$ the object $X/p$ 
has $p$-order at least~$p-2$. 
In particular, the $p$-order of $\cT$ is at least $p-1$.
\end{theorem}
\begin{proof}
It suffices to treat the case where $\cT=\bHo(\cC)$
is the homotopy category of a stable cofibration category $\cC$.
We may assume without loss of generality that $\cC$
is {\em saturated}, i.e., every morphism that becomes an isomorphism
in $\bHo(\cC)$ is a weak equivalence. 
Indeed, for an arbitrary cofibration category 
we can define another cofibration structure $\cC^\text{sat}$
on the same category $\cC$ with the same class of cofibrations as before, 
but with new weak equivalences
those morphisms that become isomorphisms in the (old) homotopy category.
By Proposition~3.16 of~\cite{cisinski-categories derivables},
this is indeed a saturated cofibration structure 
and the identity functor $\cC\to \cC^\text{sat}$ induces an
isomorphism of homotopy categories
\[  \bHo(\cC)\ \to\ \bHo(\cC^\text{sat})\ .  \]
Assuming now that the cofibration category $\cC$ is saturated,
Theorem~\ref{theorem-framing} lets us replace it
by the category $f\cC$ of frames in $\cC$ without changing the
homotopy category; the cofibration category $f\cC$ of frames
is then a $\Delta$-cofibration category 
by Theorem~\ref{thm-frames are Delta-enriched}.
The upshot is that we can assume without loss of generality that $\cC$ is
a stable $\Delta$-cofibration category.

Now we let $X$ be an arbitrary object of $\cC$.
We choose an object $Y$ whose 2-fold suspension
is isomorphic to $X$ in~$\bHo(\cC)$.
Then the smash product $M\sm Y$ with the mod-$p$ Moore space is 
isomorphic to~$X/p$, 
by the distinguished triangle~\eqref{eq-smash M is mod n}. 
This smash product $M\sm Y$ is underlying the free $M$-module 
$M\triangle Y$ that is $(p-1)$-coherent 
by Proposition~\ref{prop-powers of Moore}.
So $M\sm Y$ has $p$-order at least $p-2$ by 
Proposition~\ref{prop-order and coherence}~(ii).
\end{proof}

Theorem~\ref{thm-general topological} is trivially
true for $p=2$, so its content lies in the odd primary cases.
If $\cT$ is a topological triangulated category and $p$ an odd prime, 
then Theorem~\ref{thm-general topological} 
in particular implies that $p\cdot X/p=0$ for every object $X$ of $\cT$. 
It is an open question whether there is an odd prime~$p$,
a triangulated category $\cT$ -- necessarily non-topological -- 
and an object $X$ of $\cT$ such that the $p$-order of
$X/p$ is less than $p-2$, or even $p\cdot X/p\ne 0$.

The stable homotopy category is topological, 
so Theorem~\ref{thm-general topological} shows in particular
that its $p$-order is at least $p-1$. 
In combination with~Theorem~\ref{thm-order of Moore} we thus obtain:

\begin{cor}\label{cor-exact order of Moore}
For every prime~$p$ the homotopy category $\bSH^c_{(p)}$ 
of finite $p$-local spectra has $p$-order exactly $p-1$.
\end{cor}

\begin{rk}
Theorem~\ref{thm-general topological} can be improved to give
a sufficient condition for when the $p$-order of $X/p$ in a topological
triangulated category is at least $p-1$. A choice of model for $\cT$ 
(i.e., a stable cofibration category $\cC$ and an exact equivalence
$\cT\iso\bHo(\cC)$) provides an action of the homotopy category of
finite $\Delta$-sets on $\cT$, see Remark~\ref{rk-Ho CW action}.
The obstruction for the object $X/p$ to have $p$-order strictly 
greater than $p-2$ can be expressed in terms of this action as follows.

We let $\alpha_1:S^{2p}\to S^3$ denote a morphism of finite $\Delta$-sets
such that the geometric realizations of $S^{2p}$ and $S^3$ are homotopy 
equivalent to a $2p$-sphere respectively a 3-sphere and such that
the realization of $\alpha_1$ is a generator of the $p$-torsion 
in the homotopy group $\pi_{2p}(S^3)$; 
here $p=2$ is allowed, and then $\alpha_1$ realizes 
the homotopy class of the Hopf map $\eta$.
The arguments of Theorem~2.5 of~\cite{schwede-rigid} 
can be generalized from simplicial stable model categories
to stable $\Delta$-cofibration categories to show that if
the morphism $\alpha_1\sm X:S^{2p}\sm X\to S^3\sm X$ 
is zero (or divisible by~$p$) in $\bHo(\cC)$,
then the object $X/p$ admits a $p$-coherent $M$-action. 
By Proposition~\ref{prop-order and coherence}~(ii), the object $X/p$ then
has $p$-order at least $p-1$.

We want to emphasize that the morphism 
$\alpha_1\sm X:S^{2p}\sm X\to S^3\sm X$ does not 
have an intrinsic meaning in a topological triangulated category $\cT$, 
and depends on the choice of model for $\cT$. 
The $K_{(p)}$-local stable homotopy category 
at an odd prime~$p$ is an explicit example where different models lead to
different morphisms $\alpha_1\sm X$: 
there is the `natural model', i.e., the category of 
sequential spectra~\cite[Def.\,2.1]{BF} 
with the $K_{(p)}$-localization of the stable model structure of 
Bousfield and Friedlander.
The class $\alpha_1$ maps non-trivially to the $(2p-3)$-th homotopy group
of the localized sphere spectrum $L_{K_{(p)}}\mS$, 
so it acts non-trivially on $L_{K_{(p)}}\mS$.
However, Franke's theorem~\cite[Sec.\,2.2, Thm.\,5]{franke-adams} 
provides an `exotic' algebraic model for $\bHo(K_{(p)}\text{-local})$,
and in any algebraic model, all positive dimensional
stable homotopy classes act trivially on all objects.
\end{rk}

\begin{rk}
The proof of Theorem~\ref{thm-general topological} shows that 
the special features of topological 
over algebraic triangulated categories are closely related to
existence and properties of multiplications on mod-$p$ Moore spectra.
For primes~$p\geq 5$, the mod-$p$ Moore spectrum has a multiplication
in the stable homotopy category which is commutative and associative.
So on the level of tensor triangulated categories, there does not
seem to be any qualitative difference between the Moore spectrum $\mS/p$ 
as an object of the triangulated category $\bSH^c$
of finite spectra and $\mZ/p$ as an object of the bounded derived category 
of finitely generated abelian groups, as long as $p\geq 5$.
However, mod-$p$ Moore spectra do not have models
as strict ring spectra (or $A_\infty$ ring spectra);
this seems to be folklore, and a proof can be found 
in~\cite[Ex.\,3.3]{angeltveit}. 
With the methods of this paper, this also follows by combining 
Proposition~\ref{prop-reduction of ring spectra} 
and Theorem~\ref{thm-order of Moore}.
So the notion of $p$-order explains and quantifies
how the higher order non-associativity eventually
manifests itself in the triangulated structure of the
stable homotopy category (i.e., without any reference to the smash product).
\end{rk}

\begin{appendix}
  
\section{Homotopy category, suspension and triangulation}
\label{app-triangulation}

In this appendix we recall certain facts about the homotopy category
of a cofibration category that we need in the body of the paper.
In particular, we introduce the suspension functor
on the homotopy category of a pointed cofibration category
and show in Theorem~\ref{thm-HoC is triangulated}
that the homotopy category of a stable cofibration category
is naturally triangulated.

For us the {\em homotopy category} of a cofibration category is
any localization of $\cC$ at the class of weak equivalences.
Hence the homotopy category consists of a category $\bHo(\cC)$
with the same objects as $\cC$ and a functor $\gamma:\cC\to\bHo(\cC)$ 
that is the identity on objects, takes all weak equivalences
to isomorphisms and satisfies the following
universal property: for every category~$\cD$ and every
functor $F:\cC\to\cD$ that takes all weak equivalences to isomorphisms
there is a unique functor $\bar F:\bHo(\cC)\to\cD$ such that $\bar F\gamma=F$.

A {\em cylinder object} for an object $A$ in a cofibration category
is a quadruple $(I,i_0,i_1,p)$ consisting of an object $I$,
morphisms $i_0,i_1:A\to I$ and a weak equivalence $p:I\to A$ satisfying 
$pi_0=pi_1=\Id_A$ and such that $i_0+i_1:A\vee A\to I$
is a cofibration. Every object has a cylinder object: axiom (C4)
allows us to factor the fold map $\Id+\Id:A\vee A\to A$
as a cofibration $i_0+i_1:A\vee A\to I$ followed by a weak equivalence
$p:I\to A$.

Two morphisms $f,g:A\to Z$ in a cofibration category are {\em homotopic}
if there exists a cylinder object $(I,i_0,i_1,p)$ for  $A$ and
a morphism $H:I\to Z$ (the {\em homotopy}) such that $f=Hi_0$ and $g=Hi_1$.
Since the morphism $p$ in a cylinder object is a weak equivalence,
$\gamma(p)$ is an isomorphism in $\bHo(\cC)$ and so 
$\gamma(i_0)=\gamma(i_1)$ since they share $\gamma(p)$ as common left inverse.
So if $f$ and $g$ are homotopic via $H$, then
$\gamma(f)=\gamma(H)\gamma(i_0)=\gamma(H)\gamma(i_1)=\gamma(g)$.
In other words: homotopic morphisms become equal in the homotopy category.
The converse is not true in general, but part~(ii) of the following theorem 
says that the converse is true up to post-composition with a weak equivalence.

Parts~(i) and~(ii) of the following theorem are the dual statements
to Theorem~1 and Remark~2 of~\cite[part~I.2]{brown}. 
The results can also be found, with more detailed proofs
and slightly different terminology,
as~Theorem~6.4.5~(1a) respectively Theorem~6.4.4~(1c)
in~\cite{radulescu-ABC}. 
Part~(iii) (or rather the dual statement)
is a special case of~\cite[Cor.\,2.9]{cisinski-categories derivables}.

\begin{theorem}\label{thm-HoC a la Brown} 
Let $\cC$ be a cofibration category and $\gamma:\cC\to\bHo(\cC)$ 
a localization at the class of weak equivalences. Then:
\begin{enumerate}[\em (i)]
\item Every morphism in $\bHo(\cC)$ is a `left fraction', i.e., is
of the form $\gamma(s)^{-1}\gamma(f)$ where $f$ and~$s$ are $\cC$-morphisms
with the same target and $s$ is an acyclic cofibration.
\item Given two morphisms $f,g:A\to B$ in $\cC$, then
$\gamma(f)=\gamma(g)$ in $\bHo(\cC)$ if and only if there is an
acyclic cofibration $s:B\to\bar B$ such that $sf$ and $sg$ are homotopic.
\item The localization functor $\gamma:\cC\to\bHo(\cC)$ preserves coproducts. 
In particular, the homotopy category $\bHo(\cC)$ has finite coproducts.
\end{enumerate}
\end{theorem}

\begin{rk}\label{rk-enough fibrant objects}
On the face of it, the homotopy category of a cofibration category
raises set-theoretic issues:
in general the hom-`sets' in $\bHo(\cC)$ may not be small,
but rather proper classes.
One way to deal with this is to work with universes
in the sense of Grothendieck; 
the homotopy category of a cofibration category then always exists 
in a larger universe.

Another way to address the set theory issues is to 
restrict attention to those cofibration categories
that have `enough fibrant objects'.
An object of a cofibration category $\cC$ is {\em fibrant}
if every acyclic cofibration out of it has a retraction.
If the object $Z$ is fibrant, then the map
$$ \cC(A,Z) \ \to \ \bHo(\cC)(A,Z) \ , 
\quad f \ \longmapsto \ \gamma(f)$$
is surjective: an arbitrary morphism from $A$ to $Z$ in $\bHo(\cC)$
is of the form $\gamma(s)^{-1}\gamma(a)$ for some acyclic cofibration
$s:Z\to Z'$. Since $Z$ is fibrant, there is a retraction $r$
with $rs=\Id_Z$, and then $\gamma(s)^{-1}\gamma(a)=\gamma(ra)$.
Moreover, if two $\cC$-morphisms $f,g:A\to Z$ become equal
after applying the functor $\gamma$, then there is an
acyclic cofibration $s:Z\to Z'$ such that $sf$ is homotopic to $sg$. 
Composing with any retraction to $s$ shows that $f$ 
is already homotopic to $g$. So `homotopy' for morphisms 
into a fibrant object $Z$ is an equivalence relation and the map
$$ \cC(A,Z)/_\text{homotopy} \ \to \ \bHo(\cC)(A,Z) \ , 
\quad [f] \ \longmapsto \ \gamma(f)$$
is bijective.

We say that the cofibration category $\cC$ {\em has enough fibrant objects} 
if for every object~$X$ there is a weak equivalence 
$r:X\to Z$ with fibrant target.
For example, if $\cC$ is the collection of cofibrant objects
in an ambient Quillen model category, then it has enough fibrant objects. 
In the cofibration structure on pretriangulated dg categories
discussed in Proposition~\ref{prop-algebraic is topological},
every object is fibrant (so in particular, there are enough fibrant objects).

If $r:X\to Z$ a weak equivalence with fibrant target, then 
for every other object  $A$ the two maps
$$ \bHo(\cC)(A,X) \ \xra{\gamma(r)_*}\ \bHo(\cC)(A,Z) \ \xla{\ \gamma\ }\
\cC(A,Z)/_\text{homotopy}$$
are bijective, so the morphisms  $\bHo(\cC)(A,X)$ form a set
(as opposed to a proper class). 
So if $\cC$ has enough fibrant objects,
then the homotopy category $\bHo(\cC)$ has small hom-sets
(or is `locally small').
\end{rk}

Now we proceed to the construction of the cone and suspension functor.
From now on, $\cC$ is a {\em pointed} cofibration category,
i.e., every initial object is also terminal, hence a zero object.
The following construction of the suspension functor 
is isomorphic to  the construction 
by Quillen~\cite[I.2]{Q} and (the dual of) Brown~\cite[Part~I, Thm.\,3]{brown},
although our exposition is somewhat different.

A {\em cone} in a pointed cofibration category $\cC$ is
a cofibration $i:A\to C$ whose target $C$ is weakly contractible.
The unique morphism from any given object
to the terminal object can be factored as a cofibration followed by
a weak equivalence; so every object is the source of a cone.
Cones of objects in a pointed cofibration category
are unique up to homotopy in a rather strong sense. 
Indeed, the category $\Cone\cC$ of cones in $\cC$ has a natural structure of
cofibration category and the source functor that sends a cone $i:A\to C$
to $A$ is exact and passes to an equivalence
$\bHo(\Cone\cC)\to \bHo(\cC)$ of homotopy categories.

Let us now choose a cone for every object $A$ of
$\cC$, i.e., a cofibration $i_A:A\to CA$ with weakly contractible target. 
The {\em suspension} $\Sigma A$ of $A$ is then the quotient 
of the chosen cone inclusion $i_A:A\to CA$, i.e., a pushout:
$$\xymatrix{ A \ar[r]^-{i_A}\ar[d] & CA \ar[d]^p\\
\ast\ar[r]&\Sigma A}$$

\begin{lemma}\label{lemma-cone extensions} 
Let $i:A\to C$ be a cone and $\alpha:A\to B$ a morphism 
in a pointed cofibration category $\cC$. 
Then there exists a {\em cone extension} of $\alpha$, i.e.,
a pair $(\bar\alpha,s)$ consisting of a morphism $\bar\alpha:C\to\bar C$
and an acyclic cofibration $s:CB\to\bar C$ 
such that $\bar\alpha i=si_B\alpha$ and such that the induced morphism
$\bar\alpha\cup s:C\cup_ACB\to \bar C$ is a cofibration,
where the source is a pushout of $i$ and $i_B\alpha$.
Moreover, the composite morphism in $\bHo(\cC)$
$$ C/A \ \xra{\gamma(\bar\alpha/\alpha)} \ \bar C/B \ 
\xra{\gamma(s/B)^{-1}}\ CB/B = \Sigma B $$
is independent of the cone extension $(\bar\alpha,s)$.
\end{lemma}
\begin{proof}
Since $i$ is a cofibration we can choose a pushout:
$$\xymatrix{ A \ar[r]^-{i_B\alpha} \ar[d]_i & CB \ar[d]^-{i'} \\
C \ar[r]_-k & C\cup_A CB }$$
Then we choose a cone $l:C\cup_ACB\to \bar C$,
and $\bar\alpha=lk$ and $s=li'$ have the desired properties.

Suppose that $(\bar\alpha':C\to\bar C',s':CB\to\bar C')$ 
is another cone extension.
Let us first suppose that there is a morphism
$\varphi:\bar C\to\bar C'$ (necessarily a weak equivalence) 
such that $\varphi\bar\alpha=\bar\alpha'$ and $\varphi s=s'$.
Then we have
\begin{align*}
\gamma(s/B)^{-1}\circ \gamma(\bar\alpha/\alpha) \ &=\
\gamma(s/B)^{-1}\circ \gamma(\varphi/B)^{-1}\circ \gamma(\varphi/B)\circ \gamma(\bar\alpha/\alpha)\\
&=\ \gamma(\varphi s/B)^{-1}\circ \gamma(\varphi\bar\alpha/\alpha)
\ =\ \gamma(s'/B)^{-1}\circ \gamma(\bar\alpha'/\alpha)\ .
\end{align*}
In the general case we choose a pushout:
$$\xymatrix{ C\cup_A CB \ar[r]^-{\bar\alpha\cup s} \ar[d]_{\bar\alpha'\cup s'}
& \bar C \ar[d]^-{k'} \\
\bar C' \ar[r]_-k & P }$$
Then $(i_Pk'\bar\alpha,i_Pk's)$ is yet another  cone extension of $\alpha$,
where $i_P:P\to CP$ is a cone on $P$.
Moreover, this new cone extension receives morphisms from both 
$(\bar\alpha,s)$ and $(\bar\alpha',s')$, so by the special case all three
cone extensions give rise to the same morphism $C/A\to\Sigma B$ in $\bHo(\cC)$.
\end{proof}

Now we can define suspension on morphisms, thereby
extending it to a functor $\Sigma:\cC\to\bHo(\cC)$.
Given a $\cC$-morphism $\alpha:A\to B$, we choose
a cone extension $(\bar\alpha,s)$ with respect to the chosen cone
$i_A:A\to CA$, as in Lemma~\ref{lemma-cone extensions}.
We define $\Sigma \alpha$ as the composite in $\bHo(\cC)$
$$ \Sigma A = CA/A \ \xra{\gamma(\bar\alpha/\alpha)} \ \bar C/B \ 
\xra{\gamma(s/B)^{-1}}\ CB/B = \Sigma B\ . $$
Lemma~\ref{lemma-cone extensions} guarantees that this definition is
independent of the cone extension.

\begin{prop}\label{prop-suspension} 
The suspension construction is a functor $\Sigma:\cC\to\bHo(\cC)$.
The suspension functor takes weak equivalences to isomorphisms
and preserves coproducts.
\end{prop}
\begin{proof} 
The pair $(\Id_{CA},\Id_{CA})$ is a cone extension of the identity of $A$,
so we have $\Sigma\Id_A=\Id_{\Sigma A}$. 
Given another morphism $\beta:B\to D$
and a cone extension $(\bar\beta,t)$ of $\beta$, 
we choose a pushout:
$$\xymatrix{ CB \ar[r]^-s \ar[d]_{\bar\beta} & \bar C \ar[d]^{\bar\beta'} \\
\bar C' \ar[r]_-{\bar s} & E }$$
The morphism $\bar s$ is an acyclic cofibration since $s$ is.
So the pair $(\bar\beta'\bar\alpha,\bar st)$ 
is a cone extension of $\beta\alpha:A\to D$ and we get
\begin{align*}
\Sigma(\beta\alpha) \ &= \  \gamma(\bar s t/D)^{-1} \circ\gamma(\bar\beta'\bar\alpha/\beta\alpha) \ = \ 
\gamma(t/D)^{-1} \circ\gamma(\bar s/D)^{-1}
 \circ\gamma(\bar\beta'/\beta)
\circ\gamma(\bar\alpha/\alpha)\\
&= \ 
\gamma(t/D)^{-1}  \circ\gamma(\bar\beta/\beta)
\circ\gamma(s/B)^{-1}
\circ\gamma(\bar\alpha/\alpha)\ =\ (\Sigma\beta)\circ(\Sigma\alpha)\ .
\end{align*}
The third equation uses that $\bar\beta's=\bar s\bar\beta$.
So the suspension construction is functorial. If $\alpha:A\to B$ is
a weak equivalence and $(\bar\alpha,s)$ a cone extension,
then $\bar\alpha/\alpha:CA/A\to\bar C/B$ is a weak equivalence by
the gluing lemma. So $\gamma(\bar\alpha/\alpha)$, and hence $\Sigma\alpha$
is an isomorphism in $\bHo(\cC)$.

It remains to show that the suspension functor preserves coproducts.
Indeed, if $i_A:A\to CA$, $i_B:B\to CB$ and $i_{A\vee B}:A\vee B\to C(A\vee B)$
are the chosen cones for two objects $A$, $B$ and a coproduct $A\vee B$,
then $i_A\vee i_B:A\vee B\to CA\vee CB$ is another cone,
so Lemma~\ref{lemma-cone extensions} provides a morphism 
$\bar\alpha:CA\vee CB\to\bar C$ and an acyclic cofibration 
$s:C(A\vee B)\to\bar C$ such that 
$\bar\alpha\circ(i_A\vee i_B)=s\circ i_{A\vee B}$.
Passing to quotients and applying the localization functor
produces an isomorphism 
\begin{align*}
\Sigma A\vee\Sigma B \ = \ 
(CA\vee CB)/(A\vee B)\ &\xra{\ \gamma(\bar\alpha/A\vee B)\ }\
\bar C/(A\vee B)\\ 
&\xra{\ \gamma(s/A\vee B)^{-1}\ }\
C(A\vee B)/(A\vee B)=\Sigma(A\vee B)\ , 
\end{align*}
which is in fact the canonical morphism.
\end{proof}

Since the suspension functor takes weak equivalences to isomorphisms,
it descends to a unique functor 
$$ \Sigma \ : \ \bHo(\cC) \ \to \  \bHo(\cC) $$
such that $\Sigma\circ\gamma=\Sigma$.
Since coproducts in $\cC$ are coproducts in $\bHo(\cC)$,
this induced suspension functor again preserves coproducts.

\begin{rk} 
In many examples, cones (and hence suspensions)
can be chosen functorially already on the 
level of the cofibration category.
However, the punchline of the previous construction is that even without
functorial cones in $\cC$, suspension becomes functorial after passage
to the homotopy category. 
On the other hand, it would not be a serious loss of generality 
to assume functorial cones and suspensions. Indeed,
Theorems~\ref{theorem-framing} and~\ref{thm-frames are Delta-enriched}
together say that for every saturated cofibration category $\cC$ 
the category $f\cC$ of frames in $\cC$ is a $\Delta$-cofibration category 
such that $\bHo(\cC)$ is equivalent to $\bHo(f\cC)$. 
Moreover, if $\cC$ is pointed,
then so is $f\cC$ and the category $f\cC$ has functorial cones 
and functorial suspensions given by smash product with the 
`based interval' $I$ respectively the `circle'~$S^1$,
compare Example~\ref{eg-smash with S^1}.
\end{rk}

There is extra structure on a suspension, namely a certain
{\em collapse morphism} $\kappa_A:\Sigma A\to\Sigma A\vee\Sigma A$
in $\bHo(\cC)$.
To define it, we consider a pushout $CA\cup_A CA$ 
of two copies of the cone~$CA$ along $i_A$.
The gluing lemma
guarantees that the map $0\cup p:CA\cup_A CA\to CA/A=\Sigma A$ 
induced on horizontal pushouts of the left commutative diagram
$$\xymatrix{ CA \ar[d]_\sim & A \ar[l]_-{i_A} \ar[r]^-{i_A} \ar@{=}[d] & 
CA \ar@{=}[d] &&&
 CA \ar[d]_p & A \ar[l]_-{i_A} \ar[r]^-{i_A} \ar[d] & CA \ar[d]^p\\
\ast  & A \ar[l] \ar[r]_-{i_A}  & CA &&&
\Sigma A  & \ast \ar[l] \ar[r]  & \Sigma A}$$
is a weak equivalence. We define the $\kappa_A$ as the composite
$$ \Sigma A\ \ \xra{\gamma(0\cup p)^{-1}} \ CA\cup_ACA \
\xra{\ \gamma(p\cup p)\ } \ \Sigma A\vee \Sigma A $$
where the second morphism is the image of the $\cC$-morphism
induced on horizontal pushouts of the right commutative diagram above.

\begin{prop}\label{prop-kappa properties}
The morphism $\kappa_A:\Sigma A\to \Sigma A\vee \Sigma A$ 
satisfies the relations
$$ (0+\Id)\kappa_A\ =\ \Id  \text{\quad and\quad}
(\Id+\Id)\kappa_A\ =\ 0 $$
as endomorphisms of $\Sigma A$ in $\bHo(\cC)$.
The endomorphism 
$$ m_A\ = \  (\Id+0)\kappa_A$$ 
of $\Sigma A$ is an involution, i.e., $m_A^2=\Id$.
The morphism $\kappa_A$ is natural, i.e., for every morphism
$a:A\to B$ we have
$(\Sigma a\vee \Sigma a)\circ\kappa_A=\kappa_B\circ(\Sigma a)$. 
\end{prop}
\begin{proof}
We observe that $(0+\Id)\circ (p\cup p)=0\cup p$ 
as $\cC$-morphisms $CA\cup_A CA\to \Sigma A$, so 
$$ (0+\Id)\circ\kappa_A \ = \  
(0+\Id)\circ\gamma(p\cup p)\circ\gamma(0\cup p)^{-1}
\ = \  \gamma(0\cup p)\circ\gamma(0\cup p)^{-1} \ = \ \Id\ .$$
The square
$$\xymatrix@C=12mm{CA\cup_A CA \ar[r]^-{p\cup p} \ar[d]_{\Id\cup\Id} &
\Sigma A\vee \Sigma A \ar[d]^{\Id+\Id}\\
CA\ar[r]_-p &\Sigma A}$$
commutes in $\cC$, so the morphism 
$(\Id+\Id)\circ\kappa_A= 
(\Id+\Id)\circ\gamma(p\cup p)\circ\gamma(0\cup p)^{-1}$
factors through the cone $CA$, which is a zero object in $\bHo(\cC)$.
Thus $(\Id+\Id)\kappa_A=0$.

For the next relation we denote by $\tau$ the involution of
$CA\cup_ACA$ that interchanges the two cones. Then we have
\begin{align*}
m_A \ &= \ (\Id+0)\circ \gamma(p\cup p)\circ\gamma(0\cup p)^{-1} \ 
= \ \gamma(p\cup 0)\circ\gamma(0\cup p)^{-1} \ 
= \ \gamma(0\cup p)\circ\gamma(\tau)\circ\gamma(0\cup p)^{-1} \ .
\end{align*}
Since $\tau^2=\Id$ this leads to $m_A^2=\Id$.

Every morphism in $\bHo(\cC)$ is of the form $\gamma(s)^{-1}\gamma(\alpha)$
for $\cC$-morphisms $\alpha$ and $s$; so it suffices to prove 
the naturality statement for morphisms of the form 
$a=\gamma(\alpha)$ for a $\cC$-morphism \mbox{$\alpha:A\to B$}. 
We choose a cone extension $(\bar\alpha,s)$ of $\alpha$ and consider 
the commutative diagram:
$$\xymatrix@C=18mm{ 
\Sigma A \ar[r]^-{\bar \alpha/\alpha}  & 
\bar CB/B\ & \Sigma B  \ar[l]_{s/\Id_B}^\sim\\
CA\cup_ACA \ar[r]^-{\bar\alpha\cup\bar\alpha}
\ar[u]^{0\cup p}_\sim\ar[d]_{p\cup p}  
& \bar CB\cup_{si_B} \bar CB \ar[u]^{0\cup p}_\sim\ar[d]_{p\cup p} & 
CB\cup_B CB \ar[u]_{0\cup p}^\sim\ar[d]^{p\cup p}\ar[l]_-{s\cup s}^-\sim \\
\Sigma A \vee\Sigma A\ar[r]_-{\bar\alpha/\alpha\vee \bar\alpha/\alpha}  
& \bar CB/B\vee\bar CB/B\ & 
\Sigma B \vee\Sigma B\ar[l]^-{s/\Id_B\vee s/\Id_B}_-\sim}$$
The vertical maps going up and the horizontal maps
going left are weak equivalences, so they
become isomorphisms in the homotopy category.
After inverting these weak equivalences in $\bHo(\cC)$, the composite
through the upper right corner becomes $\kappa_{B}\circ\Sigma\gamma(\alpha)$
and the composite through the lower left corner becomes 
$(\Sigma\gamma(\alpha)\vee\Sigma\gamma(\alpha))\circ\kappa_A$.
\end{proof}

We call a pointed cofibration category 
{\em stable} if the suspension functor
$\Sigma:\bHo(\cC)\to\bHo(\cC)$ is an autoequivalence.
We can now show that the homotopy category of a stable 
cofibration categories is additive.
By Theorem~\ref{thm-HoC a la Brown}~(iii)
the coproduct in any cofibration category~$\cC$ descends to a coproduct
in the homotopy category $\cC$. We will show that for stable $\cC$ 
the coproduct $X\vee Y$ is also a product of $X$ and $Y$ in $\bHo(\cC)$ 
with respect to the morphisms 
$p_X=\Id+0:X\vee Y\to X$ and $p_Y=0+\Id:X\vee Y\to Y$. 
So we have to show that for every object $B$ 
of $\cC$ the map
\begin{equation}\label{eq:coproduct_is_product}
\bHo(\cC)(B,X\vee Y) \ \to \ \bHo(\cC)(B,X)\times \bHo(\cC)(B,Y)\ , \quad
\varphi\ \longmapsto \ (p_X\varphi,\, p_Y\varphi)  
\end{equation}
is bijective.

\begin{prop}\label{prop-stable is additive} 
Let $\cC$ be a pointed cofibration category.
  \begin{enumerate}[\em (i)]
  \item If the object $B$ is a suspension, 
    then the map~\eqref{eq:coproduct_is_product} is surjective.
  \item Let $\varphi,\psi:B\to X\vee Y$ be morphisms in $\bHo(\cC)$
    such that $p_X\varphi=p_X\psi$ and $p_Y\varphi=p_Y\psi$.
    Then $\Sigma \varphi=\Sigma\psi$.
  \item If $\cC$ is stable, then the homotopy category $\bHo(\cC)$
    is additive and for every object $A$ of $\cC$ the
    morphism $m_A:\Sigma A\to\Sigma A$ is the negative of
    the identity of $\Sigma A$.
\end{enumerate}
\end{prop}
\begin{proof}
(i) Given two morphisms $\alpha:\Sigma A\to X$ and $\beta:\Sigma A\to Y$
in $\bHo(C)$ we consider the morphism
$((\alpha\circ m_A)\vee\beta)\circ \kappa_A:\Sigma A\to X\vee Y$.
This morphism then satisfies
\[
p_X\circ((\alpha\circ m_A)\vee \beta)\circ \kappa_A\ =\
  \alpha\circ m_A\circ (\Id+0)\circ \kappa_A \ = \
  \alpha\circ m_A^2 \ = \ \alpha \]
and similarly $p_Y\circ((\alpha\circ m_A)\vee \beta)\circ \kappa_A=\beta$.
So the map~\eqref{eq:coproduct_is_product} is surjective for $B=\Sigma A$.

(ii) We first show that the composite
\begin{equation}\label{eq:Id_X_vee_Y}
 \Sigma(X\vee Y) \ \xra{\ \kappa_{X\vee Y}\ } 
\ \Sigma(X\vee Y)\vee \Sigma(X\vee Y) \
\xra{(\Sigma \iota_X)m_X(\Sigma p_X)+(\Sigma \iota_Y)(\Sigma p_Y)}
\Sigma(X\vee Y)
\end{equation}
is the identity of $\Sigma(X\vee Y)$. Here $\iota_X:X\to X\vee Y$
and $\iota_Y:Y\to X\vee Y$ are the canonical morphisms.
Indeed, after precomposition with
$\Sigma \iota_X:\Sigma X\to\Sigma(X\vee Y)$ we have
\begin{align*}
((\Sigma \iota_X)m_X(\Sigma p_X)&+(\Sigma \iota_Y)(\Sigma p_Y))\circ\kappa_{X\vee Y}
\circ(\Sigma \iota_X) \\ 
&= \ 
((\Sigma \iota_X)m_X(\Sigma p_X)+(\Sigma \iota_Y)(\Sigma p_Y))\circ
(\Sigma \iota_X\vee\Sigma \iota_X)\circ\kappa_X \\ 
&= \ ((\Sigma \iota_X)m_X+ 0)\circ\kappa_X \
= \ ((\Sigma \iota_X)m_X)\circ(\Id_{\Sigma X}+0)\circ\kappa_X \\
&= \ (\Sigma \iota_X)\circ m_X^2 \ = \ \Sigma \iota_X \ .
\end{align*}
Similarly, we have 
$((\Sigma \iota_X)m_X(\Sigma p_X)+(\Sigma \iota_Y)(\Sigma p_Y))\circ\kappa_{X\vee Y}
\circ(\Sigma \iota_Y)=\Sigma \iota_Y$.
Since the suspension functor preserves coproducts, a morphism out of
$\Sigma(X\vee Y)$ is determined by precomposition with 
$\Sigma \iota_X$ and $\Sigma \iota_Y$. 
This proves that the composite~\eqref{eq:Id_X_vee_Y} is the identity.
For $\varphi:B\to X\vee Y$ we then have
\begin{align*}
  \Sigma \varphi \ &= \ 
((\Sigma \iota_X)m_X(\Sigma p_X)+(\Sigma \iota_Y)(\Sigma p_Y))\circ\kappa_{X\vee Y}\circ(\Sigma \varphi)\\
&= \ ((\Sigma \iota_X)m_X(\Sigma p_X)+(\Sigma \iota_Y)(\Sigma p_Y))
\circ(\Sigma\varphi\vee\Sigma\varphi)\circ\kappa_B\\
&= \ ((\Sigma \iota_X)m_X\Sigma(p_X\varphi)+(\Sigma \iota_Y)\Sigma(p_Y\varphi))\circ\kappa_B\ .
\end{align*}
So $\Sigma\varphi$ is determined by
the composites $p_X\varphi$ and $p_Y\varphi$, and this proves the claim.

(iii) Since $\cC$ is stable, every object is isomorphic to
a suspension, so the map~\eqref{eq:coproduct_is_product}
is always surjective by part~(i). 
Moreover, suspension is faithful,
so the map~\eqref{eq:coproduct_is_product}
is always injective by part~(ii). 
Thus the map~\eqref{eq:coproduct_is_product}
is bijective for all objects $B, X$ and $Y$, and so coproducts in $\bHo(\cC)$
are also products.

It is well known  that in any category with zero object
that has coproducts that are also products, the morphisms sets in $\cT$ 
can then be endowed with a natural structure of abelian monoid
as follows, see for example~\cite[Thm.\,8.2.14]{kashiwara-schapira}. 
Given $f,g:B\to Z$, let $f\perp g:B\to Z\vee Z$ be the
unique morphism such that $(\Id+0)(f\perp g)=f$ and
$(0+\Id)(f\perp g)=g$. Then the assignment
$f+g= (\Id+\Id)(f\perp g)$ is an associative, commutative and
binatural operation on the
set of morphisms from $B$ to $Z$ 
with neutral element given by the zero morphism.

The collapse map $\kappa_A:\Sigma A\to \Sigma A\vee \Sigma A$ 
satisfies $(\Id+0)\kappa=m_A$ and
$(0+\Id)\kappa_A=\Id$, and so $\kappa_A=m_A\perp\Id$.
So we have $m_A+\Id=(\Id+\Id)\kappa_A=0$.
This shows that the morphism $m_A$ is the additive inverse of 
the identity of $\Sigma A$. In particular, 
the abelian monoid $\bHo(\cC)(\Sigma A, Z)$
has inverses, and is thus an abelian group. Since every object is isomorphic
to a suspension, the abelian monoid $\bHo(\cC)(B,Z)$ is a group for all
objects $B$ and $Z$, and so $\bHo(\cC)$ is an additive category.
\end{proof}

Now we introduce the class of distinguished triangles.
Given a cofibration $j:A\to B$ in a pointed cofibration category $\cC$, 
we define the {\em connecting morphism}
$\delta(j):B/A\to\Sigma A$ in $\bHo(\cC)$ as 
\begin{equation}\label{eq:connecting_morphism}
 \delta(j) \ = \
  \gamma(p\cup 0)\circ\gamma(0\cup q)^{-1} \ : B/A \ \to \ \Sigma A\ . 
\end{equation}
Here $q:B\to B/A$ is the quotient morphism,
$p\cup 0:CA\cup_jB\to\Sigma A$ is the morphism that collapses $B$ 
and $0\cup q:CA\cup_jB\to B/A$ is the weak equivalence that collapses $CA$.
The {\em elementary distinguished triangle} associated to the 
cofibration $j$ is the sequence
$$ A \ \xra{\ \gamma(j)\ }\ B \ \xra{\ \gamma(q)\ }\ B/A \ \xra{\ \delta(j) \ }\ \Sigma A \ .$$
A {\em distinguished triangle} is any triangle in the homotopy category
that is isomorphic to the elementary distinguished triangle 
of a cofibration in $\cC$.

\begin{prop}\label{prop-connecting natural}
The connecting morphism~\eqref{eq:connecting_morphism} 
is natural in the following sense:
for every commutative square in $\cC$ on the left
$$\xymatrix{ A \ar[r]^-j \ar[d]_\alpha & B\ar[d]^\beta 
&&& B/A \ar[r]^-{\delta(j)} \ar[d]_{\gamma(\beta/\alpha)} & 
\Sigma A\ar[d]^{\Sigma\gamma(\alpha)}\\
A'\ar[r]_-{j'} & B' &&& B'/A'\ar[r]_-{\delta(j')} &\Sigma A'}$$
such that $j$ and $j'$ are cofibrations, the square
on the right commutes in $\bHo(\cC)$.
\end{prop}
\begin{proof}
We choose a cone extension $(\bar\alpha,s)$ of the morphism $\alpha$
as in the commutative diagram on the left:
$$\xymatrix@R=6mm{ 
A \ar[d]_\alpha\ar[r]^-{i_A} & CA \ar[d]^-{\bar\alpha}  &&
B/A\ar[d]_{\beta/\alpha} & 
CA\cup_jB \ar[r]^-{p\cup 0}\ar[l]_-{0\cup q}^\sim\ar[d] & 
\Sigma A\ar[d]^{\bar\alpha/\alpha}\\
A' \ar[r]^-{s i_{A'}} &  \bar C && 
B'/A'\ar@{=}[d] & 
\bar C\cup_{j'}B' \ar[l]_{0\cup q'}^\sim\ar[r]^-{\bar p'\cup 0} & 
\bar C/A'\\
A'\ar@{=}[u] \ar[r]_-{i_{A'}} &CA'\ar[u]_s^{\sim}&&
B'/A' & CA'\cup_{j'}B' \ar[l]^-{0\cup q'}\ar[r]_-{p'\cup 0}
\ar[u]_{s\cup B'}^\sim & 
\Sigma A'\ar[u]_{s/A'}^\sim }$$
From this we form the commutative diagram on the right.
After passage to $\bHo(\cC)$ we can invert the weak equivalences
that point to the left or upwards, and then the composite
through the lower left corner becomes $\delta(j')\circ\gamma(\beta/\alpha)$,
the composite through the upper right corner becomes 
$\Sigma\gamma(\alpha)\circ\delta(j)$.
\end{proof}

Now we can prove that the homotopy category
of a stable category is triangulated.
For us, a triangulated category is an additive category $\cT$
equipped with an auto-equivalence $\Sigma:\cT\to \cT$
and a class of distinguished triangles of the form
$$ A \ \xra{\ f \ }\  B\  \xra{\ g\ }\  C\ \xra{\ h\ }\ \Sigma A\ , $$
closed under isomorphism, that satisfies the following axioms:

\begin{itemize}
\item[\bf (T1)] For every object $X$ 
the triangle $0\to X\xra{\Id}X\to 0$ is distinguished.
\item[\bf (T2)] [Rotation]
If a triangle $(f,g,h)$ is distinguished, then so is the triangle
$(g,h,-\Sigma f)$.
\item[\bf (T3)] [Completion of triangles] 
Given distinguished triangles $(f,g,h)$ and  $(f',g',h')$ 
and morphisms $a$ and $b$ satisfying $b f=f'a$, 
there exists a morphism $c$ making the following diagram commute:
\[ \xymatrix{ 
A \ar[r]^f \ar[d]_a &  B \ar[r]^g \ar[d]_b & 
C \ar@{..>}[d]^c \ar[r]^-h & \Sigma A \ar[d]^{\Sigma a}\\
A' \ar[r]_-{f'} & B' \ar[r]_-{g'} & C' \ar[r]_-{h'}&  \Sigma A'}    \]
\item[\bf (T4)] [Octahedral axiom]
For every pair of composable morphisms $f:A\to B$ and $f':B\to D$
there is a commutative diagram
\[
\xymatrix{ A \ar@{=}[d] \ar[r]^-f &  
B \ar[d]_{f'} \ar[r]^-g & C \ar[d]^{x} \ar[r]^-h & 
\Sigma A \ar@{=}[d] \\
A \ar[r]_-{f'f} & D \ar[r]_-{g''} \ar[d]_{g'} & 
E \ar[r]_-{h''} \ar[d]^y & \Sigma A \ar[d]^{\Sigma f}\\
& F \ar[d]_{h'} \ar@{=}[r] & F \ar[d]^(.6){(\Sigma g)\circ h'}  
\ar[r]_-{h'} & \Sigma B\\
& \Sigma B \ar[r]_-{\Sigma g} & \Sigma C &
}\]
such that the triangles $(f,g,h)$,  $(f',g',h')$,  $(f'f,g'',h'')$
and $(x,y,(\Sigma g)\circ h')$ are distinguished.
\end{itemize}

This version of the axioms appears to be weaker, at first sight,
than the original formulation: Verdier~\cite[II.1]{verdier} requires an
`if and only if' in (T2), and his formulation of (T4) appears stronger. 
However, the two formulations of the axioms are equivalent.

\begin{theorem}\label{thm-HoC is triangulated}
The suspension functor and the class of distinguished triangles make
the homotopy category $\bHo(\cC)$ of a stable cofibration category
into a triangulated category.
\end{theorem}

This result generalizes and provides a uniform approach to proofs 
of triangulations for homotopy categories of 
pretriangulated dg categories,
additive categories, derived categories of abelian categories,
stable categories of Frobenius categories
and for homotopy categories of stable model categories.
Indeed, all these classes of categories have underlying stable
cofibration categories. The only triangulated categories 
I am aware of that cannot be established 
via Theorem~\ref{thm-HoC is triangulated} 
are the exotic examples of~\cite{nomodel}.

\begin{rk}
Although I am not aware of a complete proof of 
Theorem~\ref{thm-HoC is triangulated}
in the present generality, I do not claim much originality;
indeed, various parts of this result, assuming  the same or closely 
related structure, are scattered throughout the literature.
Nevertheless, I hope that a self-contained and complete account
on the triangulation of the homotopy category,
assuming only the axioms of a stable cofibration category, is useful.

Before embarking on the proof of Theorem~\ref{thm-HoC is triangulated}, 
I want to point out the relevant related sources that I am aware of.
In Section I.2 of his monograph~\cite{Q} Quillen shows that in any 
pointed closed model category, the homotopy category supports
a functorial suspension functor with values in cogroup objects.
In Section I.3 Quillen introduces the class of cofibration sequences 
and shows that they satisfy most of (the unstable version of) the
axioms of a triangulated category.
In~\cite[Sec.\,I.4]{brown} Brown adapts Quillen's arguments 
to the more general context of `categories of fibrant objects', 
which is strictly dual to the cofibration categories that we use.
In~\cite[Sec.\,3]{heller-stable} Heller introduces the notion of
`h-c-category', an axiomatization closely related to 
(but slightly different from) that of cofibration categories; 
Heller indicates in his Theorem 9.2 how the stabilization 
of the homotopy category of an h-c-category is triangulated. 
We want to stress, though, that none of these three sources
bothers to prove the octahedral axiom.

In the context of stable model categories, a complete proof 
of the triangulation of the homotopy category, including the octahedral axiom,
is given by Hovey in~\cite[Sec.\,7.1]{hovey-book}.
However, Hovey's account is hiding the fact that the triangulation
is available under much weaker hypotheses; for example, the
fibrations, the existence of general colimits and functorial factorization 
are irrelevant for this particular purpose.

Yet another approach to triangulating the homotopy category of 
a cofibration category $\cC$ uses the whole system
of homotopy categories of suitable diagrams in~$\cC$.
This idea has been made precise in slightly different forms by different people,
for example by Grothendieck as a 
{\em derivator}~\cite{maltsiniotis-derivateurs, groth-stable derivators},
by Heller as a {\em homotopy theory}~\cite{heller-homotopy theories},
by Keller as an 
{\em epivalent tower of suspended categories}~\cite{keller-universal}
and by Franke as a 
{\em system of triangulated diagram categories}~\cite{franke-adams}.
The respective `stable' versions come with theorems showing
that the underlying category of such a stable collection of
homotopy categories is naturally triangulated, compare 
\cite[Thm.\,4.18]{groth-stable derivators} or~\cite[Sec.\,1.4, Thm.\,1]{franke-adams}.

Cisinski shows in~\cite[Cor.\,2.24]{cisinski-categories derivables} 
that the homotopy category of every cofibration category 
is underlying a right derivator 
(parametrized by finite, direct indexing categories).
I am convinced that our definition of `stable'
(i.e., that the suspension functor is an auto-equivalence of 
the homotopy category) 
implies that the associated derivator is `triangulated', i.e.,
that the right derivator is automatically also a left derivator,
hence a `derivator' (without any adjective),
and that homotopy cartesian square and homotopy co-cartesian
square coincide. However, as far as I know this link is not 
established anywhere in the literature.
\end{rk}

\begin{proof}[Proof of Theorem~\ref{thm-HoC is triangulated}]
We have seen in Proposition~\ref{prop-stable is additive}~(iii)
that the homotopy category of a stable cofibration category is
additive, and the suspension functor is an equivalence by assumption.
So it remains to verify the axioms (T1) through (T4).
We want to emphasize that the following proof 
of (T1) -- (T4) works in any pointed cofibration category, 
without a stability assumption.
The only place where stability is used is at the very end of the
rotation axiom (T2), where the morphism $(\Sigma f)\circ\delta(i_A)$
is identified with $-\Sigma f$.\medskip

{\bf (T1)} 
The unique morphism from any zero object to $X$ is a cofibration with quotient 
morphism the identity of $X$. The triangle $(0,\Id_X,0)$
is the associated elementary distinguished triangle.\smallskip

{\bf (T2 --  Rotation)} 
We start with a distinguished triangle $(f,g,h)$ and
want to show that the triangle $(g,h,-\Sigma f)$ is also distinguished.
It suffices to consider the elementary distinguished triangle 
$(\gamma(j),\gamma(q),\delta(j))$ associated to a cofibration $j:A\to B$. 
We choose pushouts for the left and the outer square as in the left diagram:
$$\xymatrix@C=12mm{ 
A \ar[r]^-j \ar[d]_{i_A} &  B \ar[d]_k \ar[r] &  \ast \ar[d] & 
B \ar[r]^-{\gamma(k)} \ar@{=}[d] &  CA\cup_jB \ar[r]^-{\gamma(p\cup 0)}
\ar[d]_{\gamma(0\cup q)}^\iso  & 
\Sigma A \ar@{=}[d] \ar[r]^-{\delta(k)} & 
\Sigma B \ar@{=}[d]\\
CA \ar[r] & CA\cup_jB\ar[r]_-{p\cup 0} & \Sigma A&
B \ar[r]_-{\gamma(q)} & B/A \ar[r]_-{\delta(j)} & 
\Sigma A \ar[r]_-{\Sigma\gamma(j)\circ\delta(i_A)}&  
\Sigma B }$$
The second square in the left diagram is then also a pushout and the morphism
$p\cup 0:CA\cup_jB\to\Sigma A$ is the quotient projection 
associated to the cofibration $k:B\to CA\cup_jB$.
Moreover, both~$i_A$ and $k$ are cofibrations,
so by naturality of the connecting morphisms we get
$\delta(k)\circ\Id_{\Sigma A}=(\Sigma\gamma(j))\circ\delta(i_A)$.
Hence the diagram on the right commutes.
The upper row is the elementary distinguished triangle 
of the cofibration $k$, and all vertical maps are isomorphisms,
so the lower triangle is distinguished, as claimed.
By definition the connecting morphism $\delta(i_A)$
coincides with the involution $m_A$ of $\Sigma A$.
In the stable context, $m_A$ is the negative of the identity
(see Proposition~\ref{prop-stable is additive}~(iii)),
so $(\Sigma f)\circ\delta(i_A)=-\Sigma f$.\smallskip

{\bf (T3 -- Completion of triangles)} 
We are given two distinguished triangles
$(f,g,h)$ and $(f',g',h')$ and two morphisms $a$ and $b$
in $\bHo(\cC)$ satisfying $b f=f'a$ as in the diagram:
$$ \xymatrix{ A \ar[r]^-f\ar[d]_a & B\ar[d]^b\ar[r]^-g & 
C \ar[r]^-h \ar@{-->}[d]^c&\Sigma A \ar[d]^-{\Sigma a}\\
A'\ar[r]_-{f'} & B' \ar[r]_-{g'} & C'\ar[r]_-{h'} &\Sigma A'}$$
We have to extend this data to a morphism of triangles,
i.e., find a morphism $c$ making the entire diagram commute.
If we can solve the problem for isomorphic triangles, then we
can also solve it for the original triangles. 
We can thus assume that the triangles $(f,g,h)$ and $(f',g',h')$ 
are the elementary distinguished triangles arising from two cofibrations
$j:A\to B$ and $j':A'\to B'$.

We start with the special case where $a=\gamma(\alpha)$
and $b=\gamma(\beta)$ for $\cC$-morphisms $\alpha:A\to A'$
and $\beta:B\to B'$. Then $\gamma(j'\alpha)=\gamma(\beta j)$,
so Theorem~\ref{thm-HoC a la Brown}~(ii) 
provides an acyclic cofibration $s:B'\to \bar B$,
a cylinder object $(I,i_0,i_1,p)$ for $A$ and a homotopy $H:I\to \bar B$
from $Hi_0=sj'\alpha$ to $Hi_1=s\beta j$.
The following diagram of cofibrations on the left commutes in $\cC$, 
so the diagram of elementary distinguished triangles
on the right commutes in $\bHo(\cC)$ by the naturality 
of the connecting morphisms:
$$ \xymatrix@C=15mm@R=7mm{
A \ar[r]^-j\ar@{=}[d] &  B &
A \ar[r]^-{\gamma(j)}\ar@{=}[d]
&  B \ar[r]^-{\gamma(q)} 
& B/A \ar[r]^-{\delta(j)} & 
\Sigma A \ar@{=}[d]
\\
A \ar[r]^-{i_0}\ar[d]_\alpha
&  I\cup_{i_1}B \ar[d]^{H\cup s\beta} \ar[u]^\sim_{jp\cup B}&
A \ar[r]^-{\gamma(i_0)}\ar[d]_{\gamma(\alpha)}
&  I\cup_{i_1}B\ar[d]_{\gamma(H\cup s\beta)} 
\ar[u]^{\gamma(jp\cup B)}_\iso \ar[r]^-{\gamma(q)} 
& (I\cup_{i_1}B)/A\ar[d]^{\gamma((H\cup s\beta)/\alpha)} 
\ar[u]_{\gamma((jp\cup B)/A)}^\iso \ar[r]^-{\delta(i_0)} & 
\Sigma A \ar[d]^{\Sigma\gamma(\alpha)} 
\\
A'\ar[r]^-{sj'} & \bar B  &
A'\ar[r]^-(.4){\gamma(sj')} & \bar B \ar[r]^-{\gamma(\bar q)} & 
\bar B/A' \ar[r]^-(.6){\delta(sj')} & \Sigma A'\\
A'\ar[r]_-{j'}\ar@{=}[u] & 
B'\ar[u]_-s^-\sim &
A'\ar[r]_-{\gamma(j')}\ar@{=}[u] & 
B'\ar[u]^-{\gamma(s)}_\iso \ar[r]_-{\gamma(q')}& 
B'/A' \ar[u]_{\gamma(s/A')}^\iso \ar[r]_-{\delta(j')}& \Sigma A'\ar@{=}[u] }$$
The morphism 
$$ c\ =\ 
\gamma(s/A')^{-1}\circ\gamma((H\cup s\beta)/\alpha)\circ\gamma((jp\cup B)/A)^{-1}\ :\ B/A\to B'/A'$$
is the desired filler.

In the general case we write $a=\gamma(s)^{-1}\gamma(\alpha)$ 
where $\alpha:A\to \bar A$ is a $\cC$-morphism and $s:A'\to \bar A$
is an acyclic cofibration. We choose a pushout 
$$\xymatrix{ \bar A \ar[r]^-k & \bar A\cup_{A'}B' \\
A' \ar[u]^-s_-\sim \ar[r]_-{j'} & B' \ar[u]_-{s'}^-\sim}$$
We write $\gamma(s')b=\gamma(t)^{-1}\gamma(\beta):B\to\bar A\cup_{A'}B'$ 
where $\beta:B\to\bar B$ is a $\cC$-morphism and $t:\bar A\cup_{A'}B'\to\bar B$
is an acyclic cofibration.
We then have 
$$ \gamma(tk)\gamma(\alpha)\ =\ 
\gamma(tk)\gamma(s)a\ =\ 
\gamma(ts')\gamma(j')a\ =\ 
\gamma(ts')b\gamma(j)
\ =\ \gamma(\beta)\gamma(j)\ , $$ 
so by the special case, applied to the cofibrations
$j:A\to B$ and $tk:\bar A\to\bar B$ and the morphisms
$\alpha:A\to\bar A$ and $\beta:B\to\bar B$,
there exists a morphism $c:B/A\to \bar B/\bar A$ in $\bHo(\cC)$ 
making the diagram
$$\xymatrix@C=15mm@R=7mm{ 
A \ar[r]^-{\gamma(j)} \ar[d]_{\gamma(\alpha)} &  
B \ar[r]^-{\gamma(q)} \ar[d]_{\gamma(\beta)} & 
B/A \ar@{..>}[d]_c \ar[r]^-{\delta(j)} & 
\Sigma A \ar[d]^{\Sigma \gamma(\alpha)}\\
\bar A \ar[r]^-(.4){\gamma(tk)} & 
\bar B \ar[r]^-{\gamma(\bar q)} & 
\bar B/\bar A \ar[r]^-{\delta(tk)}&  \Sigma \bar A \\
A' \ar[r]_-{\gamma(j')}\ar[u]^{\gamma(s)} & 
B' \ar[r]_-{\gamma(q')} \ar[u]^{\gamma(ts')}& 
B'/A' \ar[r]_-{\delta(j')}\ar[u]_{\gamma(ts'/s)}&  
\Sigma A'\ar[u]_{\Sigma\gamma(s)} }$$
commute (the lower part commutes by naturality of connecting morphisms). 
Since $s$ is an acyclic cofibration so is its cobase change $s'$.
By the gluing lemma the weak equivalences 
$s:A'\to\bar A$ and $ts':B'\to \bar B$ induce a weak equivalence
$ts'/s:B'/A'\to \bar B/\bar A$ on quotients and the composite
$$ B/A \ \xra{\quad c \quad} \ \bar B/\bar A \ 
\xra{\gamma(ts'/s)^{-1}} \ B'/A'$$
in $\bHo(\cC)$ thus solves the original problem.\smallskip

{\bf (T4 - Octahedral axiom)}
We start with the special case where $f=\gamma(j)$ and
$f'=\gamma(j')$ for cofibrations $j:A\to B$ and $j':B\to D$.
Then the composite $j'j:A\to D$ is a cofibration
with $\gamma(j'j)=f'f$. The diagram 
$$\xymatrix@C=12mm@R=10mm{ A \ar@{=}[d] \ar[r]^-{\gamma(j)} & 
B \ar[d]_{\gamma(j')} \ar[r]^-{\gamma(q_j)} & 
B/A \ar[d]^{\gamma(j'/A)} 
\ar[r]^-{\delta(j)} & \Sigma A \ar@{=}[d] \\
A \ar[r]_-{\gamma(j'j)} & D \ar[r]_-{\gamma(q_{j'j})} 
\ar[d]_{\gamma(q_{j'})} & 
D/A \ar[r]_-{\delta(j'j)} \ar[d]^(.6){\gamma(D/j)}
& \Sigma A \ar[d]^{\Sigma\gamma(j)}\\
& D/B \ar[d]_{\delta(j')} \ar@{=}[r] & D/B
\ar[d]^(.6){\delta(j'/A)=(\Sigma \gamma(q_j))\delta(j')}  
\ar[r]_-{\delta(j')} & \Sigma B\\
& \Sigma B \ar[r]_-{\Sigma\gamma(q_j)} & \Sigma (B/A) &
}$$ 
then commutes by naturality of connecting morphisms.
Moreover, the four triangles in question are the elementary
distinguished triangles of the cofibrations
$j$, $j'$, $j'j$ and $j'/A:B/A\to D/A$.

In the general case we write $f=\gamma(s)^{-1}\gamma(a)$
for a $\cC$-morphism $a:A\to B'$ and a weak equivalence $s:B\to B'$. 
Then $a$ can be factored as
$a=pj$ for a cofibration $j:A\to \bar B$ and a weak equivalence
$p:\bar B\to B'$. Altogether we then have $f=\varphi\circ\gamma(j)$
where $\varphi=\gamma(s)^{-1}\circ\gamma(p):\bar B\to B$ is an
isomorphism in $\bHo(\cC)$.
We can apply the same reasoning to the morphism $f'\varphi:\bar B\to D$
and write it as $f'\circ\varphi=\psi\circ\gamma(j')$ for
a cofibration $j':\bar B\to \bar D$ in~$\cC$ and an isomorphism
$\psi:\bar D\to D$ in $\bHo(\cC)$.
The special case can then be applied to the cofibrations
$j:A\to \bar B$ and $j':\bar B\to \bar D$.
The resulting commutative diagram that solves (T4)
for $(\gamma(j),\gamma(j'))$ can then be translated back into
a commutative diagram that solves (T4) for $(f,f')$ by conjugating with
the isomorphisms  $\varphi:\bar B\to B$ and $\psi:\bar D\to D$.
This completes the proof of the octahedral axiom (T4),
and hence the proof of Theorem~\ref{thm-HoC is triangulated}.
\end{proof}

Now we discuss how exact functors between stable cofibration
categories give rise to exact functors between the triangulated homotopy
categories.
A functor $F:\cC\to\cD$ between cofibration categories is {\em exact} 
if it preserves initial objects, cofibrations, weak equivalences and
the particular pushouts~\eqref{eq-C3pushout} along cofibrations
that are guaranteed by axiom (C3). 
Since $F$ preserves weak equivalences, the composite functor
$\gamma^{\cD}\circ F:\cC\to\bHo(\cD)$ takes weak equivalences to isomorphisms
and the universal property of the homotopy category provides a unique
{\em derived functor} $\bHo(F) :\bHo(\cC)\to\bHo(\cD)$
such that $\bHo(F)\circ\gamma^\cC=\gamma^\cD\circ F$.
An exact functor between cofibration categories in particular
preserves coproducts. Since coproducts in $\cC$ descend to
coproducts in the homotopy category, the derived functor
of any exact functor also preserves coproducts.

We will now explain that for pointed cofibration categories 
$\cC$ and $\cD$ the derived functor  $\bHo(F)$ 
commutes with suspension up to a preferred natural isomorphism
\begin{equation}\label{eq:tau^F}
  \tau_F \ :\ \bHo(F)\circ\Sigma \ \xra{\ \iso \ } \  \Sigma\circ\bHo(F)
\end{equation}
of functors from $\bHo(\cC)$ to $\bHo(\cD)$.
If $A$ is any object of $\cC$, then the cofibration $F(i_A):F(A)\to F(CA)$ 
is a cone since $F$ is exact.
Lemma~\ref{lemma-cone extensions} provides a cone extension
of the identity of $F(A)$, i.e., a morphism $\bar\alpha:F(CA)\to\bar C$,
necessarily a weak equivalence, and an acyclic cofibration
$s:C(F(A))\to\bar C$ such that $si_{F(A)}=\bar\alpha F(i_A)$. 
The composite in $\bHo(\cD)$
$$\tau_{F,A}\ : \  
F(\Sigma A) = F(CA)/F(A) \ \xra{\ \gamma(\bar\alpha/F(A))\ }\
\bar C/F(A) \ \xra{\ \gamma(s/F(A))^{-1}\ } \ \Sigma F(A) $$
is then an isomorphism, and independent 
(by Lemma~\ref{lemma-cone extensions}) of the cone extension $(\bar\alpha,s)$.

\begin{prop}\label{prop-left derived is exact} 
Let $F:\cC\to\cD$ be an exact functor between pointed cofibration
categories. 
Then the isomorphism $\tau_{F,A}:F(\Sigma A)\to \Sigma(FA)$
is natural in $A$ and makes the derived functor $\bHo(F):\bHo(\cC)\to\bHo(\cD)$ 
into an exact functor.
\end{prop}
\begin{proof} 
Let $j:A\to B$ be a cofibration in $\cC$, with $q:B\to B/A$ a quotient morphism.
Since~$F$ is exact, $F(j)$ is a cofibration in $\cD$ and
$F(q):F(B)\to F(B/A)$ is a quotient morphism for~$F(j)$.
We claim that the connecting morphism $\delta_{F(j)}:F(B/A)\to \Sigma F(A)$
of the cofibration~$F(j)$ equals $\tau_{F,A}\circ \bHo(F)(\delta(j))$.
To see this we choose a cone extension $(\bar\alpha,s)$ 
of the identity of~$F(A)$ as in the construction of $\tau_{F,A}$.
We can build the commutative diagram
$$\xymatrix@C=15mm{ 
F(B/A)\ar@{=}[d] \ar@{-->}@<3ex>@/^1pc/[rrr]^(.6){\bHo(F)(\delta(j))}& 
F(CA)\cup_{F(j)}F(B) \ar[r]^-{F(p)\cup 0}\ar[l]_-{0\cup F(q)}^\sim
\ar[d]_\sim^{\bar\alpha\cup\Id} & 
F(CA)/F(A) \ar@{=}[r]\ar[d]^{\bar\alpha/F(A)}_\sim  &
F(\Sigma A)
\ar@{-->}[dd]^{\tau_{F,A}}\\
F(B/A) \ar@{=}[d] & 
\bar C\cup_{F(j)}F(B) \ar[l]_{0\cup F(q)}^\sim\ar[r]^-{\bar p\cup 0} & 
\bar C/F(A) &\\
F(B/A) \ar@{-->}@<-3ex>@/_1pc/[rrr]_(.6){\delta(F(j))} &
C(F(A))\cup_{F(j)}F(B) \ar[l]^-{0\cup F(q)}_-\sim\ar[r]_-{p\cup 0}
\ar[u]_{s\cup\Id}^\sim & 
C(F(A))/F(A) \ar@{=}[r]\ar[u]_{s/F(A)}^\sim  & \Sigma F(A)}$$
After passing to the homotopy category of $\cD$ we can invert the
weak equivalences and form the dashed morphisms.
Since the exact functor $F$ preserves the pushout that defines
$CA\cup_jB$, the composite of the top row then becomes $\bHo(F)(\delta(j))$. 
Since the vertical composite on
the right is the isomorphism $\tau_{F,A}$, this proves the
relation $\tau_{F,A}\circ \bHo(F)(\delta(j))=\delta(F(j))$.

Now we can prove the proposition.
We have to show that for every distinguished triangle $(f,g,h)$
in $\bHo(\cC)$ the triangle
$$ F(A)\ \xra{\ \bHo(F)(f)\ }\ F(B) \xra{\ \bHo(F)(g)\ }\
F(C)\ \xra{\tau_{F,A}\circ \bHo(F)(h)\ } \ \Sigma F(A) $$
is distinguished in $\bHo(\cD)$. If suffices to consider 
the elementary distinguished triangle of a cofibration $j:A\to B$ in $\cC$;
because $\bHo(F)\circ\gamma=\gamma\circ F$, we are then dealing 
with the triangle
$$ F(A) \xra{\gamma(F(j))}
F(B) \xra{\gamma(F(q))}
F(B/A) \xra{\tau_{F,A}\circ \bHo(F)(\delta(j))}
\Sigma F(A) \ . $$
By the claim above, this triangle is the elementary distinguished triangle 
of the cofibration $F(j)$, and this concludes the proof.
\end{proof}

\begin{rk}\label{rk-uniqueness of triangulation} 
Let us consider two different
choices of cones $\{i_A:A\to CA\}_A$ respectively
$\{i'_A:A\to C'\!A\}_A$ on a pointed cofibration category $\cC$, where $A$
runs through all objects of $\cC$. These different cones
give rise to different suspension functors $\Sigma$ and $\Sigma'$
and different collections $\Delta$ and $\Delta'$ of distinguished
triangles. We can apply the previous proposition with $\cC=\cD$ and
$F=\Id$, the identity functor of $\cC$, to the two triangulations.
The proposition provides a natural isomorphism $\tau:\Sigma\to\Sigma'$
between the two suspension functors such that for every
distinguished $\Delta$-triangle $(f,g,h)$ the triangle
$$ A \ \xra{\ f\ } \  B \ \xra{\ g\ } \  C \ \xra{\ \tau_A\circ h\ } \
\Sigma' A$$ is
a distinguished $\Delta'$-triangle.
So the triangulations arising from different choices of cones
are canonically isomorphic.
\end{rk}

\end{appendix}


\begin{thebibliography}{EKMM}

\bibitem[An]{angeltveit}
V.\,Angeltveit, {\em Topological Hochschild homology and cohomology 
of $A\sb \infty$ ring spectra.}  
Geom. Topol. {\bf 12}  (2008),  no.\,2, 987--1032.

\bibitem[Ba]{baues-algebraic homotopy}
H.-J.\,Baues, {\em Algebraic homotopy}. 
Cambridge Studies in Advanced Mathematics, 15. Cambridge University Press, 
Cambridge, 1989. xx+466 pp.

\bibitem[BK]{bondal-kapranov}
A.\,I.\,Bondal, M.\,M.\,Kapranov, {\em Framed triangulated categories.} 
(Russian)  Mat.\,Sb.\,181  (1990),  no.\,5, 669--683;  
translation in  Math. USSR-Sb.\,70  (1991),  no.\,1, 93--107.

\bibitem[BF]{BF}
A.\,K.\,Bousfield, E.\,M.\,Friedlander, {\em Homotopy theory of 
{$\Gamma$}-spaces, spectra, and bisimplicial sets.} Geometric applications of 
homotopy theory (Proc.\,Conf., Evanston, Ill., 1977), II.  
Lecture Notes in Math.\,{\bf 658}, Springer, 1978, 80--130.

\bibitem[Bre]{Bredon-SLN}
G.\,E.\,Bredon, {\em Equivariant cohomology theories.}
Lecture Notes in Math.\,{\bf 34}, Springer, 1967, vi+64 pp. 

\bibitem[Bro]{brown}
K.\,S.\,Brown, {\em Abstract homotopy theory and generalized sheaf cohomology.}
Trans. Amer. Math. Soc. {\bf 186} (1974), 419--458. 

\bibitem[BRS]{buonchristiano-rourke-sanderson}
S.\,Buonchristiano, C.\,P.\,Rourke, B.\,J.\,Sanderson, 
{\em A geometric approach to homology theory.}
London Mathematical Society Lecture Notes 18,
Cambridge Univ.\,Press, Cambridge, 1976.

\bibitem[Ci]{cisinski-categories derivables}
D.-C.\,Cisinski, {\em Cat{\'e}gories d{\'e}rivables.} 
Bull.\,Soc.\,Math.\,France {\bf 138} (2010), no.\,3, 317--393. 

\bibitem[DK]{DK-function}
W.\,G.\,Dwyer, D.\,M.\,Kan, {\em Function complexes in homotopical algebra.}
Topology {\bf 19} (1980), 427--440.

\bibitem[EKMM]{EKMM}
A.\,D.\,Elmendorf, I.\,Kriz, M.\,A.\,Mandell, J.\,P.\,May,
{\em Rings, modules, and algebras in stable homotopy theory.
{W}ith an appendix by M.\,Cole},
Mathematical Surveys and Monographs, {\bf 47}, Amer. Math. Soc.,
Providence, RI, 1997, xii+249 pp.

\bibitem[Fr]{franke-adams}
J.\,Franke, {\em Uniqueness theorems for certain triangulated categories
possessing an {A}dams spectral sequence},
$K$-theory Preprint Archives \#139 (1996).
{\tt http://www.math.uiuc.edu/K-theory/0139/}

\bibitem[Gr]{groth-stable derivators}
M.\,Groth, {\em Derivators, pointed derivators and stable derivators}.
Preprint (2011). \texttt{arXiv:1112.3840}

\bibitem[He68]{heller-stable}
A.\,Heller, {\em Stable homotopy categories.}
Bull. Amer. Math. Soc. {\bf 74} (1968), 28--63. 

\bibitem[He88]{heller-homotopy theories}
A.\,Heller, {\em Homotopy theories.}  
Mem. Amer. Math. Soc.  {\bf 71}  (1988),  no. 383, vi+78 pp.

\bibitem[Ho99]{hovey-book}
M.\,Hovey, {\em Model categories.} 
Mathematical Surveys and Monographs {\bf 63}, Amer.\,Math.\,Soc.,
Providence, RI, 1999, xii+209 pp.

\bibitem[HSS]{HSS}
M.\,Hovey, B.\,Shipley, J.\,Smith, {\em Symmetric spectra.}
J.\,Amer.\,Math.\,Soc.\,{\bf 13} (2000), 149--208.

\bibitem[Ka]{kan-semisimplicial}
D.\,M.\,Kan, {\em Semisimplicial spectra.}
Illinois~J.\,Math.\,{\bf 7} (1963), 463--478. 

\bibitem[KS]{kashiwara-schapira}
M.\,Kashiwara, P.\,Schapira, {\em Categories and sheaves.} 
Grundlehren der Mathematischen Wissenschaften {\bf 332}. 
Springer, 2006. x+497 pp. 

\bibitem[Ke91]{keller-universal}
B.\,Keller, {\em Derived categories and universal problems.}
Comm. Algebra {\bf 19} (1991), no.\,3, 699--747. 

\bibitem[Ke99]{keller-cyclic of exact}
B.\,Keller, {\em On the cyclic homology of exact categories.}  
J.\,Pure Appl.\,Algebra  {\bf 136}  (1999), 1--56.

\bibitem[Ke06]{keller-differential graded}
B.\,Keller, {\em On differential graded categories.}
International Congress of Mathematicians. Vol.\,II,  151--190, 
Eur.\,Math.\,Soc., Z\"urich, 2006. 

\bibitem[Kr]{krause-chicago}
H.\,Krause, {\em Derived categories, resolutions, and Brown representability}.
In: Interactions between homotopy theory and algebra,  
101--139, Contemp.\,Math. {\bf 436}, Amer.\,Math.\,Soc., Providence, RI, 2007.

\bibitem[Ma]{maltsiniotis-derivateurs}
G.\,Maltsiniotis, {\em Introduction {\`a} la th{\'e}orie des d{\'e}rivateurs 
(d'apr{\`e}s Grothendieck)}. Preprint (2001).
Available from the author's homepage

\bibitem[MMSS]{MMSS}
M.\,A.\,Mandell, J.\,P.\,May, S.\,Schwede, B.\,Shipley,
{\em Model categories of diagram spectra.}
Proc.\ London Math.\ Soc. {\bf 82} (2001), 441--512.

\bibitem[MSS]{nomodel}
F.\,Muro, S.\,Schwede, N.\,Strickland,
{\em Triangulated categories without models.}
Invent.\,Math.\,{\bf 170} (2007), 231--241.

\bibitem[Q]{Q}
D.\,Quillen, {\em Homotopical algebra.} Lecture Notes in Math.\,{\bf 43},
Springer, 1967.

\bibitem[RB]{radulescu-ABC}
A.\,Radulescu-Banu, {\em Cofibrations in homotopy theory}.
{\tt arXiv:0610.5009}

\bibitem[RS]{rourke-sanderson}
C.\,P.\,Rourke, B.\,J.\,Sanderson,
{\em $\triangle $-sets.\,I.\,Homotopy theory.}
Quart.\,J.\,Math.\,(2) {\bf 22} (1971), 321--338. 

\bibitem[Sch07]{schwede-rigid}
S.\,Schwede, {\em The stable homotopy category is rigid.} 
Annals of Math.\,{\bf 166} (2007), 837--863.

\bibitem[Sch10]{sch-leeds}
S.\,Schwede, {\em Algebraic versus topological triangulated categories.} 
In: Triangulated categories, 389--407, 
London Mathematical Society Lecture Notes 375,
Cambridge Univ. Press, Cambridge, 2010.

\bibitem[SS]{ss-modules}
S.\,Schwede, B.\,Shipley,
{\em  Stable model categories are categories of modules.}
Topology {\bf 42}  (2003), 103--153.

\bibitem[Sh]{shipley-DGAs and HZ}
B.\,Shipley, {\em $H\mZ$-algebra spectra are differential graded algebras.}
American J.\,Math.\,{\bf 129} (2007), 351--379.

\bibitem[SW]{spanier-whitehead}
E.\,H.\,Spanier, J.\,H.\,C.\,Whitehead,
{\em A first approximation to homotopy theory.}
Proc.\,Nat.\,Acad.\,Sci. U.S.A.\,{\bf 39} (1953), 655--660. 

\bibitem[To]{toda-realizing}
H.\,Toda, {\em On spectra realizing exterior parts of the Steenrod algebra.}  
Topology {\bf 10} (1971), 53--65. 

\bibitem[Ve]{verdier}
J.-L.\,Verdier, 
{\em Des cat\'egories d\'eriv\'ees des cat\'egories ab\'eliennes.}
With a preface by Luc Illusie. Edited and with a note by Georges Maltsiniotis.
Ast\'erisque,  no.\,239  (1996), xii+253 pp. 

\bibitem[Wa]{waldhausen}
F.\,Waldhausen, {\em Algebraic $K$-theory of spaces.}  
Algebraic and geometric topology (New Brunswick, N.J., 1983),  318--419, 
Lecture Notes in Math.\,{\bf 1126}, Springer, 1985. 
\end{thebibliography}
\end{document}